\documentclass[a4paper,oneside,10pt]{article}%
\usepackage{amsmath}
\usepackage{amsfonts}
\usepackage{amssymb}
\usepackage{graphicx}
\usepackage{color}
\usepackage[square,numbers,sort&compress]{natbib}%
\providecommand{\U}[1]{\protect\rule{.1in}{.1in}}

\newtheorem{theorem}{Theorem}[section]

\newtheorem{corollary}[theorem]{Corollary}

\newtheorem{assumption}[theorem]{Assumption}
\newtheorem{example}[theorem]{Example}

\newtheorem{lemma}[theorem]{Lemma}

\newtheorem{proposition}[theorem]{Proposition}
\newtheorem{remark}[theorem]{Remark}

\newenvironment{proof}[1][Proof]{\noindent \textbf{#1.} }{\  \rule{0.5em}{0.5em}}
\numberwithin{equation}{section}

\begin{document}

\title{ Linear quadratic problems for fully coupled forward-backward stochastic
control systems}
\author{Mingshang Hu\thanks{Zhongtai Securities Institute for Financial Studies,
Shandong University, Jinan, Shandong 250100, PR China. humingshang@sdu.edu.cn.
Research supported by NSF (No. 11671231) and Young Scholars Program of
Shandong University (No. 2016WLJH10). }
\and Shaolin Ji\thanks{Zhongtai Securities Institute for Financial Studies,
Shandong University, Jinan, Shandong 250100, PR China. jsl@sdu.edu.cn
(Corresponding author). Research supported by NSF No. 11571203.}
\and Xiaole Xue\thanks{Zhongtai Securities Institute for Financial Studies,
Shandong University, Jinan, Shandong 250100, China. xiaolexue1989@gmail.com,
xuexiaole.good@163.com. Research supported by NSF (Nos. 11701214 and 11801315)
and Natural Science Foundation of Shandong Province(ZR2018QA001)} }
\maketitle

\begin{abstract}
This paper is concerned with optimal control of stochastic fully coupled
forward-backward linear quadratic (FBLQ) problems with indefinite control
weight costs. In order to obtain the state feedback representation of the
optimal control, we propose a new decoupling technique and obtain one kind of
non-Riccati-type ordinary differential equations (ODEs). By applying the
completion-of-squares method, we prove the existence of the solutions for the
obtained ODEs under some assumptions and derive the state feedback form of the
optimal control. For this FBLQ problem, the optimal control depends on the
entire trajectory of the state process. Some sepcial cases are given to
illustrate our results.

\end{abstract}




\section{Introduction}

The fully coupled forward-backward stochastic differential equations (FBSDEs)
are an important class of stochastic differential equations and there are many
literatures on the well-posedness of them. When the coefficients of a fully
coupled FBSDE are deterministic and the diffusion coefficient of the forward
equation is nondegenerate, Ma, Protter and Yong \cite{MPY} proposed the
four-step scheme approach. Under some monotonicity conditions, Hu and Peng
\cite{Hu-Peng95} first obtained an existence and uniqueness result which was
generalized by Peng and Wu \cite{Peng-W}. Yong \cite{Yong1997} developed this
approach and called it the method of continuation. The fixed point approach is
due to Antonelli \cite{Antonelli}, Pardoux and Tang \cite{Pardoux-Tang}. The
readers may refer to Ma and Yong \cite{Ma-Y}, Cvitani\'{c} and Zhang
\cite{Cvi-Zhang}, Ma, Wu, Zhang and Zhang \cite{Ma-WZZ}, Yong and Zhou
\cite{Yong-Zhou} for the FBSDE theory.

As a well-defined dynamic system, it is appealing to investigate the optimal
control of the fully coupled FBSDEs. In this paper, the optimal control of a
linear fully coupled FBSDE with a quadratic criteria is investigated. We call
this kind of problem the stochastic forward-backward linear-quadratic (FBLQ) problem.

It is well-known that the stochastic linear-quadratic (LQ) problems play an
important role in optimal control theory. On one hand, many nonlinear control
problems can be approximated by the LQ control problems; on the other hand,
solutions to the LQ control problems show elegant properties because of their
brief and beautiful structures. Stochastic LQ regulator problems have been
first studied by Wonham \cite{Wohanm} and by many researchers later
\cite{Ben,Sun-LY,Tang03,Koh-T}. Most of them imposed the positiveness for the
coefficient of the control in the cost functional. Chen, Li and Zhou found
even when the coefficient is negative, the stochastic control problem is still
well-posed (see \cite{Chen-LZ,Chen-Z}). For stochastic LQ problems, one method
is applying the stochastic maximum principle to obtain the optimal control and
then solving the corresponding Hamiltonian system by a decoupling technique
which leads to a Riccati equation. Finally the optimal control is expressed in
the form of state feedback. Another method is the completion-of-squares method
which yields the same Riccati equation and state feedback form of the optimal
control. Dokuchaev and Zhou \cite{Dokuchaev-Zhou} first proposed the
stochastic backward linear-quadratic (BLQ) problem in which the state equation
is described by a backward stochastic differential equation (BSDE). Applying
the completion-of-squares method and the decoupling method, Lim and Zhou
\cite{Lim-Z} completely solved it and obtained the state feedback representation.

Up to our knowledge, there are only a few results for the stochastic FBLQ
problem and except some special examples in the literatures, there are no
systematical results related to the state feedback form of the optimal
control. Our main contribution of this paper is to obtain the state feedback
form of the optimal control for the FBLQ problem. After applying the
stochastic maximum principle, we find that the decoupling technique for
stochastic LQ and BLQ problems is no longer applicable to the FBLQ problem. In
more details, for the stochastic FBLQ problem, the obtained Hamiltonian system
(\ref{eq-ham}) consists of two parts: $(\bar{X}(\cdot),m(\cdot))$ (the forward
state process $\bar{X}(\cdot)$ and its backward adjoint process $m(\cdot)$)
and $(\bar{Y}(\cdot),h(\cdot))$ (the backward state process $\bar{Y}(\cdot)$
and its forward adjoint process $h(\cdot)$). Both of them are fully coupled
FBSDEs. Following the decoupling method for the stochastic LQ problem, we try
to decouple the above Hamiltonian system by
\[%
\begin{array}
[c]{cc}%
h(t)= & P_{1}(t)\bar{X}(t)+P_{2}(t)\bar{Y}(t)+\varphi_{1}(t),\\
m(t)= & P_{3}(t)\bar{X}(t)+P_{4}(t)\bar{Y}(t)+\varphi_{2}(t).
\end{array}
\]
In other words, we want to use the state process $(\bar{X}(\cdot),\bar
{Y}(\cdot))$ to represent the adjoint process $(m(\cdot),h(\cdot))$. But after
calculation, we can't get the Riccati-type equations for $P_{i}(t)$,
$i=1,2,3,4$ through this decoupling approach. To overcome this difficulty, we
propose the following new decoupling technique: we regard the forward
stochastic differential equation (SDE) $(\bar{X}(\cdot),h(\cdot))$ as the
state process, the BSDE $(\bar{Y}(\cdot),m(\cdot))$ as the adjoint process and
decouple the Hamiltonian system (\ref{eq-ham}) by%
\begin{equation}%
\begin{array}
[c]{cc}%
m(t)= & P_{1}(t)\bar{X}(t)+P_{2}(t)^{\intercal}h(t)+\varphi_{1}(t),\\
\bar{Y}(t)= & P_{2}(t)\bar{X}(t)-P_{3}(t)h(t)+\varphi_{2}(t).
\end{array}
\label{eq-decoupling}%
\end{equation}
Using the above decoupling technique, we derive the equations for $P_{i}(t)$,
$i=1,2,3$, $\varphi_{1}(\cdot)$, $\varphi_{2}(\cdot)$ and obtain the optimal
control which can be explicitly expressed as a feedback form of the state
process $(\bar{X}(\cdot),\bar{Y}(\cdot))$ (see Corollary \ref{cor-1}).

Although we can decouple the Hamiltonian system (\ref{eq-ham}) by
(\ref{eq-decoupling}), the obtained equations for $P_{i}(t)$, $i=1,2,3$ are no
longer Riccati-type ones. They are highly nonlinear ordinary differential
equations (ODEs) and the solvability of them is challenging. In this paper, we
propose a project to obtain the existence of the solutions $P_{i}(t)$,
$i=1,2,3$. We first construct a sequence of Riccati equations for $_{i}%
\tilde{P}(t)$. Then, applying the completion-of-squares method, we establish
the the relations between $P_{i}(t)$, $i=1,2,3$ and $_{i}\tilde{P}(t)$ (see
Theorem \ref{th-p-ex}) which are different from the stochastic LQ and BLQ
problems. With the help of these relations and the good properties of
$_{i}\tilde{P}(t)$, we obtain the existence of the solutions $P_{i}(t)$,
$i=1,2,3$. Especially, we relax the positiveness of the control weight in the
cost functional as in Chen et. al \cite{Chen-LZ,Chen-Z}. For this indefinite
case, the control $\bar{u}(\cdot)$ obtained by our decoupling technique is
only a candidate of the optimal control. By applying the completion-of-squares
method, it can be verified that $\bar{u}(\cdot)$ is indeed the optimal control
of the FBLQ problem. Furthermore, although the optimal control for the FBLQ
problem may not be unique, we can still prove that the optimal state feedback
optimal control law is unique (see Theorem \ref{th-infi}). Finally, it is
worth pointing out that we can't solve the FBLQ problem by the decoupling
method or the completion-of-squares method alone.

The rest of the paper is organized as follows. In Section 2, we give the
preliminaries and the formulation of the FBLQ problem. A new decoupling
technique is introduced in Section 3. Applying the completion-of-squares
method, we prove the existence and uniqueness results for non-Riccati-type
equations in Section 4. In Section 5, we obtain the feedback optimal control
for the FBLQ problem. Several special cases are given to illustrate our
results in Section 6.

\section{Preliminaries and formulation of FBLQ problem}

Let $(\Omega,\mathcal{F},P)$ be a complete probability space on which a
standard $d$-dimensional Brownian motion $B=(B_{1}(t),B_{2}(t),...B_{d}%
(t))_{0\leq t\leq T}^{\intercal}$ is defined. Assume that $\mathbb{F=}\{
\mathcal{F}_{t},0\leq t\leq T\}$ is the $P$-augmentation of the natural
filtration of $B$, where $\mathcal{F}_{0}$ contains all $P$-null sets of
$\mathcal{F}$. Denote by $\mathbb{R}^{n}$ the $n$-dimensional real Euclidean
space and $\mathbb{R}^{n\times k}$ the set of $n\times k$ real matrices. Let
$\langle\cdot,\cdot\rangle$ (resp. $\left\vert \cdot\right\vert $) denote the
usual scalar product (resp. usual norm) of $\mathbb{R}^{n}$ and $\mathbb{R}%
^{n\times k}$. The scalar product (resp. norm) of $M=(m_{ij})$, $N=(n_{ij}%
)\in\mathbb{R}^{n\times k}$ is denoted by $\langle M,N\rangle
=tr\{MN^{\intercal}\}$ (resp.$\Vert M\Vert=\sqrt{MM^{\intercal}}$), where the
superscript $^{\intercal}$ denotes the transpose of vectors or matrices.

For each given $p\geq1$, we introduce the following spaces.

\noindent$\mathbb{S}^{n}$: the space of all $n\times n$ symmetric
matrices;\newline\noindent$\mathbb{S}_{+}^{n}$: the subspace of all
nonnegative definite matrices of $\mathbb{S}^{n}$; \newline\noindent
$\mathbb{\hat{S}}_{+}^{n}$: the subspace of all positive definite matrices of
$\mathbb{S}^{n}$; \newline\noindent$L^{p}(\mathcal{F}_{T};\mathbb{R}^{n})$ :
the space of $\mathcal{F}_{T}$-measurable $\mathbb{R}^{n}$-valued random
vectors $\eta$ such that
\[
||\eta||_{p}:=(\mathbb{E}[|\eta|^{p}])^{\frac{1}{p}}<\infty;
\]
\noindent$L^{\infty}(\mathcal{F}_{T};\mathbb{R}^{n})$: the space of
$\mathcal{F}_{T}$-measurable $\mathbb{R}^{n}$-valued random vectors $\eta$
such that
\[
||\eta||_{\infty}=\mathrm{ess~sup}_{\omega\in\Omega}|\eta(\omega)|<\infty;
\]
\noindent$L^{\infty}(0,T;\mathbb{R}^{n\times k})$: the space of essential
bounded measurable $\mathbb{R}^{n\times k}$-valued functions; \newline%
\noindent$C([0,T],\mathbb{R}^{n})$: the space of continuos $\mathbb{R}^{n}%
$-valued functions; \newline\noindent$L_{\mathbb{F}}^{p}(0,T;\mathbb{R}^{n})$:
the space of $\mathbb{F}$-adapted $\mathbb{R}^{n}$-valued stochastic processes
on $[0,T]$ such that
\[
\mathbb{E}\left[  \int_{0}^{T}|f(r)|^{p}dr\right]  <\infty;
\]
\noindent$L_{\mathbb{F}}^{\infty}(0,T;\mathbb{R}^{n})$: the space of
$\mathbb{F}$-adapted $\mathbb{R}^{n}$-valued stochastic processes on $[0,T]$
such that
\[
||f(\cdot)||_{\infty}=\mathrm{ess~sup}_{(t,\omega)\in\lbrack0,T]\times\Omega
}|f(t,\omega)|<\infty;
\]
\noindent$L_{\mathbb{F}}^{p,q}(0,T;\mathbb{R}^{n})$: the space of $\mathbb{F}%
$-adapted $\mathbb{R}^{n}$-valued stochastic processes on $[0,T]$ such that
\[
||f(\cdot)||_{p,q}=\{\mathbb{E}[(\int_{0}^{T}|f(t)|^{p}dt)^{\frac{q}{p}%
}]\}^{\frac{1}{q}}<\infty;
\]
\noindent$L_{\mathbb{F}}^{p}(\Omega;C([0,T],\mathbb{R}^{n}))$: the space of
$\mathbb{F}$-adapted $\mathbb{R}^{n}$-valued continuous stochastic processes
on $[0,T]$ such that
\[
\mathbb{E}[\sup\limits_{0\leq t\leq T}|f(t)|^{p}]<\infty.
\]

Consider the following linear forward-backward stochastic control system
\begin{equation}
\left\{
\begin{array}
[c]{rl}%
dX(t)= & [A_{1}(t)X(t)+B_{1}(t)Y(t)+C_{1}(t)Z(t)+D_{1}(t)u(t)]dt\\
& +[A_{2}(t)X(t)+B_{2}(t)Y(t)+C_{2}(t)Z(t)+D_{2}(t)u(t)]dB(t),\\
dY(t)= & -[A_{3}(t)X(t)+B_{3}(t)Y(t)+C_{3}(t)Z(t)+D_{3}(t)u(t)]dt+Z(t)dB(t),\\
X(0)= & x_{0},\ Y(T)=FX(T)+\xi,
\end{array}
\right.  \label{state-lq}%
\end{equation}
and minimizing the following cost functional
\begin{equation}%
\begin{array}
[c]{rl}%
J(u(\cdot))= & \frac{1}{2}\mathbb{E}\left[  \int_{0}^{T}\left(  \left\langle
A_{4}(t)X(t),X(t)\right\rangle +\left\langle B_{4}(t)Y(t),Y(t)\right\rangle
+\left\langle C_{4}(t)Z(t),Z(t)\right\rangle \right.  \right. \\
& \left.  \left.  +\left\langle D_{4}(t)u(t),u(t)\right\rangle \right)
dt+\left\langle GX(T),X(T)\right\rangle +\left\langle HY(0),Y(0)\right\rangle
\right]
\end{array}
\label{cost-lq}%
\end{equation}
where $A_{i}(\cdot)$, $B_{i}(\cdot)$, $C_{i}(\cdot)$, $D_{i}(\cdot)$ are
deterministic matrix-valued functions of suitable sizes, $\xi\in
L^{2}(\mathcal{F}_{T};\mathbb{R}^{m})$, $F$, $G$, $H$ are $\mathbb{R}^{m\times
n}-$, $\mathbb{R}^{n\times n}-$, $\mathbb{R}^{m\times m}-$valued matrices
respectively. To simplify the presentation, we only consider the case $d=1$.
The results for $d>1$ are similar. The solution to (\ref{state-lq}) is
$(X(\cdot),Y(\cdot),Z(\cdot))\in L_{\mathbb{F}}^{2}(\Omega;C([0,T],\mathbb{R}%
^{n}))\times L_{\mathbb{F}}^{2}(\Omega;C([0,T],\mathbb{R}^{m}))\times
L_{\mathbb{F}}^{2,2}(0,T;\mathbb{R}^{m})$. The admissible control set is all
the elements in $L_{\mathbb{F}}^{2}(0,T;\mathbb{R}^{k})$. Let $u(\cdot)$ be an
admissible control, and the corresponding state is $(X(\cdot),Y(\cdot
),Z(\cdot))$.

Let $\bar{u}(\cdot)$ be an optimal control and $(\bar{X}(\cdot),\bar{Y}%
(\cdot),$ $\bar{Z}(\cdot))$ be the corresponding optimal state. Then by
stochastic maximum principle (see \cite{Peng-1993,Wu98,Hu-JX}), the optimal
control $\bar{u}(\cdot)$ satisfies%
\begin{equation}
D_{4}(t)\bar{u}(t)+D_{1}(t)^{\intercal}m(t)+D_{2}(t)^{\intercal}%
n(t)+D_{3}(t)^{\intercal}h(t)=0, \label{eq-con-ne}%
\end{equation}
where
\begin{equation}
\left\{
\begin{array}
[c]{rl}%
dh(t)= & \left[  B_{3}(t)^{\intercal}h(t)+B_{1}(t)^{\intercal}m(t)+B_{2}%
(t)^{\intercal}n(t)+B_{4}(t)\bar{Y}(t)\right]  dt\\
& +\left[  C_{3}(t)^{\intercal}h(t)+C_{1}(t)^{\intercal}m(t)+C_{2}%
(t)^{\intercal}n(t)+C_{4}(t)\bar{Z}(t)\right]  dB(t),\\
dm(t)= & -\left[  A_{3}(t)^{\intercal}h(t)+A_{1}(t)^{\intercal}m(t)+A_{2}%
(t)^{\intercal}n(t)+A_{4}(t)\bar{X}(t)\right]  dt\\
& +n(t)dB(t),\\
h(0)= & H\bar{Y}(0),\text{ }m(T)=G\bar{X}(T)+F^{\intercal}h(T).
\end{array}
\right.  \label{eq-hmn}%
\end{equation}

\begin{assumption}
\label{assum-1} For any $u(\cdot)\in L_{\mathbb{F}}^{2}(0,T;\mathbb{R}^{k})$,
(\ref{state-lq})(resp. (\ref{eq-hmn})) has a unique solution in $L_{\mathbb{F}%
}^{2}(\Omega;C([0,T],\mathbb{R}^{n}))\times L_{\mathbb{F}}^{2}(\Omega
;C([0,T],\mathbb{R}^{m}))\times L_{\mathbb{F}}^{2,2}(0,T;\mathbb{R}^{m}%
)$(resp. $L_{\mathbb{F}}^{2}(\Omega;C([0,T],\mathbb{R}^{m}))\times
L_{\mathbb{F}}^{2}(\Omega;C([0,T],\mathbb{R}^{n}))\times L_{\mathbb{F}}%
^{2,2}(0,T;\mathbb{R}^{n})$).
\end{assumption}

\begin{remark}
It is well-known that there are many conditions which can guarantee the
existence and uniqueness of (\ref{state-lq}) and (\ref{eq-hmn}) (see
\cite{Ma-Y}, \cite{Cvi-Zhang}, \cite{Peng-W}, \cite{Hu-JX}) such as
monotonicity conditions or weakly coupled conditions and so on.
\end{remark}

\begin{assumption}
\label{assum-bound} The data appearing in the FBLQ problem satisfy
$A_{i}(\cdot)\in L^{\infty}(0,T;\mathbb{R}^{n\times n})$, $B_{i}(\cdot)$,
$C_{i}(\cdot)\in L^{\infty}(0,T;\mathbb{R}^{n\times m})$, $D_{i}(\cdot)\in
L^{\infty}(0,T;\mathbb{R}^{n\times k})$, for $i=1$, $2$, $A_{3}(\cdot)\in
L^{\infty}(0,T;\mathbb{R}^{m\times n})$, $B_{3}(\cdot)$, $C_{3}(\cdot)\in
L^{\infty}(0,T;\mathbb{R}^{m\times m})$, $D_{3}(\cdot)\in L^{\infty
}(0,T;\mathbb{R}^{m\times k})$, $A_{4}(\cdot)\in L^{\infty}(0,T;\mathbb{S}%
^{n})$, $B_{4}(\cdot)$, $C_{4}(\cdot)\in L^{\infty}(0,T;\mathbb{S}^{m})$,
$D_{4}(\cdot)\in L^{\infty}(0,T;\mathbb{S}^{k})$, $F\in\mathbb{R}^{m\times n}%
$, $G\in\mathbb{S}^{n}$, $H\in\mathbb{S}^{m}$.
\end{assumption}

Sometimes we need the data to satisfy the following assumptions:

\begin{assumption}
\label{assum-pos} $A_{4}(\cdot)\in L^{\infty}(0,T;\mathbb{S}_{+}^{n})$,
$B_{4}(\cdot)\in L^{\infty}(0,T;\mathbb{S}_{+}^{m})$, $G\in\mathbb{\hat{S}%
}_{+}^{n}$, $H\in\mathbb{S}_{+}^{m}$.
\end{assumption}

\begin{assumption}
\label{assum-pos-d4} $C_{4}(\cdot)\in L^{\infty}(0,T;\mathbb{S}_{+}^{m})$,
$D_{4}(\cdot)\in L^{\infty}(0,T;\mathbb{\hat{S}}_{+}^{k})$.
\end{assumption}

Note that (\ref{eq-con-ne}) becomes a sufficient condition for the optimal
control under some positiveness assumptions on the coefficients.

\begin{theorem}
\label{th-mp-lq}(see \cite{Meng09,Huang-S})Suppose that Assumptions
\ref{assum-1}, \ref{assum-bound}, \ref{assum-pos} and \ref{assum-pos-d4} hold.
If there exists an admissible control $\bar{u}(\cdot)$ satisfying
(\ref{eq-con-ne}), where $(h(\cdot),m(\cdot),n(\cdot))$ is defined in
(\ref{eq-hmn}), then $\bar{u}(\cdot)$ is the unique optimal control for the
FBLQ problem (\ref{state-lq})-(\ref{cost-lq}).
\end{theorem}

In the rest of this paper, sometimes we write $A$ for a (deterministic or
stochastic) process, omitting the variable $t$, whenever no confusion arises.
Under this convention, when $A\geq(>)0$ means $A(t)\geq(>)0$, $\forall
t\in\lbrack0,T]$.

\section{A new decoupling technique for FBLQ problem}

\subsection{FBLQ problem with positive definite control weight
cost\label{sub-def}}

In this subsection, we only consider the FBLQ problem (\ref{state-lq}%
)-(\ref{cost-lq}) with positive definite control weight cost. In other words,
we assume that $D_{4}>0$. The Hamiltonian system for the FBLQ problem is%
\begin{equation}
\left\{
\begin{array}
[c]{rl}%
d\bar{X}(t)= & [A_{1}(t)\bar{X}(t)+B_{1}(t)\bar{Y}(t)+C_{1}(t)\bar{Z}%
(t)+D_{1}(t)\bar{u}(t)]dt\\
& +[A_{2}(t)\bar{X}(t)+B_{2}(t)\bar{Y}(t)+C_{2}(t)\bar{Z}(t)+D_{2}(t)\bar
{u}(t)]dB(t),\\
d\bar{Y}(t)= & -[A_{3}(t)\bar{X}(t)+B_{3}(t)\bar{Y}(t)+C_{3}(t)\bar
{Z}(t)+D_{3}(t)\bar{u}(t)]dt+\bar{Z}(t)dB(t),\\
dh(t)= & \left[  B_{3}(t)^{\intercal}h(t)+B_{1}(t)^{\intercal}m(t)+B_{2}%
(t)^{\intercal}n(t)+B_{4}(t)\bar{Y}(t)\right]  dt\\
& +\left[  C_{3}(t)^{\intercal}h(t)+C_{1}(t)^{\intercal}m(t)+C_{2}%
(t)^{\intercal}n(t)+C_{4}(t)\bar{Z}(t)\right]  dB(t),\\
dm(t)= & -\left[  A_{3}(t)^{\intercal}h(t)+A_{1}(t)^{\intercal}m(t)+A_{2}%
(t)^{\intercal}n(t)+A_{4}(t)\bar{X}(t)\right]  dt\\
& +n(t)dB(t),\\
\bar{X}(0)= & x_{0},\ \bar{Y}(T)=F\bar{X}(T)+\xi,\text{ }h(0)=H\bar
{Y}(0),\text{ }m(T)=G\bar{X}(T)+F^{\intercal}h(T).
\end{array}
\right.  \label{eq-ham}%
\end{equation}
Set
\[
\tilde{X}(\cdot)=(\bar{X}(\cdot)^{\intercal},h(\cdot)^{\intercal})^{\intercal
}\text{, }\tilde{Y}(\cdot)=(m(\cdot)^{\intercal},\bar{Y}(\cdot)^{\intercal
})^{\intercal}\text{, }\tilde{Z}(\cdot)=(n(\cdot)^{\intercal},\bar{Z}%
(\cdot)^{\intercal})^{\intercal}\text{.}%
\]
Due to (\ref{eq-con-ne}), we have $\bar{u}(t)=-D_{4}(t)^{-1}(D_{1}%
(t)^{\intercal}m(t)+D_{2}(t)^{\intercal}n(t)+D_{3}(t)^{\intercal}h(t))$. Then
the Hamiltonian system (\ref{eq-ham}) can be rewritten as
\begin{equation}
\left\{
\begin{array}
[c]{rl}%
d\tilde{X}(t)= & [\tilde{A}_{1}(t)\tilde{X}(t)+\tilde{B}_{1}(t)\tilde
{Y}(t)+\tilde{C}_{1}(t)\tilde{Z}(t)]dt\\
& +[\tilde{A}_{2}(t)\tilde{X}(t)+\tilde{B}_{2}(t)\tilde{Y}(t)+\tilde{C}%
_{2}(t)\tilde{Z}(t)]dB(t),\\
d\tilde{Y}(t)= & -[\tilde{A}_{3}(t)\tilde{X}(t)+\tilde{B}_{3}(t)\tilde
{Y}(t)+\tilde{C}_{3}(t)\tilde{Z}(t)]dt+\tilde{Z}(t)dB(t),\\
\tilde{X}(0)= & (x_{0}^{\intercal},(H\bar{Y}(0))^{\intercal})^{\intercal
},\ \tilde{Y}(T)=\tilde{F}\tilde{X}(T)+\tilde{\xi},
\end{array}
\right.  \label{eq-xt-til}%
\end{equation}
where%
\[
\tilde{A}_{1}(t)=\left(
\begin{array}
[c]{ccc}%
A_{1}(t) &  & -D_{1}(t)D_{4}(t)^{-1}D_{3}(t)^{\intercal}\\
0 &  & B_{3}(t)^{\intercal}%
\end{array}
\right)  ,\;\tilde{B}_{1}(t)=\left(
\begin{array}
[c]{ccc}%
-D_{1}(t)D_{4}(t)^{-1}D_{1}(t)^{\intercal} &  & B_{1}(t)\\
B_{1}(t)^{\intercal} &  & B_{4}(t)
\end{array}
\right)  ,
\]%
\[
\tilde{C}_{1}(t)=\left(
\begin{array}
[c]{ccc}%
-D_{1}(t)D_{4}(t)^{-1}D_{2}(t)^{\intercal} &  & C_{1}(t)\\
B_{2}(t)^{\intercal} &  & 0
\end{array}
\right)  ,\;\tilde{A}_{2}(t)=\left(
\begin{array}
[c]{ccc}%
A_{2}(t) &  & -D_{2}(t)D_{4}(t)^{-1}D_{3}(t)^{\intercal}\\
0 &  & C_{3}(t)^{\intercal}%
\end{array}
\right)  ,
\]%
\[
\tilde{B}_{2}(t)=\left(
\begin{array}
[c]{ccc}%
-D_{2}(t)D_{4}(t)^{-1}D_{1}(t)^{\intercal} &  & B_{2}(t)\\
C_{1}(t)^{\intercal} &  & 0
\end{array}
\right)  ,\;\tilde{C}_{2}(t)=\left(
\begin{array}
[c]{ccc}%
-D_{2}(t)D_{4}(t)^{-1}D_{2}(t)^{\intercal} &  & C_{2}(t)\\
C_{2}(t)^{\intercal} &  & C_{4}(t)
\end{array}
\right)  ,
\]%
\[
\tilde{A}_{3}(t)=\left(
\begin{array}
[c]{ccc}%
A_{4}(t) &  & A_{3}(t)^{\intercal}\\
A_{3}(t) &  & -D_{3}(t)D_{4}(t)^{-1}D_{3}(t)^{\intercal}%
\end{array}
\right)  ,\;\tilde{B}_{3}(t)=\left(
\begin{array}
[c]{ccc}%
A_{1}(t)^{\intercal} &  & 0\\
-D_{3}(t)D_{4}(t)^{-1}D_{1}(t)^{\intercal} &  & B_{3}(t)
\end{array}
\right)  ,
\]%
\[
\tilde{C}_{3}(t)=\left(
\begin{array}
[c]{ccc}%
A_{2}(t)^{\intercal} &  & 0\\
-D_{3}(t)D_{4}(t)^{-1}D_{2}(t)^{\intercal} &  & C_{3}(t)
\end{array}
\right)  ,\;\tilde{F}=\left(
\begin{array}
[c]{ccc}%
G &  & F^{\intercal}\\
F &  & 0
\end{array}
\right)  ,\text{\ \ }\tilde{\xi}=\left(
\begin{array}
[c]{c}%
0\\
\xi
\end{array}
\right)  .
\]

In order to obtain the state feedback form of the optimal control, the
following new decoupling technique is introduced: we conjecture that
$\tilde{X}(\cdot)$ and $\tilde{Y}(\cdot)$ are related by
\[
\tilde{Y}(t)=Q(t)\tilde{X}(t)+\varphi(t)
\]
with $Q(\cdot)\in C([0,T],\mathbb{R}^{(n+m)\times(n+m)})$ and $\varphi
(\cdot)\in L_{\mathbb{F}}^{2}(\Omega;C([0,T],\mathbb{R}^{n+m}))$. Applying the
same steps as in Section 4 of \cite{Yong-06} or Appendix in \cite{Hu-JX}, we
obtain $Q(\cdot)$ satisfies the following matrix ODE%
\begin{equation}
\left\{
\begin{array}
[c]{rl}%
dQ(t)= & -\left[  Q(t)\tilde{A}_{1}(t)+Q(t)\tilde{B}_{1}(t)Q(t)+Q(t)\tilde
{C}_{1}(t)K(t)+\tilde{A}_{3}(t)+\tilde{B}_{3}(t)Q(t)\right. \\
& \left.  +\tilde{C}_{3}(t)K(t)\right]  dt,\\
Q(T)= & \tilde{F},
\end{array}
\right.  \label{eq-p-til}%
\end{equation}
and $(\varphi(\cdot),v(\cdot))\in L_{\mathbb{F}}^{2}(\Omega;C([0,T],\mathbb{R}%
^{n+m}))\times L_{\mathbb{F}}^{2,2}(0,T;\mathbb{R}^{n+m})$ satisfies the
following linear BSDE%
\begin{equation}
\left\{
\begin{array}
[c]{rl}%
d\varphi(t)= & -\left\{  \left[  Q(t)\tilde{B}_{1}(t)+\tilde{B}_{3}(t)+\left(
Q(t)\tilde{C}_{1}(t)+\tilde{C}_{3}(t)\right)  \right.  \right. \\
& \left.  \left.  \cdot(I_{n+m}-Q(t)\tilde{C}_{2}(t))^{-1}Q(t)\tilde{B}%
_{2}(t)\right]  \varphi(t)\right. \\
& \left.  +\left(  Q(t)\tilde{C}_{1}(t)+\tilde{C}_{3}(t)\right)
(I_{n+m}-Q(t)\tilde{C}_{2}(t))^{-1}v(t)\right\}  dt+v(t)dB(t),\\
\varphi(T)= & \tilde{\xi},
\end{array}
\right.  \label{eq-phi}%
\end{equation}
where
\[%
\begin{array}
[c]{rl}%
K(t)= & (I_{n+m}-Q(t)\tilde{C}_{2}(t))^{-1}\left(  Q(t)\tilde{A}%
_{2}(t)+Q(t)\tilde{B}_{2}(t)Q(t)\right)  .
\end{array}
\]

Set%
\[
Q(t)=\left(
\begin{array}
[c]{ccc}%
Q_{1}(t) &  & Q_{2}(t)\\
Q_{3}(t) &  & -Q_{4}(t)
\end{array}
\right)  ,\text{ }K(t)=\left(
\begin{array}
[c]{ccc}%
K_{1}(t) &  & K_{2}(t)\\
K_{3}(t) &  & K_{4}(t)
\end{array}
\right)  ,\text{ }\varphi\left(  \cdot\right)  =\left(
\begin{array}
[c]{c}%
\varphi_{1}(\cdot)\\
\varphi_{2}(\cdot)
\end{array}
\right)  ,
\]%
\[
v\left(  \cdot\right)  =\left(
\begin{array}
[c]{c}%
v_{1}(\cdot)\\
v_{2}(\cdot)
\end{array}
\right)  ,\ \left(
\begin{array}
[c]{ccc}%
J_{1}(t) &  & J_{2}(t)\\
J_{3}(t) &  & J_{4}(t)
\end{array}
\right)  =(I_{n+m}-Q(t)\tilde{C}_{2}(t))^{-1}Q(t)\tilde{B}_{2}(t),
\]%
\[
\text{ }\left(
\begin{array}
[c]{ccc}%
I_{1}(t) &  & I_{2}(t)\\
I_{3}(t) &  & I_{4}(t)
\end{array}
\right)  =(I_{n+m}-Q(t)\tilde{C}_{2}(t))^{-1},
\]
where $Q_{1}(\cdot)$, $K_{1}(\cdot)$, $J_{1}(\cdot)$, $I_{1}(\cdot)$ are
$\mathbb{R}^{n\times n}$-valued, $Q_{2}(\cdot)$, $K_{2}(\cdot)$, $J_{2}%
(\cdot)$, $I_{2}(\cdot)$ are $\mathbb{R}^{n\times m}$-valued, $Q_{3}(\cdot)$,
$K_{3}(\cdot)$, $J_{3}(\cdot)$, $I_{3}(\cdot)$ are $\mathbb{R}^{m\times n}%
$-valued, $Q_{4}(\cdot)$, $K_{4}(\cdot)$, $J_{4}(\cdot)$, $I_{4}(\cdot)$ are
$\mathbb{R}^{m\times m}$-valued, $\varphi_{1}(\cdot)\in L_{\mathbb{F}}%
^{2}(\Omega;C([0,T],\mathbb{R}^{n}))$, $\varphi_{2}(\cdot)\in L_{\mathbb{F}%
}^{2}(\Omega;C([0,T],\mathbb{R}^{m}))$, $v_{1}(\cdot)\in L_{\mathbb{F}}%
^{2,2}(0,T;\mathbb{R}^{n})$, $v_{2}(\cdot)\in L_{\mathbb{F}}^{2,2}%
(0,T;\mathbb{R}^{m})$.

\begin{theorem}
\label{th-opti-cont}Suppose that Assumptions \ref{assum-1}, \ref{assum-bound},
\ref{assum-pos} and \ref{assum-pos-d4} hold. Moreover, suppose that
(\ref{eq-p-til}) has a solution $Q(\cdot)\in C\left(  \left[  0,T\right]
;\mathbb{R}^{(n+m)\times(n+m)}\right)  $ such that $I_{m}+HQ_{4}(0)$ is
invertible and $(I_{n+m}-Q(t)\tilde{C}_{2}(t))^{-1}\in L^{\infty}\left(
0,T;\mathbb{R}^{(n+m)\times(n+m)}\right)  $. Then Problem (\ref{state-lq}%
)-(\ref{cost-lq}) has a unique optimal control
\begin{equation}%
\begin{array}
[c]{rl}%
\bar{u}(t)= & -D_{4}(t)^{-1}\left(  D_{1}(t)^{\intercal}Q_{1}(t)+D_{2}%
(t)^{\intercal}K_{1}(t)\right)  X^{\ast}(t)\\
& -D_{4}(t)^{-1}\left(  D_{1}(t)^{\intercal}Q_{2}(t)+D_{2}(t)^{\intercal}%
K_{2}(t)+D_{3}(t)^{\intercal}\right)  h^{\ast}(t)\\
& -D_{4}(t)^{-1}\left[  D_{1}(t)^{\intercal}\varphi_{1}(t)+D_{2}%
(t)^{\intercal}\left(  J_{1}(t)\varphi_{1}(t)+J_{2}(t)\varphi_{2}(t)\right.
\right. \\
& \ \ \left.  \left.  +I_{1}(t)v_{1}(t)+I_{2}(t)v_{2}(t)\right)  \right]  ,
\end{array}
\label{eq-op-con}%
\end{equation}
where $\tilde{X}^{\ast}(\cdot):=\left(  X^{\ast}(\cdot)^{\intercal},h^{\ast
}(\cdot)^{\intercal}\right)  ^{\intercal}$ is the solution to the following
SDE
\begin{equation}
\left\{
\begin{array}
[c]{rl}%
d\tilde{X}^{\ast}(t)= & \left\{  \left(  \tilde{A}_{1}(t)+\tilde{B}%
_{1}(t)Q(t)+\tilde{C}_{1}(t)K(t)\right)  \tilde{X}^{\ast}(t)\right. \\
& \left.  +\tilde{B}_{1}(t)\varphi(t)+\tilde{C}_{1}(t)(I_{n+m}-Q(t)\tilde
{C}_{2}(t))^{-1}\left[  Q(t)\tilde{B}_{2}(t)\varphi(t)+v(t)\right]
\displaystyle\right\}  dt\\
& +\left\{  \left(  \tilde{A}_{2}(t)+\tilde{B}_{2}(t)Q(t)+\tilde{C}%
_{2}(t)K(t)\right)  \tilde{X}^{\ast}(t)\right. \\
& \left.  +\tilde{B}_{2}(t)\varphi(t)+\tilde{C}_{2}(t)(I_{n+m}-Q(t)\tilde
{C}_{2}(t))^{-1}\left[  Q(t)\tilde{B}_{2}(t)\varphi(t)+v(t)\right]
\displaystyle\right\}  dB(t),\\
\tilde{X}^{\ast}(0)= & \left(  x_{0}^{\intercal},(\left(  I_{m}+HQ_{4}%
(0)\right)  ^{-1}H\left(  Q_{3}(0)x_{0}+\varphi_{2}(0))\right)  ^{\intercal
}\right)  ^{\intercal}.
\end{array}
\right.  \label{eq-op-xh}%
\end{equation}
Furthermore, the solution to (\ref{eq-ham}) with respect to $\bar{u}(\cdot)$
defined in (\ref{eq-op-con}) satisfies
\begin{equation}%
\begin{array}
[c]{cl}%
\bar{X}(t)= & X^{\ast}(t),\text{ \ }h(t)=h^{\ast}(t),\text{ \ }\bar
{Y}(t)=Q_{3}(t)X^{\ast}(t)-Q_{4}(t)h^{\ast}(t)+\varphi_{2}(t),\\
\bar{Z}(t)= & K_{3}(t)X^{\ast}(t)+K_{4}(t)h^{\ast}(t)+J_{3}(t)\varphi
_{1}(t)+J_{4}(t)\varphi_{2}(t)\\
& +I_{3}(t)v_{1}(t)+I_{4}(t)v_{2}(t),\\
m(t)= & Q_{1}(t)X^{\ast}(t)+Q_{2}(t)h^{\ast}(t)+\varphi_{1}(t),\\
n(t)= & K_{1}(t)X^{\ast}(t)+K_{2}(t)h^{\ast}(t)+J_{1}(t)\varphi_{1}%
(t)+J_{2}(t)\varphi_{2}(t)\\
& +I_{1}(t)v_{1}(t)+I_{2}(t)v_{2}(t).
\end{array}
\label{eq-131-1}%
\end{equation}

\end{theorem}

\begin{proof}
By $(I_{n+m}-Q(t)\tilde{C}_{2}(t))^{-1}\in L^{\infty}\left(  0,T;\mathbb{R}%
^{(n+m)\times(n+m)}\right)  $, one has that $K(\cdot)\in L^{\infty}\left(
0,T;\mathbb{R}^{(n+m)\times(n+m)}\right)  $ and (\ref{eq-phi}) is a BSDE with
Lipschitz coefficients. Then (\ref{eq-phi}) has a unique solution $\left(
\varphi(\cdot),v(\cdot)\right)  \in L_{\mathbb{F}}^{2}(\Omega;C([0,T],\newline%
\mathbb{R}^{n+m}))\times$ $L_{\mathbb{F}}^{2,2}(0,T;\mathbb{R}^{n+m})$. It
yields that the stochastic differential equation (\ref{eq-op-xh}) admits a
unique strong solution $\left(  X^{\ast}(\cdot)^{\intercal},h^{\ast}%
(\cdot)^{\intercal}\right)  ^{\intercal}\in$ $L_{\mathbb{F}}^{2}%
(\Omega;C([0,T],\mathbb{R}^{n+m}))$. Thus the control $\bar{u}(\cdot)$ defined
in (\ref{eq-op-con}) is admissible. Putting this $\bar{u}(\cdot)$ into
(\ref{eq-ham}) and reversing the above decoupling technique, it can be
verified that $(\bar{X}(\cdot),$ $h(\cdot),$ $\bar{Y}(\cdot),$ $\bar{Z}%
(\cdot),$ $m(\cdot),$ $n(\cdot))$ defined in (\ref{eq-131-1}) solves
(\ref{eq-ham}) and $\bar{u}(\cdot)$ satisfies (\ref{eq-con-ne}). By Theorem
\ref{th-mp-lq}, this $\bar{u}(\cdot)$ is the unique optimal control. This
completes the proof.
\end{proof}

\begin{remark}
We give a sufficient condition which guarantee the existence of solution to
(\ref{eq-p-til}) in Corollary \ref{Cor-equation Q}.
\end{remark}

\begin{corollary}
\label{cor-1}(i) Under the same assumptions as in Theorem \ref{th-opti-cont},
if $Q_{4}(\cdot)$ in (\ref{eq-p-til}) is invertible on $[0,T)$, then
\[
h(t)=Q_{4}(t)^{-1}Q_{3}(t)\bar{X}(t)-Q_{4}(t)^{-1}\bar{Y}(t)+Q_{4}%
(t)^{-1}\varphi_{2}(t)
\]
and%
\[%
\begin{array}
[c]{rl}%
\bar{u}(t)= & -D_{4}(t)^{-1}\left\{  D_{1}(t)^{\intercal}Q_{1}(t)+D_{2}%
(t)^{\intercal}K_{1}(t)\right. \\
& \left.  +\left[  D_{1}(t)^{\intercal}Q_{2}(t)+D_{2}(t)^{\intercal}%
K_{2}(t)+D_{3}(t)^{\intercal}\right]  Q_{4}(t)^{-1}Q_{3}(t)\right\}  \bar
{X}(t)\\
& +D_{4}(t)^{-1}\left[  D_{1}(t)^{\intercal}Q_{2}(t)+D_{2}(t)^{\intercal}%
K_{2}(t)+D_{3}(t)^{\intercal}\right]  Q_{4}(t)^{-1}\left(  \bar{Y}%
(t)-\varphi_{2}(t)\right) \\
& -D_{4}(t)^{-1}\left[  D_{1}(t)^{\intercal}\varphi_{1}(t)+D_{2}%
(t)^{\intercal}\left(  J_{1}(t)\varphi_{1}(t)+J_{2}(t)\varphi_{2}(t)\right.
\right. \\
& \left.  \left.  +I_{1}(t)v_{1}(t)+I_{2}(t)v_{2}(t)\right)  \right]  .
\end{array}
\]
(ii) If $\xi=0$, then the optimal control for the fully coupled
forward-backward control system in Theorem \ref{th-opti-cont} depends only on
$(\bar{X}(\cdot),h(\cdot))$. Moreover, $h(\cdot)$ has the following
closed-form:%
\[%
\begin{array}
[c]{rl}%
h(t)= & \left(  I_{m}+HQ_{4}(0)\right)  ^{-1}H\left(  Q_{3}(0)x_{0}%
+\varphi_{2}(0)\right) \\
& +\Phi(t)\int_{0}^{t}\Phi(s)^{-1}\left(  b_{1}(s)-a_{2}(s)b_{2}(s)\right)
\bar{X}(s)ds\\
& +\Phi(t)\int_{0}^{t}\Phi(s)^{-1}b_{2}(s)\bar{X}(s)dB(s),
\end{array}
\]
where $a_{1}(t)=B_{3}(t)^{\intercal}-B_{4}(t)Q_{4}(t)+B_{1}(t)^{\intercal
}Q_{2}(t)+B_{2}(t)^{\intercal}K_{2}(t)$, $b_{1}(t)=B_{4}(t)Q_{3}%
(t)+B_{1}(t)^{\intercal}Q_{1}(t)+B_{2}(t)^{\intercal}K_{1}(t)$, $a_{2}%
(t)=C_{3}(t)^{\intercal}+C_{4}(t)K_{4}(t)+C_{1}(t)^{\intercal}Q_{2}%
(t)+C_{2}(t)^{\intercal}K_{2}(t)$, $b_{2}(t)=C_{4}(t)K_{3}(t)+C_{1}%
(t)^{\intercal}Q_{1}(t)+C_{2}(t)^{\intercal}K_{1}(t)$, and $\Phi(\cdot)$ is
the solution of the following linear equation:%
\[%
\begin{array}
[c]{rl}%
d\Phi(t)= & a_{1}(t)\Phi(t)dt+a_{2}(t)\Phi(t)dB(t),\text{ }\Phi(0)=I_{m}.
\end{array}
\]

\end{corollary}

This corollary can be directly derived from Theorem \ref{th-opti-cont}. So we
omit the proof.

\begin{remark}
By Corollary \ref{cor-1}, the optimal control at time $t$ depends on the
entire past history of the state process $X(\cdot)$. This is different from
the classical stochastic LQ problems. Furthermore, if $Q_{4}(\cdot)$ in
(\ref{eq-p-til}) is invertible on $[0,T)$, then the optimal control at time
$t$ will depend only on the current state pair $(\bar{X}(t),\bar{Y}(t)).$
\end{remark}

\subsection{FBLQ problem with indefinite control weight cost\label{indef}}

In this subsection, we relax the assumption $D_{4}>0$ and deduce formally the
following non-Riccati-type equations (\ref{eq-p1}), (\ref{eq-p3}) and
(\ref{eq-p4}) which play an important role in solving the FBLQ problem (see
Section \ref{section-feedback}).

Set
\begin{equation}%
\begin{array}
[c]{cc}%
m(t)= & P_{1}(t)\bar{X}(t)+P_{2}(t)^{\intercal}h(t)+\varphi_{1}(t),\\
\bar{Y}(t)= & P_{2}(t)\bar{X}(t)-P_{3}(t)h(t)+\varphi_{2}(t),
\end{array}
\label{eq-Y-m}%
\end{equation}
where $(\bar{X}(\cdot),\bar{Y}(\cdot),\bar{Z}(\cdot),h(\cdot),m(\cdot
),n(\cdot))$ is the solution to Hamiltonian system (\ref{eq-ham}),
$P_{i}(\cdot),i=1,2,3$ satisfy some ODEs which will be determined later, and
$\varphi\left(  \cdot\right)  =\left(  \varphi_{1}(\cdot)^{\intercal}%
,\varphi_{2}(\cdot)^{\intercal}\right)  ^{\intercal}$, $v\left(  \cdot\right)
=\left(  v_{1}(\cdot)^{\intercal},v_{2}(\cdot)^{\intercal}\right)
^{\intercal}$ satisfies the following BSDE%
\[
d\varphi_{1}(t)=-\gamma_{1}(t)dt+v_{1}(t)dB(t),\text{ }d\varphi_{2}%
(t)=-\gamma_{2}(t)dt+v_{2}(t)dB(t).
\]
Applying It\^{o}'s formula to $\bar{Y}(\cdot)$, $m(\cdot)$ in (\ref{eq-Y-m})
and comparing with the diffusion terms of the equation (\ref{eq-ham}), we have%
\begin{equation}%
\begin{array}
[c]{l}%
\bar{Z}(t)\\
=L_{1}(t)^{-1}\left\{  (P_{2}(t)A_{2}(t)+P_{2}(t)B_{2}(t)P_{2}(t)-P_{3}%
(t)C_{1}(t)^{\intercal}P_{1}(t))\bar{X}(t)\right. \\
\ \ -(P_{2}(t)B_{2}(t)P_{3}(t)+P_{3}(t)C_{3}(t)^{\intercal}+P_{3}%
(t)C_{1}(t)^{\intercal}P_{2}(t)^{\intercal})h(t)+P_{2}(t)D_{2}(t)\bar{u}(t)\\
\ \ \left.  +P_{2}(t)B_{2}(t)\varphi_{2}(t)-P_{3}(t)C_{1}(t)^{\intercal
}\varphi_{1}(t)-P_{3}(t)C_{2}(t)^{\intercal}n(t)+v_{2}(t)\right\}  ,
\end{array}
\label{eq-Zbar}%
\end{equation}%
\begin{equation}%
\begin{array}
[c]{l}%
0=-(I_{n}-P_{2}(t)^{\intercal}C_{2}(t)^{\intercal})n(t)\\
\text{ \ }+(P_{1}(t)A_{2}(t)+P_{1}(t)B_{2}(t)P_{2}(t)+P_{2}(t)^{\intercal
}C_{1}(t)^{\intercal}P_{1}(t))\bar{X}(t)\\
\text{ \ }+(P_{2}(t)^{\intercal}C_{3}(t)^{\intercal}+P_{2}(t)^{\intercal}%
C_{1}(t)^{\intercal}P_{2}(t)^{\intercal}-P_{1}(t)B_{2}(t)P_{3}(t))h(t)+P_{1}%
(t)D_{2}(t)\bar{u}(t)\\
\text{ \ }+(P_{1}(t)C_{2}(t)+P_{1}(t)B_{2}(t)\varphi_{2}(t)+P_{2}%
(t)^{\intercal}C_{4}(t))\bar{Z}(t)+P_{2}(t)^{\intercal}C_{1}(t)^{\intercal
}\varphi_{1}(t)+v_{1}(t),
\end{array}
\label{eq-n}%
\end{equation}
where
\begin{equation}
L_{1}(t)=I_{m}-P_{2}(t)C_{2}(t)+P_{3}(t)C_{4}(t). \label{eq-l1}%
\end{equation}
Combining (\ref{eq-Zbar}) and (\ref{eq-n}), we have
\begin{equation}%
\begin{array}
[c]{rl}%
n(t)= & L_{2}(t)^{-1}(L_{3}(t)\bar{X}(t)+L_{4}(t)h(t)+S_{1}(t)D_{2}(t)\bar
{u}(t)+S_{2}(t)),
\end{array}
\label{eq-n-u}%
\end{equation}
where
\begin{equation}%
\begin{array}
[c]{rl}%
L_{2}(t)= & I_{n}-P_{2}(t)^{\intercal}C_{2}(t)^{\intercal}+\left(
P_{1}(t)C_{2}(t)+P_{2}(t)^{\intercal}C_{4}(t)\right)  L_{1}(t)^{-1}%
P_{3}(t)C_{2}(t)^{\intercal},\\
L_{3}(t)= & P_{1}(t)A_{2}(t)+P_{1}(t)B_{2}(t)P_{2}(t)+P_{2}(t)^{\intercal
}C_{1}(t)^{\intercal}P_{1}(t)\\
& +\left(  P_{2}(t)^{\intercal}C_{4}(t)+P_{1}(t)C_{2}(t)\right)  L_{1}%
(t)^{-1}\\
& \ \ \cdot(P_{2}(t)A_{2}(t)+P_{2}(t)B_{2}(t)P_{2}(t)-P_{3}(t)C_{1}%
(t)^{\intercal}P_{1}(t)),\\
L_{4}(t)= & P_{2}(t)^{\intercal}C_{3}(t)^{\intercal}+P_{2}(t)^{\intercal}%
C_{1}(t)^{\intercal}P_{2}(t)^{\intercal}-P_{1}(t)B_{2}(t)P_{3}(t)\\
& -\left(  P_{2}(t)^{\intercal}C_{4}(t)+P_{1}(t)C_{2}(t)\right)  L_{1}%
(t)^{-1}\\
& \ \ \cdot(P_{2}(t)B_{2}(t)P_{3}(t)+P_{3}(t)C_{3}(t)^{\intercal}%
+P_{3}(t)C_{1}(t)^{\intercal}P_{2}(t)^{\intercal}),\\
S_{1}(t)= & P_{1}(t)+\left(  P_{2}(t)^{\intercal}C_{4}(t)+P_{1}(t)C_{2}%
(t)\right)  L_{1}(t)^{-1}P_{2}(t),\\
S_{2}(t)= & P_{1}(t)B_{2}(t)\varphi_{2}(t)+P_{2}(t)^{\intercal}C_{1}%
(t)^{\intercal}\varphi_{1}(t)+v_{1}(t)\\
& +(P_{1}(t)C_{2}(t)+P_{2}(t)^{\intercal}C_{4}(t))L_{1}(t)^{-1}\\
& \ \ \cdot\lbrack P_{2}(t)B_{2}(t)\varphi_{2}(t)-P_{3}(t)C_{1}(t)^{\intercal
}\varphi_{1}(t)+v_{2}(t)].
\end{array}
\label{def-lr}%
\end{equation}
Putting them into (\ref{eq-con-ne}), we obtain%
\begin{equation}
\bar{u}(t)=L_{6}(t)\bar{X}(t)+L_{7}(t)h(t)+S_{3}(t), \label{eq-u-infi}%
\end{equation}
where%
\begin{equation}%
\begin{array}
[c]{rl}%
L_{5}(t)= & D_{4}(t)+D_{2}(t)^{\intercal}L_{2}(t)^{-1}S_{1}(t)D_{2}(t),\\
L_{6}(t)= & -L_{5}(t)^{-1}(D_{1}(t)^{\intercal}P_{1}(t)+D_{2}(t)^{\intercal
}L_{2}(t)^{-1}L_{3}(t)),\\
L_{7}(t)= & -L_{5}(t)^{-1}(D_{1}(t)^{\intercal}P_{2}(t)^{\intercal}%
+D_{2}(t)^{\intercal}L_{2}(t)^{-1}L_{4}(t)+D_{3}(t)^{\intercal}),\\
S_{3}(t)= & -L_{5}(t)^{-1}\left[  D_{1}(t)^{\intercal}\varphi_{1}%
(t)+D_{2}(t)^{\intercal}L_{2}(t)^{-1}S_{2}(t)\right]  .
\end{array}
\label{def-L1}%
\end{equation}

\begin{remark}
Instead of requiring $D_{4}>0$, here we assume that $L_{5}(t)$ is invertible.
\end{remark}

From (\ref{eq-Zbar})-(\ref{eq-u-infi}), we deduce that%
\begin{equation}%
\begin{array}
[c]{cc}%
n(t)= & L_{8}(t)\bar{X}(t)+L_{9}(t)h(t)+S_{4}(t),\\
\bar{Z}(t)= & L_{10}(t)\bar{X}(t)+L_{11}(t)h(t)+S_{5}(t),
\end{array}
\label{eq-nz}%
\end{equation}
where
\begin{equation}%
\begin{array}
[c]{rl}%
L_{8}(t)= & L_{2}(t)^{-1}\left(  L_{3}(t)+S_{1}(t)D_{2}(t)L_{6}(t)\right)  ,\\
L_{9}(t)= & L_{2}(t)^{-1}\left(  L_{4}(t)+S_{1}(t)D_{2}(t)L_{7}(t)\right)  ,\\
S_{4}(t)= & L_{2}(t)^{-1}\left[  S_{1}(t)D_{2}(t)S_{3}(t)+S_{2}(t)\right]  ,\\
L_{10}(t)= & L_{1}(t)^{-1}\left[  P_{2}(t)A_{2}(t)+P_{2}(t)B_{2}%
(t)P_{2}(t)-P_{3}(t)C_{1}(t)^{\intercal}P_{1}(t)\right. \\
& \left.  -P_{3}(t)C_{2}(t)^{\intercal}L_{8}(t)+P_{2}(t)D_{2}(t)L_{6}%
(t)\right]  ,\\
L_{11}(t)= & L_{1}(t)^{-1}\left[  P_{2}(t)D_{2}(t)L_{7}(t)-P_{3}%
(t)C_{2}(t)^{\intercal}L_{9}(t)\right. \\
& \left.  -P_{2}(t)B_{2}(t)P_{3}(t)-P_{3}(t)C_{3}(t)^{\intercal}-P_{3}%
(t)C_{1}(t)^{\intercal}P_{2}(t)^{\intercal}\right]  ,\\
S_{5}(t)= & L_{1}(t)^{-1}[P_{2}(t)D_{2}(t)S_{3}(t)-P_{3}(t)C_{2}%
(t)^{\intercal}S_{4}(t)+P_{2}(t)B_{2}(t)\varphi_{2}(t)\\
& -P_{3}(t)C_{1}(t)^{\intercal}\varphi_{1}(t)+v_{2}(t)].
\end{array}
\label{def-l8}%
\end{equation}

Now we determine the equations satisfied by $P_{i}(\cdot)$, $i=1,2,3$. We
first put (\ref{eq-Y-m}), (\ref{eq-u-infi}) and (\ref{eq-nz}) into
(\ref{eq-ham}) and obtain a new form of the Hamiltonian system (\ref{eq-ham}).
Then applying It\^{o}'s formula to $m(t)$ in (\ref{eq-Y-m}) and comparing with
the drift term of the new form of (\ref{eq-ham}), we have
\begin{equation}%
\begin{array}
[c]{l}%
\dot{P}_{1}(t)\bar{X}(t)+P_{1}(t)\left\{  A_{1}(t)\bar{X}(t)+B_{1}(t)\left[
P_{2}(t)\bar{X}(t)-P_{3}(t)h(t)+\varphi_{2}(t)\right]  \right. \\
+C_{1}(t)\left[  L_{10}(t)\bar{X}(t)+L_{11}(t)h(t)+S_{5}(t)\right]
+D_{1}(t)\left[  L_{6}(t)\bar{X}(t)+L_{7}(t)h(t)+S_{3}(t)\right]  \}\\
+\dot{P}_{2}(t)^{\intercal}h(t)+P_{2}(t)^{\intercal}\left\{  B_{3}%
(t)^{\intercal}h(t)+B_{1}(t)^{\intercal}\left[  P_{1}(t)\bar{X}(t)+P_{2}%
(t)^{\intercal}h(t)+\varphi_{1}(t)\right]  \right. \\
+B_{2}(t)^{\intercal}\left[  L_{8}(t)\bar{X}(t)+L_{9}(t)h(t)+S_{4}(t)\right]
+B_{4}(t)\left[  P_{2}(t)\bar{X}(t)-P_{3}(t)h(t)+\varphi_{2}(t)\right]  \}\\
-\gamma_{1}(t)\\
=-\left\{  A_{3}(t)^{\intercal}h(t)+A_{1}(t)^{\intercal}\left[  P_{1}%
(t)\bar{X}(t)+P_{2}(t)^{\intercal}h(t)+\varphi_{1}(t)\right]  \right. \\
\left.  +A_{2}(t)^{\intercal}\left[  L_{8}(t)\bar{X}(t)+L_{9}(t)h(t)+S_{4}%
(t)\right]  +A_{4}(t)\bar{X}(t)\right\}  .
\end{array}
\label{eq-m-equal}%
\end{equation}
Hence, $P_{1}(\cdot)$, $P_{2}(\cdot)^{\intercal}$ and $\varphi_{2}(\cdot)$
should be solutions of
\begin{equation}
\left\{
\begin{array}
[c]{l}%
dP_{1}(t)\\
=-\{P_{1}(t)A_{1}(t)+P_{1}(t)B_{1}(t)P_{2}(t)+P_{1}(t)C_{1}(t)L_{10}%
(t)+P_{1}(t)D_{1}(t)L_{6}(t)\\
\ \ +P_{2}(t)^{\intercal}B_{1}(t)^{\intercal}P_{1}(t)+P_{2}(t)^{\intercal
}B_{2}(t)^{\intercal}L_{8}(t)+P_{2}(t)^{\intercal}B_{4}(t)P_{2}(t)+A_{4}(t)\\
\ \ +A_{1}(t)^{\intercal}P_{1}(t)+A_{2}(t)^{\intercal}L_{8}(t)\}dt,\\
P_{1}(T)=G,
\end{array}
\right.  \label{eq-p1}%
\end{equation}%
\begin{equation}
\left\{
\begin{array}
[c]{l}%
dP_{2}(t)^{\intercal}\\
=-\{-P_{1}B_{1}(t)P_{3}(t)+P_{1}(t)C_{1}(t)L_{11}(t)+P_{1}(t)D_{1}%
(t)L_{7}(t)+P_{2}(t)^{\intercal}B_{3}(t)^{\intercal}\\
\ \ +P_{2}(t)^{\intercal}B_{1}(t)^{\intercal}P_{2}(t)^{\intercal}%
+P_{2}(t)^{\intercal}B_{2}(t)^{\intercal}L_{9}(t)-P_{2}(t)^{\intercal}%
B_{4}(t)P_{3}(t)\\
\ \ +A_{3}(t)^{\intercal}+A_{1}(t)^{\intercal}P_{2}(t)^{\intercal}%
+A_{2}(t)^{\intercal}L_{9}(t)\}dt,\\
P_{2}(T)^{\intercal}=F^{\intercal},
\end{array}
\right.  \label{eq-p2}%
\end{equation}%
\begin{equation}
\left\{
\begin{array}
[c]{rl}%
d\varphi_{1}(t)= & -\left\{  P_{1}(t)\left[  B_{1}(t)\varphi_{2}%
(t)+C_{1}(t)S_{5}(t)+D_{1}(t)S_{3}(t)\right]  \right. \\
& +P_{2}(t)^{\intercal}\left[  B_{1}(t)^{\intercal}\varphi_{1}(t)+B_{2}%
(t)^{\intercal}S_{4}(t)+B_{4}(t)\varphi_{2}(t)\right] \\
& \left.  +A_{1}(t)^{\intercal}\varphi_{1}(t)+A_{2}(t)^{\intercal}%
S_{4}(t)\right\}  dt+v_{1}(t)dB(t),\\
\varphi_{1}(T)= & 0
\end{array}
\right.  \label{eq-phi1}%
\end{equation}
respectively. Applying It\^{o}'s formula to $\bar{Y}(t)$ in (\ref{eq-Y-m}) and
comparing with the drift term of the new form of (\ref{eq-ham}), we have%
\begin{equation}%
\begin{array}
[c]{l}%
\dot{P}_{2}(t)\bar{X}(t)+P_{2}(t)\{A_{1}(t)\bar{X}(t)+B_{1}(t)\left[
P_{2}(t)\bar{X}(t)-P_{3}(t)h(t)+\varphi_{2}(t)\right] \\
+C_{1}(t)\left[  L_{10}(t)\bar{X}(t)+L_{11}(t)h(t)+S_{5}(t)\right]
+D_{1}(t)\left[  L_{6}(t)\bar{X}(t)+L_{7}(t)h(t)+S_{3}(t)\right]  \}\\
-\dot{P}_{3}(t)h(t)-P_{3}(t)\{B_{3}(t)^{\intercal}h(t)+B_{1}(t)^{\intercal
}\left[  P_{1}(t)\bar{X}(t)+P_{2}(t)^{\intercal}h(t)+\varphi_{1}(t)\right] \\
+B_{2}(t)^{\intercal}\left[  L_{8}(t)\bar{X}(t)+L_{9}(t)h(t)+S_{4}(t)\right]
+B_{4}(t)\left[  P_{2}(t)\bar{X}(t)-P_{3}(t)h(t)+\varphi_{2}(t)\right]  \}\\
-\gamma_{2}(t)\\
=-\{A_{3}(t)\bar{X}(t)+B_{3}(t)\left[  P_{2}(t)\bar{X}(t)-P_{3}(t)h(t)+\varphi
_{2}(t)\right] \\
\ \ +C_{3}(t)\left[  L_{10}(t)\bar{X}(t)+L_{11}(t)h(t)+S_{5}(t)\right] \\
\ \ +D_{3}(t)\left[  L_{6}(t)\bar{X}(t)+L_{7}(t)h(t)+S_{3}(t)\right]  \}.
\end{array}
\label{eq-hmn2}%
\end{equation}
$P_{2}(\cdot)$, $P_{3}(\cdot)$ and $\varphi_{2}(\cdot)$ should be solutions of%
\begin{equation}
\left\{
\begin{array}
[c]{l}%
dP_{2}(t)\\
=-\{P_{2}(t)A_{1}(t)+P_{2}(t)B_{1}(t)P_{2}(t)+P_{2}(t)C_{1}(t)L_{10}(t)\\
\ \ +P_{2}(t)D_{1}(t)L_{6}(t)-P_{3}(t)B_{1}(t)^{\intercal}P_{1}(t)-P_{3}%
(t)B_{2}(t)^{\intercal}L_{8}(t)\\
\ \ -P_{3}(t)B_{4}(t)P_{2}(t)+A_{3}(t)+B_{3}(t)P_{2}(t)+C_{3}(t)L_{10}%
(t)+D_{3}(t)L_{6}(t)\}dt,\\
P_{2}(T)=F,
\end{array}
\right.  \label{eq-p3}%
\end{equation}%
\begin{equation}
\left\{
\begin{array}
[c]{l}%
dP_{3}(t)\\
=-\{P_{2}(t)B_{1}(t)P_{3}(t)-P_{2}(t)C_{1}(t)L_{11}(t)-P_{2}(t)D_{1}%
(t)L_{7}(t)+P_{3}(t)B_{3}(t)^{\intercal}\\
\ \ +P_{3}(t)B_{1}(t)^{\intercal}P_{2}(t)^{\intercal}+P_{3}(t)B_{2}%
(t)^{\intercal}L_{9}(t)-P_{3}(t)B_{4}(t)P_{3}(t)\\
\ \ +B_{3}(t)P_{3}(t)-C_{3}(t)L_{11}(t)-D_{3}(t)L_{7}(t)\}dt,\\
P_{3}(T)=0,
\end{array}
\right.  \label{eq-p4}%
\end{equation}%
\begin{equation}
\left\{
\begin{array}
[c]{rl}%
d\varphi_{2}(t)= & -\left\{  P_{2}(t)\left[  B_{1}(t)\varphi_{2}%
(t)+C_{1}(t)S_{5}(t)+D_{1}(t)S_{3}(t)\right]  \right. \\
& -P_{3}(t)\left[  B_{1}(t)^{\intercal}\varphi_{1}(t)+B_{2}(t)^{\intercal
}S_{4}(t)+B_{4}(t)\varphi_{2}(t)\right] \\
& \left.  +B_{3}(t)\varphi_{2}(t)+C_{3}(t)S_{5}(t)+D_{3}(t)S_{3}(t)\right\}
dt+v_{2}(t)dB(t),\\
\varphi_{2}(T)= & \xi
\end{array}
\right.  \label{eq-phi2}%
\end{equation}
respectively. It can be verified that the equation (\ref{eq-p1}),
(\ref{eq-p4}) are symmetric and (\ref{eq-p3}) is indeed the transpose of
(\ref{eq-p2}).

\begin{remark}
\label{re-d4}If $D_{4}>0$ and $C_{4}\geq0$, then the following relations
hold:
\[%
\begin{array}
[c]{rl}%
Q_{1}(t)=P_{1}(t), & Q_{2}(t)=P_{2}(t)^{\intercal},\\
Q_{3}(t)=P_{2}(t), & Q_{4}(t)=P_{3}(t).
\end{array}
\]

\end{remark}

\section{Non-Riccati-type equations}

In this section, we study the existence and uniqueness results for solutions
to non-Riccati-type equations (\ref{eq-p1}), (\ref{eq-p3}) and (\ref{eq-p4}).

\subsection{Auxiliary Riccati-type equations}

Our aim of this subsection is to reveal the origin of the following auxiliary
Riccati equation (\ref{eq-rec}) and (\ref{eq-phi-til}). Hence we will present
the material in this subsection in an informal way although they can be
verified rigorously.

We first introduce an auxiliary stochastic LQ problem which leads to a
Riccati-type equation for $\tilde{P}(\cdot)$. Then the relations between
$P(\cdot)$ and $\tilde{P}(\cdot)$ are deduced and with the help of good
properties of $\tilde{P}(\cdot)$, we will obtain the existence results for the
solutions of (\ref{eq-m-equal})-(\ref{eq-phi2}).

Inspired by \cite{Ma-Y, Koh-Z, Lim-Z}, for the FBLQ problem (\ref{state-lq}%
)-(\ref{cost-lq}), we regard the BSDE as a controlled forward SDE and the term
$Z(\cdot)$ as a control. Thus, it becomes a forward LQ problem. Set $\tilde
{X}(t)=(X(t)^{\intercal},Y(t)^{\intercal})^{\intercal}$ and $\tilde
{u}(t)=(u(t)^{\intercal},Z(t)^{\intercal})^{\intercal}$. The state equation
becomes
\begin{equation}%
\begin{array}
[c]{rl}%
d\tilde{X}(t)= & \left[  \tilde{A}(t)\tilde{X}(t)+\tilde{B}(t)\tilde
{u}(t)\right]  dt+\left[  \tilde{C}(t)\tilde{X}(t)+\tilde{D}(t)\tilde
{u}(t)\right]  dB(t),
\end{array}
\label{eq-pro-x}%
\end{equation}
and the cost functional becomes
\begin{equation}
J(u(\cdot),Z(\cdot))=\frac{1}{2}\mathbb{E}\left[  \displaystyle\int_{0}%
^{T}\left(  \tilde{X}(t)^{\intercal}\tilde{Q}(t)\tilde{X}(t)+\tilde
{u}(t)^{\intercal}\tilde{R}(t)\tilde{u}(t)\right)  dt+\left\langle
GX(T),X(T)\right\rangle \right]  , \label{eq-pro-J}%
\end{equation}
where%
\[%
\begin{array}
[c]{ll}%
\tilde{A}(t)=\left(
\begin{array}
[c]{ccc}%
A_{1}(t) &  & B_{1}(t)\\
-A_{3}(t) &  & -B_{3}(t)
\end{array}
\right)  , & \tilde{B}(t)=\left(
\begin{array}
[c]{ccc}%
D_{1}(t) &  & C_{1}(t)\\
-D_{3}(t) &  & -C_{3}(t)
\end{array}
\right)  ,
\end{array}
\]%
\[%
\begin{array}
[c]{ll}%
\tilde{C}(t)=\left(
\begin{array}
[c]{ccc}%
A_{2}(t) &  & B_{2}(t)\\
0 &  & 0
\end{array}
\right)  , & \tilde{D}(t)=\left(
\begin{array}
[c]{ccc}%
D_{2}(t) &  & C_{2}(t)\\
0 &  & I_{m}%
\end{array}
\right)  ,
\end{array}
\]%
\[%
\begin{array}
[c]{ll}%
\tilde{Q}(t)=\left(
\begin{array}
[c]{ccc}%
A_{4}(t) &  & 0\\
0 &  & B_{4}(t)
\end{array}
\right)  , & \tilde{R}(t)=\left(
\begin{array}
[c]{ccc}%
D_{4}(t) &  & 0\\
0 &  & C_{4}(t)
\end{array}
\right)  .
\end{array}
\]

Now we solve the above LQ problem by the completion-of-squares technique
similar as in Theorem 3.1 in \cite{Chen-Z}. Suppose that $(\tilde{\varphi
}(\cdot),\tilde{v}(\cdot))$ satisfies the following BSDE%
\begin{equation}%
\begin{array}
[c]{rl}%
d\tilde{\varphi}(t)= & -\tilde{\gamma}(t)dt+\tilde{v}(t)dB(t),
\end{array}
\label{eq-phi-til}%
\end{equation}
where $\tilde{\gamma}(\cdot)$ will be determined later. For a function
$\tilde{P}(\cdot)$ to be determined, applying It\^{o}'s formula to
\[
\left(  \tilde{X}(t)-\tilde{\varphi}(t)\right)  ^{\intercal}\tilde
{P}(t)\left(  \tilde{X}(t)-\tilde{\varphi}(t)\right)  +\int_{0}^{t}\left(
\tilde{X}(s)^{\intercal}\tilde{Q}(s)\tilde{X}(s)+\tilde{u}(s)^{\intercal
}\tilde{R}(s)\tilde{u}(s)\right)  ds,
\]
we have
\[%
\begin{array}
[c]{l}%
d\left[  \left(  \tilde{X}(t)-\tilde{\varphi}(t)\right)  ^{\intercal}\tilde
{P}(t)\left(  \tilde{X}(t)-\tilde{\varphi}(t)\right)  +\displaystyle\int%
_{0}^{t}\left(  \tilde{X}(s)^{\intercal}\tilde{Q}(s)\tilde{X}(s)+\tilde
{u}(s)^{\intercal}\tilde{R}(s)\tilde{u}(s)\right)  ds\right] \\
=\left\{  \left[  \tilde{u}(t)+M_{1}(t)^{-1}\left(  M_{2}(t)\tilde{X}%
(t)-M_{3}(t)\tilde{\varphi}(t)-M_{4}(t)\tilde{v}(t)\right)  \right]
^{\intercal}M_{1}(t)\right. \\
\text{ \ \ }\cdot\left[  \tilde{u}(t)+M_{1}(t)^{-1}\left(  M_{2}(t)\tilde
{X}(t)-M_{3}(t)\tilde{\varphi}(t)-M_{4}(t)\tilde{v}(t)\right)  \right] \\
\text{ \ \ }+\tilde{X}(t)^{\intercal}\left[  \dot{\tilde{P}}(t)+\tilde
{P}(t)\tilde{A}(t)+\tilde{A}(t)^{\intercal}\tilde{P}(t)+\tilde{C}%
(t)^{\intercal}\tilde{P}(t)\tilde{C}(t)+\tilde{Q}(t)\right. \\
\text{\ \ }\left.  -M_{2}(t)^{\intercal}M_{1}(t)^{-1}M_{2}(t)\right]
\tilde{X}(t)\\
\text{ \ \ }+\tilde{X}(t)^{\intercal}\left[  M_{2}(t)^{\intercal}M_{1}%
(t)^{-1}\left(  M_{3}(t)\tilde{\varphi}(t)-M_{4}(t)\tilde{v}(t)\right)
-\tilde{A}(t)^{\intercal}\tilde{P}(t)\tilde{\varphi}(t)-\dot{\tilde{P}%
}(t)\tilde{\varphi}(t)\right. \\
\text{ \ \ }\left.  -\tilde{P}(t)\tilde{\gamma}(t)-\tilde{C}(t)^{\intercal
}\tilde{P}(t)\tilde{v}(t)\right] \\
\text{ \ \ }+\left[  M_{2}(t)^{\intercal}M_{1}(t)^{-1}\left(  M_{3}%
(t)\tilde{\varphi}(t)-M_{4}(t)\tilde{v}(t)\right)  -\tilde{A}(t)^{\intercal
}\tilde{P}(t)\tilde{\varphi}(t)-\dot{\tilde{P}}(t)\tilde{\varphi}(t)\right. \\
\text{ \ \ }\left.  -\tilde{P}(t)\tilde{\gamma}(t)-\tilde{C}(t)^{\intercal
}\tilde{P}(t)\tilde{v}(t)\right]  ^{\intercal}\tilde{X}(t)\\
\text{\ \ }\left.  +M_{5}(t)\right\}  dt+\{...\}dB(t),
\end{array}
\]
where
\begin{equation}%
\begin{array}
[c]{l}%
M_{1}(t)=\tilde{R}(t)+\tilde{D}(t)^{\intercal}\tilde{P}(t)\tilde{D}(t),\text{
}M_{2}(t)=\tilde{B}(t)^{\intercal}\tilde{P}(t)+\tilde{D}(t)^{\intercal}%
\tilde{P}(t)\tilde{C}(t),\\
M_{3}(t)=\tilde{B}(t)^{\intercal}\tilde{P}(t),\text{ }M_{4}(t)=\tilde
{D}(t)^{\intercal}\tilde{P}(t),\\
M_{5}(t)=-\tilde{\gamma}(t)^{\intercal}\tilde{P}(t)\tilde{\varphi}%
(t)-\tilde{\varphi}(t)^{\intercal}\tilde{P}(t)\tilde{\gamma}(t)+\tilde
{\varphi}(t)^{\intercal}\dot{\tilde{P}}(t)\tilde{\varphi}(t)+\tilde
{v}(t)^{\intercal}\tilde{P}(t)\tilde{v}(t)\\
\text{ \ \ \ \ \ \ \ \ \ }-\left(  M_{3}(t)\tilde{\varphi}(t)+M_{4}%
(t)\tilde{v}(t)\right)  ^{\intercal}M_{1}(t)^{-1}\left(  M_{3}(t)\tilde
{\varphi}(t)+M_{4}(t)\tilde{v}(t)\right)  .
\end{array}
\label{def-m1}%
\end{equation}
{Thus, we can obtain the form of the Riccati equation and the optimal
control\ $(\bar{u}(\cdot),\bar{Z}(\cdot))$ as following:}
\begin{equation}
\left\{
\begin{array}
[c]{l}%
\dot{\tilde{P}}(t)+\tilde{P}(t)\tilde{A}(t)+\tilde{A}(t)^{\intercal}\tilde
{P}(t)+\tilde{C}(t)^{\intercal}\tilde{P}(t)\tilde{C}(t)+\tilde{Q}(t)\\
\ \ -M_{2}(t)^{\intercal}M_{1}(t)^{-1}M_{2}(t)=0,\\
\tilde{R}(t)+\tilde{D}(t)^{\intercal}\tilde{P}(t)\tilde{D}(t)>0,
\end{array}
\right.  \label{eq-rec}%
\end{equation}%
\begin{equation}%
\begin{array}
[c]{rl}%
(\bar{u}(t),\bar{Z}(t))^{\intercal}= & -M_{1}(t)^{-1}\left[  M_{2}(t)(\bar
{X}(t),\bar{Y}(t))^{\intercal}-M_{3}(t)\tilde{\varphi}(t)-M_{4}(t)\tilde
{v}(t)\right]  ,
\end{array}
\label{eq-con-th}%
\end{equation}
and%
\[%
\begin{array}
[c]{rl}%
\tilde{\gamma}(t)= & \tilde{P}(t)^{-1}\left[  M_{2}(t)^{\intercal}%
M_{1}(t)^{-1}\left(  M_{3}(t)\tilde{\varphi}(t)-M_{4}(t)\tilde{v}(t)\right)
-\tilde{A}(t)^{\intercal}\tilde{P}(t)\tilde{\varphi}(t)\right. \\
& \left.  -\dot{\tilde{P}}(t)\tilde{\varphi}(t)-\tilde{C}(t)^{\intercal}%
\tilde{P}(t)\tilde{v}(t)\right]  .
\end{array}
\]

Set%
\[
\tilde{P}(t)=\left(
\begin{array}
[c]{ccc}%
\tilde{P}_{1}(t) &  & \tilde{P}_{2}(t)^{\intercal}\\
\tilde{P}_{2}(t) &  & \tilde{P}_{3}(t)
\end{array}
\right)  ,\text{ }\tilde{\varphi}(t)=\binom{\tilde{\varphi}_{1}(t)}%
{\tilde{\varphi}_{2}(t)},\text{ }\tilde{v}(t)=\binom{\tilde{v}_{1}(t)}%
{\tilde{v}_{2}(t)}.
\]
By the relationship between the adjoint process and the state process for
stochastic LQ problems, we have%
\begin{equation}
\left(  m(t)^{\intercal},-h(t)^{\intercal}\right)  ^{\intercal}=\tilde
{P}(t)\left(  \tilde{X}(t)-\tilde{\varphi}(t)\right)  \text{.}
\label{relation-square}%
\end{equation}
Comparing (\ref{relation-square}) with (\ref{eq-Y-m}), we obtain the relations
between $\tilde{P}(\cdot)$, $\left(  \tilde{\varphi}(\cdot),\tilde{v}%
(\cdot)\right)  $ and $P(\cdot)$, $\left(  \varphi(\cdot),v(\cdot)\right)  $
as following:
\[%
\begin{array}
[c]{rl}%
P_{1}(t)= & \tilde{P}_{1}(t)-\tilde{P}_{2}(t)^{\intercal}\tilde{P}_{3}%
(t)^{-1}\tilde{P}_{2}(t),\\
P_{2}(t)= & -\tilde{P}_{3}(t)^{-1}\tilde{P}_{2}(t),\\
P_{3}(t)= & \tilde{P}_{3}(t)^{-1},
\end{array}
\]%
\[%
\begin{array}
[c]{rl}%
\varphi_{1}(t)= & -P_{1}(t)\tilde{\varphi}_{1}(t),\text{ }v_{1}(t)=-P_{1}%
(t)\tilde{v}_{1}(t),\\
\varphi_{2}(t)= & \tilde{\varphi}_{2}(t)-P_{2}(t)\tilde{\varphi}_{1}(t),\text{
}v_{2}(t)=\tilde{v}_{2}(t)-P_{2}(t)\tilde{v}_{1}(t),
\end{array}
\]
or the equivalent form
\[%
\begin{array}
[c]{rl}%
\tilde{P}_{1}(t)= & P_{1}(t)+P_{2}(t)^{\intercal}P_{3}(t)^{-1}P_{2}(t),\\
\tilde{P}_{2}(t)= & -P_{3}(t)^{-1}P_{2}(t),\\
\tilde{P}_{3}(t)= & P_{3}(t)^{-1},
\end{array}
\]%
\[%
\begin{array}
[c]{rl}%
\tilde{\varphi}_{1}(t)= & -P_{1}(t)^{-1}\varphi_{1}(t),\text{ }\tilde{v}%
_{1}(t)=-P_{1}(t)^{-1}v_{1}(t),\\
\tilde{\varphi}_{2}(t)= & \varphi_{2}(t)-P_{2}(t)P_{1}(t)^{-1}\varphi
_{1}(t),\text{ }\tilde{v}_{2}(t)=v_{2}(t)-P_{2}(t)P_{1}(t)^{-1}v_{1}(t).
\end{array}
\]

Note that $P_{3}(T)=0$ which makes $P_{3}(T)^{-1}$ meaningless. So we need to
modify the terminal conditions of $\tilde{P}(\cdot)$, $\tilde{\varphi}(\cdot)$
and $P(\cdot)$, $\varphi(\cdot)$. For $i=1,2,...$, consider the solutions
\[
_{i}P(t)=\left(
\begin{array}
[c]{ccc}%
P_{1,i}(t) &  & P_{2,i}(t)^{\intercal}\\
P_{2,i}(t) &  & P_{3,i}(t)
\end{array}
\right)  ,\text{ }_{i}\varphi(t)=\binom{\varphi_{1,i}(t)}{\varphi_{2,i}%
(t)},\text{ }_{i}v(t)=\binom{v_{1,i}(t)}{v_{2,i}(t)}%
\]
to equations (\ref{eq-p1}), (\ref{eq-p3}), (\ref{eq-p4}), (\ref{eq-phi1}),
(\ref{eq-phi2}) with the terminal conditions
\begin{equation}
_{i}P(T)=\left(
\begin{array}
[c]{ccc}%
G &  & F^{\intercal}\\
F &  & \frac{1}{i}I_{m}%
\end{array}
\right)  ,\text{ }_{i}\varphi(T)=\binom{0}{\xi}. \label{eq-ter-i}%
\end{equation}
Correspondingly, we consider the Riccati equation (\ref{eq-rec}) and
(\ref{eq-phi-til}) for%
\[
_{i}\tilde{P}(t)=\left(
\begin{array}
[c]{ccc}%
\tilde{P}_{1,i}(t) &  & \tilde{P}_{2,i}(t)^{\intercal}\\
\tilde{P}_{2,i}(t) &  & \tilde{P}_{3,i}(t)
\end{array}
\right)  ,\text{ }_{i}\tilde{\varphi}(t)=\binom{\tilde{\varphi}_{1,i}%
(t)}{\tilde{\varphi}_{2,i}(t)},\text{ }_{i}\tilde{v}(t)=\binom{\tilde{v}%
_{1,i}(t)}{\tilde{v}_{2,i}(t)}%
\]
with terminal conditions
\begin{equation}
_{i}\tilde{P}(T)=\left(
\begin{array}
[c]{ccc}%
G+iF^{\intercal}F &  & -iF^{\intercal}\\
-iF &  & iI_{m}%
\end{array}
\right)  ,\text{ }_{i}\tilde{\varphi}(T)=\binom{0}{\xi}. \label{eq-termi-p-t}%
\end{equation}

\begin{remark}
In fact, (\ref{eq-rec}) and (\ref{eq-phi-til}) with terminal conditions
(\ref{eq-termi-p-t}) correspond to the following stochastic control problem:
the state equation is (\ref{eq-pro-x}) and the cost functional is
\begin{equation}%
\begin{array}
[c]{rl}%
J_{i}(u(\cdot),Z(\cdot))= & \frac{1}{2}\mathbb{E}\left[  \int_{0}^{T}\left(
\tilde{X}(t)^{\intercal}\tilde{Q}(t)\tilde{X}(t)+\tilde{u}(t)^{\intercal
}\tilde{R}(t)\tilde{u}(t)\right)  dt+\left\langle GX(T),X(T)\right\rangle
\right] \\
& +\frac{1}{2}i\mathbb{E}\left[  \left(  Y(T)-FX(T)-\xi\right)  ^{\intercal
}\left(  Y(T)-FX(T)-\xi\right)  \right]  .
\end{array}
\label{obj-pen}%
\end{equation}

\end{remark}

Theorem \ref{th-p-ex} justifies the above heuristic derivation.

\begin{assumption}
\label{assm-p-i-ex}There exist a natural number $i_{0}$ such that for $i\geq
i_{0}$, (\ref{eq-rec}) has a positive definite solution $_{i}\tilde{P}(\cdot)$
which satisfies the terminal condition (\ref{eq-termi-p-t}).
\end{assumption}

\begin{remark}
Under the assumption $D_{4}>0$ and $C_{4}\geq0$, it is easy to check that
$\tilde{R}+\tilde{D}^{\intercal}\tilde{D}>0$. Then, by Theorem 4.1 in
\cite{Chen-Z}, Assumption \ref{assm-p-i-ex} holds for $i_{0}=1$.
\end{remark}

Set%
\begin{align*}
L_{1,i}(t)  &  =I_{m}-P_{2,i}(t)C_{2}(t)+P_{3,i}(t)C_{4}(t),\\
L_{2,i}(t)  &  =I_{n}-P_{2,i}(t)^{\intercal}C_{2}(t)^{\intercal}+\left(
P_{1,i}(t)C_{2}(t)+P_{2,i}(t)^{\intercal}C_{4}(t)\right)  L_{1}(t)^{-1}%
P_{3,i}(t)C_{2}(t)^{\intercal}.
\end{align*}

\begin{theorem}
\label{th-p-ex}Suppose that Assumptions \ref{assum-bound}, \ref{assum-pos} and
\ref{assm-p-i-ex} hold. For $i\geq i_{0}$, define
\begin{equation}%
\begin{array}
[c]{rl}%
P_{1,i}(t)= & \tilde{P}_{1,i}(t)-\tilde{P}_{2,i}(t)^{\intercal}\tilde{P}%
_{3,i}(t)^{-1}\tilde{P}_{2,i}(t),\\
P_{2,i}(t)= & -\tilde{P}_{3,i}(t)^{-1}\tilde{P}_{2,i}(t),\\
P_{3,i}(t)= & \tilde{P}_{3,i}(t)^{-1},
\end{array}
\label{eq-re}%
\end{equation}%
\[%
\begin{array}
[c]{rl}%
\varphi_{1,i}(t)= & -P_{1,i}(t)\tilde{\varphi}_{1,i}(t),\text{ }%
v_{1,i}(t)=-P_{1,i}(t)\tilde{v}_{1,i}(t),\\
\varphi_{2,i}(t)= & \tilde{\varphi}_{2,i}(t)-P_{2,i}(t)\tilde{\varphi}%
_{1,i}(t),\text{ }v_{2,i}(t)=\tilde{v}_{2,i}(t)-P_{2,i}(t)\tilde{v}_{1,i}(t),
\end{array}
\]
where $_{i}\tilde{P}(\cdot)$ and $(_{i}\tilde{\varphi}(\cdot),$ $_{i}\tilde
{v}(\cdot))$ are solutions to (\ref{eq-rec}) and (\ref{eq-phi-til}). Suppose
that $L_{1,i}(\cdot)^{-1}$ and $L_{2,i}(\cdot)^{-1}\,$exist. Then the above
defined $\left(  P_{1,i}(\cdot),P_{2,i}(\cdot),P_{3,i}(\cdot)\right)  $ solves
(\ref{eq-p1}), (\ref{eq-p3}), (\ref{eq-p4}) and $\left(  _{i}\varphi
(\cdot),\text{ }_{i}v(\cdot)\right)  $ solves (\ref{eq-phi1}), (\ref{eq-phi2})
with (\ref{eq-ter-i}) for each $i\geq i_{0}$.
\end{theorem}

We put the proof in Appendix 7.1.

\begin{lemma}
\label{le-con-equ}Under the same assumptions as Theorem \ref{th-p-ex}, for
$i\geq i_{0}$, we have%
\[
M_{1,i}(t)^{-1}M_{2,i}(t)=-\left(
\begin{array}
[c]{ccc}%
L_{6,i}(t)+L_{7,i}(t)P_{3,i}(t)^{-1}P_{2,i}(t) &  & -L_{7,i}(t)P_{3,i}%
(t)^{-1}\\
L_{10,i}(t)+L_{11,i}(t)P_{3,i}(t)^{-1}P_{2,i}(t) &  & -L_{11,i}(t)P_{3,i}%
(t)^{-1}%
\end{array}
\right)  ,
\]%
\[
M_{1,i}(t)^{-1}\left[  M_{3,i}(t)_{i}\tilde{\varphi}(t)+M_{4,i}(t)_{i}%
\tilde{v}(t)\right]  =\left(
\begin{array}
[c]{c}%
L_{7,i}(t)P_{3,i}(t)^{-1}\varphi_{2,i}(t)+S_{3,i}(t)\\
L_{11,i}(t)P_{3,i}(t)^{-1}\varphi_{2,i}(t)+S_{5,i}(t)
\end{array}
\right)  .
\]

\end{lemma}

The proof is in Appendix 7.2. This lemma will be used in the proof of Theorem
\ref{th-infi}.

\begin{remark}
\label{new-rem1}If $C_{4}>0$ and $D_{4}>0$, then it can be verified that
Assumption \ref{assm-p-i-ex} holds. If $C_{2}=0$ and $D_{4}>0$, then
$L_{1,i}(\cdot)^{-1}$ and $L_{2,i}(\cdot)^{-1}\,$in Theorem \ref{th-p-ex} exist.
\end{remark}

\subsection{Existence and uniqueness results}

In this subsection, we study the solvability of (\ref{eq-p1}), (\ref{eq-p3}),
(\ref{eq-p4}) by Theorem \ref{th-p-ex}. {\color{blue} }

\begin{lemma}
\label{le-com-rec}Suppose $\tilde{P}_{1}(\cdot)$ and $\tilde{P}_{2}(\cdot)$
are solutions to Riccati equation (\ref{eq-rec}) with terminal conditions
$\tilde{P}_{1}(T)\geq\tilde{P}_{2}(T)$, then $\tilde{P}_{1}(t)\geq\tilde
{P}_{1}(t)$ for $t\in\left[  0,T\right]  $.
\end{lemma}

\begin{proof}
By Theorem 6.1 in \cite{Yong-Zhou}, the value function of the corresponding LQ
problem is $x^{\intercal}\tilde{P}_{1}(t)x$ (resp. $x^{\intercal}\tilde{P}%
_{2}(t)x$) for all $\left(  t,x\right)  \in\lbrack0,T]\times\mathbb{R}^{n}$.
The proof can be obtained from $\tilde{P}_{1}(T)\geq\tilde{P}_{2}(T)$.
\end{proof}

\begin{theorem}
\label{th-conver} Suppose that the same assumptions as Theorem \ref{th-p-ex}
hold and $(\tilde{R}(\cdot)+\tilde{D}(\cdot)^{\intercal}$ $_{i}\tilde{P}%
(\cdot)\tilde{D}(\cdot))^{-1}$ is bounded for each $i\geq i_{0}$. Then
$P_{3,i}(t)\geq P_{3,i+1}(t)\geq0$, and $P_{1,i+1}(t)\geq P_{1,i}(t)\geq0$ for
$i\geq i_{0}$. Moreover, suppose that $P_{1,i}(\cdot)$ has upper bound and
$\left\vert P_{2,i}(\cdot)\right\vert $, $L_{1,i}(\cdot)^{-1}$, $L_{2,i}%
(\cdot)^{-1}\,$and $L_{5,i}(\cdot)^{-1}$ are uniformly bounded for each $i\geq
i_{0}$. Then (\ref{eq-p1}), (\ref{eq-p3}), (\ref{eq-p4}) have a unique
solution $\left(  P_{1}(\cdot),P_{2}(\cdot),P_{3}(\cdot)\right)  $.
\end{theorem}

\begin{proof}
It can be verified that
\[%
\begin{array}
[c]{rl}%
_{i+1}\tilde{P}(T) & =\left(
\begin{array}
[c]{ccc}%
G+(i+1)F^{\intercal}F &  & -(i+1)F^{\intercal}\\
-(i+1)F &  & (i+1)I_{m}%
\end{array}
\right) \\
& \geq\text{ }_{i}\tilde{P}(T)=\left(
\begin{array}
[c]{ccc}%
G+iF^{\intercal}F &  & -iF^{\intercal}\\
-iF &  & iI_{m}%
\end{array}
\right)  .
\end{array}
\]
By Lemma \ref{le-com-rec}, we have $_{i+1}\tilde{P}(t)\geq$ $_{i}\tilde{P}(t)$
which yields that $\tilde{P}_{3,i+1}(t)\geq\tilde{P}_{3,i}(t)$ and
$_{i+1}\tilde{P}(t)^{-1}\leq$ $_{i}\tilde{P}(t)^{-1}$. Moreover, note that
\[
\left(
\begin{array}
[c]{ccc}%
\tilde{P}_{1,i}(t) &  & \tilde{P}_{2,i}(t)^{\intercal}\\
\tilde{P}_{2,i}(t) &  & \tilde{P}_{3,i}(t)
\end{array}
\right)  ^{-1}=\left(
\begin{array}
[c]{cc}%
\left(  \tilde{P}_{1,i}(t)-\tilde{P}_{2,i}(t)^{\intercal}\tilde{P}%
_{3,i}(t)^{-1}\tilde{P}_{2,i}(t)\right)  ^{-1} & ...\\
... & ...
\end{array}
\right)  .
\]
By the relationship (\ref{eq-re}), we obtain that $P_{3,i+1}(t)\leq
P_{3,i}(t)$ and $P_{1,i+1}(t)\geq P_{1,i}(t)$.

Thus, $\{P_{3,i}(t)\}_{i\geq i_{0}}$ (resp. $\{P_{1,i}(t)\}_{i\geq i_{0}}$) is
a bounded deceasing (resp. increasing) sequence in $C\left(  \left[
0,T\right]  ;\mathbb{S}_{+}^{m}\right)  $ (resp. $C\left(  \left[  0,T\right]
;\mathbb{S}_{+}^{n}\right)  $ ) and therefore has a limit. The convergence of
$\{P_{2,i}(t)\}_{i\geq i_{0}}$ can be obtained by the following Proposition
\ref{converge-P2i}. Denote by $\left(  P_{1}(\cdot),P_{2}(\cdot),P_{3}%
(\cdot)\right)  $ the limit of $\{\left(  P_{1,i}(\cdot),P_{2,i}%
(\cdot),P_{3,i}(\cdot)\right)  \}_{i\geq i_{0}}$. By the bounded convergence
theorem, one can obtain that $\left(  P_{1}(\cdot),P_{2}(\cdot),P_{3}%
(\cdot)\right)  $ is the solution to (\ref{eq-p1}), (\ref{eq-p3}),
(\ref{eq-p4}). This completes the proof.
\end{proof}

\begin{corollary}
\label{Cor-equation Q}Suppose that Assumptions \ref{assum-bound},
\ref{assum-pos} and \ref{assum-pos-d4} hold. Moreover, suppose that
$P_{1,i}(\cdot)$ has upper bound and $\left\vert P_{2,i}(\cdot)\right\vert $,
$L_{1,i}(\cdot)^{-1}$ and $L_{2,i}(\cdot)^{-1}$ are uniformly bounded for each
$i\geq1$. Then the equation {(\ref{eq-p-til}) has a unique solution.}
\end{corollary}

\begin{proof}
By Remark \ref{new-rem1}, Assumption \ref{assm-p-i-ex} holds. Since $D_{4}>0$,
it is easy to verify that $L_{5,i}(\cdot)^{-1}\leq D_{4}^{-1}$ for each
$i\geq1$. By {Remark \ref{re-d4} and Theorem \ref{th-conver}, }then the
equation {(\ref{eq-p-til}) has a unique solution.}
\end{proof}

\begin{proposition}
\label{converge-P2i}Suppose that all assumptions in Theorem \ref{th-conver}
hold. Then for $i\geq i_{0}$, we have%
\[%
\begin{array}
[c]{c}%
\left\vert P_{1}(t)-P_{1,i}(t)\right\vert +\left\vert P_{2}(t)-P_{2,i}%
(t)\right\vert +\left\vert P_{3}(t)-P_{3,i}(t)\right\vert \leq Ci^{-2}.
\end{array}
\]
where $\left(  P_{1}(\cdot),P_{2}(\cdot),P_{3}(\cdot)\right)  $ is the limit
of $\{P_{1,i}(\cdot)$, $P_{2,i}(\cdot),$ $P_{3,i}(\cdot)\}_{i\geq i_{0}}$ and
$C$ is a constant independent of $i$.
\end{proposition}

\begin{proof}
Set $\Delta_{1,i}(t)=P_{1,i+1}(t)-P_{1,i}(t)$, $\Delta_{2,i}(t)=P_{2,i+1}%
(t)-P_{2,i}(t)$, $\Delta_{3,i}(t)=P_{3,i+1}(t)-P_{3,i}(t)$. By (\ref{eq-p1}),
(\ref{eq-p3}), (\ref{eq-p4}) and the boundedness assumptions, we have%
\begin{equation}%
\begin{array}
[c]{l}%
\left\vert \Delta_{1,i}(t)\right\vert +\left\vert \Delta_{2,i}(t)\right\vert
+\left\vert \Delta_{3,i}(t)\right\vert \\
\leq\frac{m}{i(i+1)}+C^{\prime}\int_{t}^{T}\left(  \left\vert \Delta
_{1,i}(s)\right\vert +\left\vert \Delta_{2,i}(s)\right\vert +\left\vert
\Delta_{3,i}(s)\right\vert \right)  ds,
\end{array}
\label{delta-123}%
\end{equation}
where $C^{\prime}$ is a constant independent of $i$. Then, by Gronwall's
inequality we have
\[
\left\vert \Delta_{1,i}(t)\right\vert +\left\vert \Delta_{2,i}(t)\right\vert
+\left\vert \Delta_{3,i}(t)\right\vert \leq Ci^{-2},
\]
where $C=me^{C^{\prime}T}.$
\end{proof}

\begin{remark}
Since $\{P_{1,i}(t)\}_{i\geq i_{0}}$ is increasing and $\{P_{3,i}(t)\}_{i\geq
i_{0}}$ is decreasing, the sequence $\{_{i}P(t)\}_{i\geq i_{0}}$ is not
monotonic which is different from the indefinite stochastic LQ problem in
\cite{Chen-LZ}. With the help of the solutions $\{_{i}\tilde{P}(t)\}_{i\geq
i_{0}}$ to the auxiliary Riccati equations, we study the components of
$_{i}P(t)$ and prove the existence of the solutions to (\ref{eq-p1}),
(\ref{eq-p3}), (\ref{eq-p4}).
\end{remark}

Now we give an example to show that there exists a unique solution to
(\ref{eq-p1}), (\ref{eq-p3}), (\ref{eq-p4}).

\begin{example}
Consider a special case of problem (\ref{state-lq})-(\ref{cost-lq}) in which
the controlled system is governed by a partially coupled FBSDE. Suppose that
$n=m=1$, $B_{1}(t)=C_{1}(t)=B_{2}(t)=C_{2}(t)=0$, $D_{2}(t)^{2}\geq\delta>0$,
$C_{4}(t)\geq\delta>0$ and $D_{4}(t)\geq\delta>0$. Due to $C_{2}\equiv0$,
$L_{1,i}(t)=I_{m}+P_{3,i}(t)C_{4}(t)$ and $L_{2,i}(t)\equiv I_{n}$ are
invertible and bounded. It is easy to verify the other assumptions in Theorem
\ref{th-conver} except that $P_{1,i}(\cdot)$ has upper bound and $\left\vert
P_{2,i}(\cdot)\right\vert $ is bounded. Note that $\left(  P_{1,i}%
(\cdot),P_{2,i}(\cdot),P_{3,i}(\cdot)\right)  $ satisfies the following
equations:%
\[%
\begin{array}
[c]{rl}%
dP_{1,i}(t)= & -\left\{  \left(  2A_{1}(t)+A_{2}(t)^{2}\right)  P_{1,i}%
(t)+A_{4}(t)\right. \\
& +\left[  B_{4}(t)+A_{2}(t)^{2}C_{4}(t)(1+P_{3,i}(t)C_{4}(t))^{-1}\right]
P_{2,i}(t)^{2}\\
& -\frac{1}{\Delta}\left[  \left(  D_{1}(t)+A_{2}(t)D_{2}(t)\right)
P_{1,i}(t)\right. \\
& \left.  \left.  +A_{2}(t)D_{2}(t)C_{4}(t)P_{2,i}(t)^{2}(1+P_{3,i}%
(t)C_{4}(t))^{-1}\right]  ^{2}\right\}  dt,
\end{array}
\]%
\[%
\begin{array}
[c]{rl}%
dP_{2,i}(t)= & -\left\{  P_{2,i}(t)\left[  A_{1}(t)-B_{4}(t)P_{3,i}%
(t)+B_{3}(t)+C_{3}(t)A_{2}(t)(1+P_{3,i}(t)C_{4}(t))^{-1}\right]  \right. \\
& +A_{3}(t)-\frac{1}{\Delta}\left[  \left(  D_{1}(t)+D_{2}(t)C_{3}%
(t)(1+P_{3}(t)C_{4}(t))^{-1}\right)  \right. \\
& \cdot D_{2}(t)A_{2}(t)C_{4}(t)(1+P_{3}(t)C_{4}(t))^{-1}P_{2,i}(t)^{3}\\
& +D_{2}(t)D_{3}(t)A_{2}(t)C_{4}(t)(1+P_{3}(t)C_{4}(t))^{-1}P_{2,i}%
(t)^{2}+\left(  D_{1}(t)+D_{2}(t)A_{2}(t)\right)  P_{1,i}(t)\\
& \cdot\left(  D_{1}(t)+D_{2}(t)C_{3}(t)(1+P_{3,i}(t)C_{4}(t))^{-1}\right)
P_{2,i}(t)\\
& \left.  \left.  +\left(  D_{1}(t)+D_{2}(t)A_{2}(t)\right)  D_{3}%
(t)P_{1,i}(t)\right]  \right\}  dt,
\end{array}
\]%
\[%
\begin{array}
[c]{rl}%
dP_{3,i}(t)= & -\left\{  -B_{4}(t)P_{3,i}(t)^{2}+\left[  2B_{3}(t)+C_{3}%
(t)^{2}(1+P_{3,i}(t)C_{4}(t))^{-1}\right]  P_{3,i}(t)\right. \\
& \left.  +\frac{1}{\Delta}\left[  P_{2,i}(t)D_{1}(t)+D_{3}(t)+C_{3}%
(t)D_{2}(t)P_{2,i}(t)(1+P_{3,i}(t)C_{4}(t))^{-1}\right]  ^{2}\right\}  dt,
\end{array}
\]
where
\[
\Delta=D_{4}(t)+D_{2}(t)^{2}P_{1,i}(t)+C_{4}(t)D_{2}(t)^{2}(1+P_{3,i}%
(t)C_{4}(t))^{-1}P_{2,i}(t)^{2}.
\]
{By Theorem \ref{th-conver}, $P_{3,i}(\cdot)$ is bounded. }Then one can check
that
\[
\left\vert P_{2,i}(t)\right\vert \leq\left\vert F\right\vert +C\int_{0}%
^{T}\left(  \left\vert P_{2,i}(s)\right\vert +1\right)  ds,
\]
where $C$ is a constant independent of $i$. By Gronwall's inequality,
$\left\vert P_{2,i}(\cdot)\right\vert $ is bounded. Because
\[
0\leq P_{1,i}(t)\leq G+C\int_{0}^{T}\left(  P_{1,i}(s)+1\right)  ds,
\]
where $C$ is a constant independent of $i$, we deduce that $P_{1,i}(\cdot)$
has a upper bound by Gronwall's inequality. Thus, (\ref{eq-p1}),
(\ref{eq-p3}), (\ref{eq-p4}) have a unique solution $\left(  P_{1}%
(\cdot),P_{2}(\cdot),P_{3}(\cdot)\right)  $ by Theorem \ref{th-conver}.
\end{example}

\section{Feedback optimal control for FBLQ problem\label{section-feedback}}

In this section, we prove the existence of optimal control without the
positiveness of $C_{4}(\cdot)$ and $D_{4}(\cdot)$. We first give the following lemma.

\begin{lemma}
\label{le-gamma}Suppose all assumptions in Theorem \ref{th-conver} hold. Then
$\left\{  \mathbb{E}\int_{0}^{T}|M_{5,i}(t)|dt\right\}  _{i\geq i_{0}}$ is
uniformly bounded, where $M_{5,i}(\cdot)$ is defined by replacing $\tilde
{P}(\cdot)$ with $_{i}\tilde{P}(\cdot)$ in (\ref{def-m1}).
\end{lemma}

The proof is in Appendix 7.3.

\begin{theorem}
\label{th-infi} Suppose Assumption \ref{assum-1} and all assumptions in
Theorem \ref{th-conver} hold, and $P_{3}(t)=\underset{i\rightarrow\infty
}{\lim}P_{3,i}(t)>0$, for $t\in\lbrack0,T)$. Then there exists an optimal
control $\bar{u}(\cdot)$ for the FBLQ problem (\ref{state-lq})-(\ref{cost-lq}%
). {Furthermore, any optimal control $\bar{u}(\cdot)$ satisfies}
\begin{equation}%
\begin{array}
[c]{rl}%
\bar{u}(t)= & \left[  L_{6}(t)+L_{7}(t)P_{3}(t)^{-1}P_{2}(t)\right]  \bar
{X}(t)\\
& -L_{7}(t)P_{3}(t)^{-1}\bar{Y}(t)+L_{7}(t)P_{3}(t)^{-1}\varphi_{2}%
(t)+S_{3}(t),
\end{array}
\label{eq-0113-2}%
\end{equation}
where $P_{1}(t)=\underset{i\rightarrow\infty}{\lim}P_{1,i}(t)$, $P_{2}%
(t)=\underset{i\rightarrow\infty}{\lim}P_{2,i}(t)$, and $L_{6}(t)$, $L_{7}%
(t)$, $S_{3}(t)$ are defined in (\ref{def-L1}).
\end{theorem}

\begin{proof}
By Theorem \ref{th-conver}, $\left(  P_{1}(\cdot),P_{2}(\cdot),P_{3}%
(\cdot)\right)  $ solves the equations (\ref{eq-p1}), (\ref{eq-p3}),
(\ref{eq-p4}). Then there exists a unique solution $(\left(  \varphi_{1}%
(\cdot)^{\intercal},\varphi_{2}(\cdot)^{\intercal}\right)  ^{\intercal
},\left(  v_{1}(\cdot)^{\intercal},v_{2}(\cdot)^{\intercal}\right)
^{\intercal})$ to the BSDE (\ref{eq-phi1}) and (\ref{eq-phi2}). Set
\[%
\begin{array}
[c]{l}%
N_{1}(t)=A_{1}(t)+B_{1}(t)P_{2}(t)+C_{1}(t)L_{10}(t)+D_{1}(t)L_{6}%
(t),\;N_{2}(t)=-B_{1}(t)P_{3}(t)+C_{1}(t)L_{11}(t)+D_{1}(t)L_{7}(t),\\
N_{3}(t)=B_{1}(t)\varphi_{2}(t)+C_{1}(t)S_{5}(t)+D_{1}(t)S_{3}(t),\;N_{4}%
(t)=A_{2}(t)+B_{2}(t)P_{2}(t)+C_{2}(t)L_{10}(t)+D_{2}(t)L_{6}(t),\\
N_{5}(t)=-B_{2}(t)P_{3}(t)+C_{2}(t)L_{11}(t)+D_{2}(t)L_{7}(t),\;N_{6}%
(t)=-B_{2}(t)\varphi_{2}(t)+C_{2}(t)S_{5}(t)+D_{2}(t)S_{3}(t),\\
N_{7}(t)=B_{1}(t)^{\intercal}P_{1}(t)+B_{2}(t)^{\intercal}L_{8}(t)+B_{4}%
(t)^{\intercal}P_{2}(t),\;N_{8}(t)=B_{3}(t)^{\intercal}+B_{1}(t)^{\intercal
}P_{2}(t)^{\intercal}+B_{2}(t)^{\intercal}L_{9}(t)-B_{4}(t)^{\intercal}%
P_{3}(t),\\
N_{9}(t)=B_{1}(t)^{\intercal}\varphi_{1}(t)+B_{2}(t)^{\intercal}S_{4}%
(t)+B_{4}(t)^{\intercal}\varphi_{2}(t),\;N_{10}(t)=C_{1}(t)^{\intercal}%
P_{1}(t)+C_{2}(t)^{\intercal}L_{8}(t)+C_{4}(t)^{\intercal}L_{10}(t),\\
N_{11}(t)=C_{3}(t)^{\intercal}+C_{1}(t)^{\intercal}P_{2}(t)^{\intercal}%
+C_{2}(t)^{\intercal}L_{9}(t)+C_{4}(t)^{\intercal}L_{11}(t),\;N_{12}%
(t)=C_{1}(t)^{\intercal}\varphi_{1}(t)+C_{2}(t)^{\intercal}S_{4}%
(t)+C_{4}(t)^{\intercal}S_{5}(t).
\end{array}
\]

Consider the following linear SDE for $\left(  X^{\ast}(\cdot),h^{\ast}%
(\cdot)\right)  $:
\begin{equation}
\left\{
\begin{array}
[c]{rl}%
dX^{\ast}(t)= & [N_{1}(t)X^{\ast}(t)+N_{2}(t)h^{\ast}(t)+N_{3}(t)]dt\\
& +[N_{4}(t)X^{\ast}(t)+N_{5}(t)h^{\ast}(t)+N_{6}(t)]dB(t),\\
dh^{\ast}(t)= & [N_{7}(t)X^{\ast}(t)+N_{8}(t)h^{\ast}(t)+N_{9}(t)]dt\\
& +[N_{10}(t)X^{\ast}(t)+N_{11}(t)h^{\ast}(t)+N_{12}(t)]dB(t),\\
X^{\ast}(0)= & x_{0},\text{ }h^{\ast}(0)=\left(  I_{m}+HP_{3}(0)\right)
^{-1}H\left(  P_{2}(0)x_{0}+\varphi_{2}(0)\right)  .
\end{array}
\right.  \label{eq-0113-1}%
\end{equation}
Since (\ref{eq-0113-1}) has bounded coefficients, it has a unique solution
$\left(  X^{\ast}(\cdot)^{\intercal},h^{\ast}(\cdot)^{\intercal}\right)
^{\intercal}\in$ $L_{\mathbb{F}}^{2}(\Omega;C([0,T],\mathbb{R}^{n+m}))$. Set
\begin{equation}%
\begin{array}
[c]{rl}%
\bar{u}(t)= & L_{6}(t)X^{\ast}(t)+L_{7}(t)h^{\ast}(t)+S_{3}(t)
\end{array}
\label{eq-u-indef}%
\end{equation}
which is an admissible control. It can be verified that%
\begin{equation}%
\begin{array}
[c]{cl}%
\bar{X}(t)= & X^{\ast}(t),\text{ }h(t)=h^{\ast}(t),\\
\bar{Y}(t)= & P_{2}(t)\bar{X}(t)-P_{3}(t)h(t)+\varphi_{2}(t),\text{ }\bar
{Z}(t)=L_{10}(t)\bar{X}(t)+L_{11}(t)h(t)+S_{5}(t),\\
m(t)= & P_{1}(t)\bar{X}(t)+P_{2}(t)^{\intercal}h(t)+\varphi_{1}(t),\text{
}n(t)=L_{8}(t)\bar{X}(t)+L_{9}(t)h(t)+S_{4}(t)
\end{array}
\label{eq-xhyz}%
\end{equation}
solves the Hamiltonian system (\ref{eq-ham}). Now we prove that $\bar{u}%
(\cdot)$ is an optimal control in two steps.

\textbf{Step 1:} For $t\in\lbrack0,T)$, set
\[%
\begin{array}
[c]{rl}%
\tilde{P}(t)= & \left(
\begin{array}
[c]{cc}%
P_{1}(t)+P_{2}(t)^{\intercal}P_{3}(t)^{-1}P_{2}(t) & -P_{2}(t)^{\intercal
}P_{3}(t)^{-1}\\
-P_{3}(t)^{-1}P_{2}(t) & P_{3}(t)^{-1}%
\end{array}
\right)  .
\end{array}
\]
For any given $\varepsilon>0$, by Theorem \ref{th-p-ex}, $\tilde{P}(\cdot)$
solves the equation (\ref{eq-rec}) on $[0,T-\varepsilon]$. By the
completion-of-squares technique, we have%
\begin{equation}%
\begin{array}
[c]{rl}
& J(\bar{u}(\cdot))\\
= & \frac{1}{2}[\left(  \left(  x_{0}-\tilde{\varphi}_{1}(0)\right)
^{\intercal},\left(  \bar{Y}(0)-\tilde{\varphi}_{2}(0)\right)  ^{\intercal
}\right)  \tilde{P}(0)\left(  \left(  x_{0}-\tilde{\varphi}_{1}(0)\right)
^{\intercal},\left(  \bar{Y}(0)-\tilde{\varphi}_{2}(0)\right)  ^{\intercal
}\right)  ^{\intercal}+\bar{Y}(0)^{\intercal}H\bar{Y}(0)]\\
& +\frac{1}{2}\mathbb{E}\left[  \bar{X}(T)^{\intercal}G\bar{X}(T)-\binom
{\bar{X}(T-\varepsilon)-\tilde{\varphi}_{1}(T-\varepsilon)}{\bar
{Y}(T-\varepsilon)-\tilde{\varphi}_{2}(T-\varepsilon)}^{\intercal}\tilde
{P}(T-\varepsilon)\binom{\bar{X}(T-\varepsilon)-\tilde{\varphi}_{1}%
(T-\varepsilon)}{\bar{Y}(T-\varepsilon)-\tilde{\varphi}_{2}(T-\varepsilon
)}\right] \\
& +\frac{1}{2}\mathbb{E}\int_{0}^{T-\varepsilon}M_{5}(t)dt\\
& +\frac{1}{2}\mathbb{E\displaystyle}\int_{0}^{T-\varepsilon}\left\{  \left[
\binom{\bar{u}(t)}{\bar{Z}(t)}+M_{1}(t)^{-1}\left(  M_{2}(t)\binom{\bar{X}%
(t)}{\bar{Y}(t)}-M_{3}(t)\tilde{\varphi}(t)-M_{4}(t)\tilde{v}(t)\right)
\right]  ^{\intercal}\right. \\
& \left.  \cdot M_{1}(t)\left[  \binom{\bar{u}(t)}{\bar{Z}(t)}+M_{1}%
(t)^{-1}\left(  M_{2}(t)\binom{\bar{X}(t)}{\bar{Y}(t)}-M_{3}(t)\tilde{\varphi
}(t)-M_{4}(t)\tilde{v}(t)\right)  \right]  \right\}  dt\\
& +\frac{1}{2}\mathbb{E[}\int_{T-\varepsilon}^{T}(\left\langle A_{4}(t)\bar
{X}(t),\bar{X}(t)\right\rangle +\left\langle B_{4}(t)\bar{Y}(t),\bar
{Y}(t)\right\rangle +\left\langle C_{4}(t)\bar{Z}(t),\bar{Z}(t)\right\rangle
+\left\langle D_{4}(t)\bar{u}(t),\bar{u}(t)\right\rangle )dt]\\
= & \left(  I\right)  +\left(  II\right)  +\left(  III\right)  +\left(
IV\right)  +\left(  V\right)  .
\end{array}
\label{eq-j-ub}%
\end{equation}
The part $(I)$ is simplified as follows.
\[%
\begin{array}
[c]{l}%
\left(  \left(  x_{0}-\tilde{\varphi}_{1}(0)\right)  ^{\intercal},\left(
\bar{Y}(0)-\tilde{\varphi}_{2}(0)\right)  ^{\intercal}\right)  \tilde
{P}(0)\left(  \left(  x_{0}-\tilde{\varphi}_{1}(0)\right)  ^{\intercal
},\left(  \bar{Y}(0)-\tilde{\varphi}_{2}(0)\right)  ^{\intercal}\right)
^{\intercal}+\bar{Y}(0)^{\intercal}H\bar{Y}(0)\\
=R_{1}(\bar{Y}(0))^{\intercal}\left(  P_{3}(0)^{-1}+H\right)  R_{1}(\bar
{Y}(0))+R_{2},
\end{array}
\]
where%
\begin{equation}%
\begin{array}
[c]{rl}%
R_{1}(y)= & y-\left(  I_{m}+P_{3}(0)H\right)  ^{-1}\left(  P_{2}%
(0)x_{0}+\varphi_{2}(0)\right)  ,\\
R_{2}= & \left(  x_{0}+P_{1}(0)^{-1}\varphi_{1}(0)\right)  ^{\intercal}%
P_{1}(0)\left(  x_{0}+P_{1}(0)^{-1}\varphi_{1}(0)\right) \\
& +\left[  P_{2}(0)x_{0}+\varphi_{2}(0)\right]  ^{\intercal}\left(
I_{m}+HP_{3}(0)\right)  ^{-1}H\left[  P_{2}(0)x_{0}+\varphi_{2}(0)\right]  .
\end{array}
\label{eq-r12}%
\end{equation}
One can check that $I_{m}-P_{3}(0)\left(  I_{m}+HP_{3}(0)\right)
^{-1}H-\left(  I_{m}+P_{3}(0)H\right)  ^{-1}=0$ which implies $R_{1}(\bar
{Y}(0))=0$.

Then, we prove the part $(II)$ converges to $0$ as $\varepsilon\rightarrow0$.
Noting that $\bar{Y}(\cdot)-P_{2}(\cdot)\bar{X}(\cdot)-\varphi_{2}%
(\cdot)=-P_{3}(\cdot)h(\cdot)$, we have
\[%
\begin{array}
[c]{l}%
\;\bar{X}(T)^{\intercal}G\bar{X}(T)-\binom{\bar{X}(T-\varepsilon
)-\tilde{\varphi}_{1}(T-\varepsilon)}{\bar{Y}(T-\varepsilon)-\tilde{\varphi
}_{2}(T-\varepsilon)}^{\intercal}\tilde{P}(T-\varepsilon)\binom{\bar
{X}(T-\varepsilon)-\tilde{\varphi}_{1}(T-\varepsilon)}{\bar{Y}(T-\varepsilon
)-\tilde{\varphi}_{2}(T-\varepsilon)}\\
=\bar{X}(T)^{\intercal}G\bar{X}(T)-\left(  \bar{X}(T-\varepsilon
)+P_{1}(T-\varepsilon)^{-1}\varphi_{1}(T-\varepsilon)\right)  ^{\intercal
}P_{1}(T-\varepsilon)\\
\ \ \cdot\left(  \bar{X}(T-\varepsilon)+P_{1}(T-\varepsilon)^{-1}\varphi
_{1}(T-\varepsilon)\right)  +h(T-\varepsilon)^{\intercal}P_{3}(T-\varepsilon
)h(T-\varepsilon)\\
\overset{L^{1}}{\rightarrow}0\text{ (as }\varepsilon\rightarrow0\text{).}%
\end{array}
\]
The part $(V)$ converges to $0$ as $\varepsilon\rightarrow0$ due to the
integrability of $\bar{X}(\cdot)$, $\bar{Y}(\cdot)$ and $\bar{Z}(\cdot)$. By
Lemma \ref{le-con-equ}, (\ref{eq-0113-2}) and (\ref{eq-xhyz}), we deduce that
the part $(IV)$ equals to $0$.

Finally, by Lemma \ref{le-gamma} and letting $\varepsilon\rightarrow0$ on both
sides of (\ref{eq-j-ub}), we have
\begin{equation}%
\begin{array}
[c]{rl}%
J(\bar{u}(\cdot))= & \frac{1}{2}R_{2}+\frac{1}{2}\mathbb{E}\int_{0}^{T}%
M_{5}(t)dt.
\end{array}
\label{eq-j-op}%
\end{equation}
\textbf{Step 2.} We first give a lower bound for the cost functional by the
completion-of-squares technique. For an admissible control $u(\cdot)$, let
$\left(  X(\cdot),Y(\cdot),Z(\cdot)\right)  $ be the corresponding state
process. Set $\tilde{X}(t)=\left(  X(t)^{\intercal},Y(t)^{\intercal}\right)
^{\intercal}$. Applying It\^{o}'s formula to
\[
\left(  \tilde{X}(t)-\text{ }_{i}\tilde{\varphi}(t)\right)  ^{\intercal}\text{
}_{i}\tilde{P}(t)\left(  \tilde{X}(t)-\text{ }_{i}\tilde{\varphi}(t)\right)
\]
and taking expectations, we have%
\[%
\begin{array}
[c]{l}%
\mathbb{E}\left[  X(T)^{\intercal}GX(T)\right]  +Y(0)^{\intercal}%
HY(0)=R_{1,i}(Y(0))^{\intercal}\left(  P_{3,i}(0)^{-1}+H\right)
R_{1,i}(Y(0))+R_{2,i}\\
\text{ \ }+\mathbb{E}\left[  \int_{0}^{T}d\left(  \tilde{X}(t)-\text{ }%
_{i}\tilde{\varphi}(t)\right)  ^{\intercal}\text{ }_{i}\tilde{P}(t)\left(
\tilde{X}(t)-\text{ }_{i}\tilde{\varphi}(t)\right)  \right]  ,
\end{array}
\]
where $R_{1,i}(y)$ and $R_{2,i}$ are defined by replacing $P(0)$ with
$_{i}P(0)$ in (\ref{eq-r12}). By the completion-of-squares technique,
\[%
\begin{array}
[c]{l}%
J(u(\cdot))\\
=\frac{1}{2}\left[  R_{1,i}(Y(0))^{\intercal}\left(  P_{3,i}(0)^{-1}+H\right)
R_{1,i}(Y(0))+R_{2,i}\right]  +\frac{1}{2}\mathbb{E}\int_{0}^{T}M_{5,i}(t)dt\\
\text{ \ }+\frac{1}{2}\mathbb{E}\displaystyle\int_{0}^{T}\left\{  \left[
\left(  u(t)^{\intercal},Z(t)^{\intercal}\right)  ^{\intercal}+M_{1,i}%
(t)^{-1}\left(  M_{2,i}(t)\tilde{X}(t)-M_{3,i}(t)\tilde{\varphi}%
(t)-M_{4,i}(t)\tilde{v}(t)\right)  \right]  ^{\intercal}\right. \\
\text{ \ }\left.  M_{1,i}(t)\left[  \left(  u(t)^{\intercal},Z(t)^{\intercal
}\right)  ^{\intercal}+M_{1,i}(t)^{-1}\left(  M_{2,i}(t)\tilde{X}%
(t)-M_{3,i}(t)\tilde{\varphi}(t)-M_{4,i}(t)\tilde{v}(t)\right)  \right]
\right\}  dt.
\end{array}
\]
Note that $M_{1,i}(t)>0$. Letting $i\rightarrow\infty$ and appealing to
Fatou's lemma and Lemma \ref{le-gamma}, we have
\begin{equation}%
\begin{array}
[c]{l}%
J(u(\cdot))\\
\geq J(\bar{u}(\cdot))+\frac{1}{2}R_{1,i}(Y(0))^{\intercal}\left(
P_{3,i}(0)^{-1}+H\right)  R_{1,i}(Y(0))\\
\text{ \ }+\frac{1}{2}\mathbb{E}\displaystyle\int_{0}^{T}\left\{  \left[
\left(  u(t)^{\intercal},Z(t)^{\intercal}\right)  ^{\intercal}+M_{1}%
(t)^{-1}\left(  M_{2}(t)\tilde{X}(t)-M_{3}(t)\tilde{\varphi}(t)-M_{4}%
(t)\tilde{v}(t)\right)  \right]  ^{\intercal}\right. \\
\text{ \ \ }\left.  M_{1}(t)\left[  \left(  u(t)^{\intercal},Z(t)^{\intercal
}\right)  ^{\intercal}+M_{1}(t)^{-1}\left(  M_{2}(t)\tilde{X}(t)-M_{3}%
(t)\tilde{\varphi}(t)-M_{4}(t)\tilde{v}(t)\right)  \right]  \right\}  dt\\
\geq J(\bar{u}(\cdot)).
\end{array}
\label{eq-128-1}%
\end{equation}
Since $\bar{u}(\cdot)$ achieves the lower bound, it is clear that $\bar
{u}(\cdot)$ is optimal.

For any other optimal control $\check{u}(\cdot)$, by (\ref{eq-128-1}) and
$J(\check{u}(\cdot))=J(\bar{u}(\cdot))$, we have%
\[
\left(  \check{u}(t)^{\intercal},\check{Z}(t)^{\intercal}\right)  ^{\intercal
}+M_{1}(t)^{-1}\left(  M_{2}(t)\left(  \check{X}(t)^{\intercal},\check
{Y}(t)^{\intercal}\right)  ^{\intercal}-M_{3}(t)\tilde{\varphi}(t)-M_{4}%
(t)\tilde{v}(t)\right)  =0.
\]
By Lemma \ref{le-con-equ}, we obtain (\ref{eq-0113-2}). This completes the proof.
\end{proof}

In the following we solve a special case of the FBLQ\ problem in which
$D_{4}<0$ and $C_{4}<0$.

\begin{example}
Suppose that all variables are $1$-dimensional. For the FBLQ\ problem
(\ref{state-lq})-(\ref{cost-lq}), suppose that $A_{3}(t)=B_{1}(t)=C_{1}%
(t)=B_{2}(t)=C_{2}(t)=F=\xi=0$ and $D_{1}(t)+D_{2}(t)A_{2}(t)=0$. Then the
solutions to (\ref{eq-p1}), (\ref{eq-p3}) and (\ref{eq-p4}) are $P_{1,i}%
(t)=Ge^{\int_{t}^{T}\left(  2A_{1}(s)+A_{2}(s)^{2}\right)  ds}+\int_{t}%
^{T}A_{4}(s)e^{\int_{t}^{s}\left(  2A_{1}(r)+A_{2}(r)^{2}\right)  dr}ds$,
$P_{2,i}(t)\equiv0$ and $P_{3,i}(\cdot)$ satisfies%
\[
\left\{
\begin{array}
[c]{rl}%
dP_{3,i}(t)= & -\left\{  -B_{4}(t)P_{3,i}(t)^{2}+\left[  2B_{3}(t)+C_{3}%
(t)^{2}(1+P_{3,i}(t)C_{4}(t))^{-1}\right]  P_{3,i}(t)\right. \\
& \left.  +\frac{D_{3}(t)^{2}}{D_{4}(t)+D_{2}(t)^{2}P_{1,i}(t)}\right\}  dt,\\
P_{3,i}(T)= & i^{-1}.
\end{array}
\right.
\]
Suppose that $D_{4}<0$, $C_{4}<0$, $D_{4}(t)+D_{2}(t)^{2}P_{1,i}(t)\geq
\delta>0$, $D_{3}(t)^{2}>0$ and $1+\check{P}_{3,i}(t)C_{4}(t)\geq\delta>0$
where
\[
\left\{
\begin{array}
[c]{rl}%
d\breve{P}_{3,i}(t)= & -\left\{  \left[  2B_{3}(t)+C_{3}(t)^{2}\delta
^{-1}\right]  \breve{P}_{3,i}(t)+\frac{D_{3}(t)^{2}}{D_{4}(t)+D_{2}%
(t)^{2}P_{1,i}(t)}\right\}  dt,\\
\check{P}_{3,i}(T)= & i^{-1}.
\end{array}
\right.
\]
By Comparison theorem we have $P_{3,i}(t)\leq\breve{P}_{3,i}(t)$ which leads
to $1+P_{3,i}(t)C_{4}(t)\geq\delta$. Then, by Theorem \ref{th-conver}
$(P_{1}(\cdot)$, $P_{2}(\cdot)$, $P_{3}(\cdot))$ has a unique solution.
Moreover,
\[
P_{3}(t)=\int_{t}^{T}\frac{D_{3}(s)^{2}}{D_{4}(s)+D_{2}(s)^{2}P_{1}(s)}%
e^{\int_{t}^{s}\left(  2B_{3}(r)+C_{3}(r)^{2}(1+P_{3}(r)C_{4}(r))^{-1}%
-B_{4}(r)P_{3}(r)\right)  dr}ds.
\]
It is obvious that $P_{3}(t)>0$ for $t<T$. Thus, by Theorem \ref{th-infi} the
optimal control is
\begin{equation}%
\begin{array}
[c]{rl}%
\bar{u}(t)= & -\left(  D_{4}(t)+D_{2}(t)^{2}P_{1}(t)\right)  ^{-1}\cdot\left[
\left(  D_{1}(t)+A_{2}(t)D_{2}(t)\right)  P_{1}(t)\bar{X}(t)-D_{3}%
(t)P_{3}(t)^{-1}\bar{Y}(t)\right]  .
\end{array}
\label{optimal control-example}%
\end{equation}

\end{example}

\begin{remark}
Although the forward-backward stochastic control system in the above example
is completely decoupled, in order to obtain the optimal control $\bar{u}%
(\cdot)$ in (\ref{optimal control-example}) we still need to solve a fully
coupled FBSDE.
\end{remark}

\section{Some special cases}

{In this section}, we illustrate our results for the indefinite stochastic LQ,
BLQ and deterministic FBLQ problems.

\subsection{Indefinite stochastic LQ problem}

If $A_{3}(\cdot)=D_{3}(\cdot)=B_{i}(\cdot)=C_{i}(\cdot)=F=H=\xi=0,$ $i=2,3,4$,
then the FBLQ problem (\ref{state-lq})-(\ref{cost-lq}) degenerates to the
following indefinite stochastic LQ problem as in \cite{Chen-LZ}: minimizing
the following cost functional
\[
J(u(\cdot))=\frac{1}{2}\mathbb{E}\left[  \int_{0}^{T}\left(  \left\langle
A_{4}(t)X(t),X(t)\right\rangle +\left\langle D_{4}(t)u(t),u(t)\right\rangle
\right)  dt+\left\langle GX(T),X(T)\right\rangle \right]
\]
subject to%
\[
\left\{
\begin{array}
[c]{rl}%
dX(t)= & [A_{1}(t)X(t)+D_{1}(t)u(t)]dt+[A_{2}(t)X(t)+D_{2}(t)u(t)]dB(t),\\
X(0)= & x_{0}.
\end{array}
\right.
\]
By Theorem \ref{th-infi}, the optimal control is%
\[
\bar{u}(t)=-\left(  D_{4}(t)+D_{2}(t)^{\intercal}P_{1}(t)D_{2}(t)\right)
^{-1}\left(  D_{1}(t)^{\intercal}P_{1}(t)+D_{2}(t)^{\intercal}P_{1}%
(t)A_{2}(t)\right)  \bar{X}(t),
\]
where
\[
\left\{
\begin{array}
[c]{l}%
\dot{P}_{1}(t)+A_{1}(t)^{\intercal}P_{1}(t)+P_{1}(t)A_{1}(t)+A_{4}%
(t)+A_{2}(t)^{\intercal}P_{1}(t)A_{2}(t)\\
-(D_{1}(t)^{\intercal}P_{1}(t)+D_{2}(t)^{\intercal}P_{1}(t)A_{2}%
(t))^{\intercal}\left(  D_{4}(t)+D_{2}(t)^{\intercal}P_{1}(t)D_{2}(t)\right)
^{-1}\\
\ \ \cdot(D_{1}(t)^{\intercal}P_{1}(t)+D_{2}(t)^{\intercal}P_{1}%
(t)A_{2}(t))=0,\\
P_{1}(T)=G,\\
D_{4}(t)+D_{2}(t)^{\intercal}P_{1}(t)D_{2}(t)>0.
\end{array}
\right.
\]
The state feedback representation of the optimal control and the Riccati
equation for $P_{1}(\cdot)$ are just the corresponding ones in Theorem 3.2 in
Chen, Li and Zhou \cite{Chen-LZ}.

\subsection{BLQ problem}

If $A_{i}(\cdot)=B_{i}(\cdot)=C_{i}(\cdot)=D_{i}(\cdot)=F=G=A_{3}(\cdot
)=A_{4}(\cdot)=0,$ $i=1,2$ and $D_{4}(\cdot)>0$, then the problem
(\ref{state-lq})-(\ref{cost-lq}) degenerates to the following BLQ problem as
in \cite{Lim-Z}: minimizing the following cost functional
\[%
\begin{array}
[c]{rl}%
J(u(\cdot))= & \frac{1}{2}\mathbb{E}\left[  \int_{0}^{T}\left(  \left\langle
B_{4}(t)Y(t),Y(t)\right\rangle +\left\langle C_{4}(t)Z(t),Z(t)\right\rangle
+\left\langle D_{4}(t)u(t),u(t)\right\rangle \right)  dt\right] \\
& +\frac{1}{2}\left\langle HY(0),Y(0)\right\rangle ,
\end{array}
\]
subject to%
\[
\left\{
\begin{array}
[c]{rl}%
dY(t)= & -[B_{3}(t)Y(t)+C_{3}(t)Z(t)+D_{3}(t)u(t)]+Z(t)dB(t),\\
Y(T)= & \xi.
\end{array}
\right.
\]
By Theorem \ref{th-opti-cont}, the optimal control is
\[
\bar{u}(t)=-D_{4}(t)^{-1}D_{3}(t)^{\intercal}h(t),
\]
and the following relation holds:%
\[
\bar{Y}(t)=-Q_{4}(t)h(t)-\varphi_{2}(t),
\]
where
\[
\left\{
\begin{array}
[c]{l}%
dQ_{4}(t)\\
=-\left\{  Q_{4}(t)B_{3}(t)^{\intercal}+B_{3}(t)Q_{4}(t)-Q_{4}(t)B_{4}%
(t)Q_{4}(t)-D_{3}(t)D_{4}(t)^{-1}D_{3}(t)^{\intercal}\right. \\
\left.  +C_{3}(t)(I_{m}+Q_{4}(t)C_{4}(t))^{-1}Q_{4}(t)C_{3}(t)^{\intercal
}\right\}  dt,\\
Q_{4}(T)=0,
\end{array}
\right.
\]%
\[
\left\{
\begin{array}
[c]{l}%
d\varphi_{2}(t)\\
=-\left\{  Q_{4}(t)B_{3}(t)\varphi_{2}(t)+B_{3}(t)\varphi_{2}(t)+C_{3}%
(t)(I_{m}+P_{3}(t)C_{4}(t))^{-1}v_{2}(t)\right\}  dt\\
\ \ +v_{2}(t)dB(t),\\
\varphi_{2}(T)=\xi.
\end{array}
\right.
\]
The equation for $Q_{4}(\cdot)$ is just the Riccati equation (3.4) in Lim and
Zhou \cite{Lim-Z}. And the optimal control is consistent with the one in
Theorem 3.3 in \cite{Lim-Z}.

\begin{remark}
It is worth pointing out that our results in this paper can be also applied to
the indefinite BLQ problem.
\end{remark}

\subsection{Deterministic FBLQ problem}

If $C_{1}(\cdot)=A_{2}(\cdot)=B_{2}(\cdot)=C_{2}(\cdot)=D_{2}(\cdot
)=C_{3}(\cdot)=C_{4}(\cdot)=\xi=0$ and $D_{4}(\cdot)>0$, then the problem
(\ref{state-lq})-(\ref{cost-lq}) degenerates to a deterministic FBLQ problem.
For this case, (\ref{eq-p1}), (\ref{eq-p3}), (\ref{eq-p4}) become
\begin{equation}
\left\{
\begin{array}
[c]{l}%
\dot{P}_{1}(t)+P_{1}(t)A_{1}(t)+A_{1}(t)^{\intercal}P_{1}(t)+P_{1}%
(t)B_{1}(t)P_{2}(t)+P_{2}(t)^{\intercal}B_{1}(t)^{\intercal}P_{1}(t)\\
-P_{1}(t)D_{1}(t)D_{4}(t)^{-1}D_{1}(t)^{\intercal}P_{1}(t)+P_{2}%
(t)^{\intercal}B_{4}(t)P_{2}(t)+A_{4}(t)=0,\\
\dot{P}_{2}(t)+P_{2}(t)A_{1}(t)+B_{3}(t)P_{2}(t)-P_{3}(t)B_{4}(t)P_{2}(t)\\
-P_{2}(t)D_{1}(t)D_{4}(t)^{-1}D_{1}(t)^{\intercal}P_{1}(t)+P_{2}%
(t)B_{1}(t)P_{2}(t)-P_{3}(t)B_{1}(t)^{\intercal}P_{1}(t)\\
-D_{3}(t)D_{4}(t)^{-1}D_{1}(t)^{\intercal}P_{1}(t)+A_{3}(t)=0,\\
\dot{P}_{3}(t)+P_{3}(t)B_{3}(t)^{\intercal}+B_{3}(t)P_{3}(t)+P_{2}%
(t)B_{1}(t)P_{3}(t)+P_{3}(t)B_{1}(t)^{\intercal}P_{2}(t)^{\intercal}\\
-P_{3}(t)B_{4}(t)P_{3}(t)+(P_{2}(t)D_{1}(t)+D_{3}(t))D_{4}(t)^{-1}%
(P_{2}(t)D_{1}(t)+D_{3}(t))^{\intercal}=0,\\
P_{1}(T)=G\text{, }P_{2}(T)=F\text{, }P_{3}(T)=0.
\end{array}
\right.  \label{eq-p-de}%
\end{equation}
By Theorems \ref{th-opti-cont} and \ref{th-infi}, we obtain the following proposition.

\begin{proposition}
Suppose that Assumptions \ref{assum-1}, \ref{assum-bound}, \ref{assum-pos} and
\ref{assum-pos-d4} hold. If (\ref{eq-p-de}) has a solution $\left(
P_{1}(\cdot),\text{ }P_{2}(\cdot),\text{ }P_{3}(\cdot)\right)  \in C\left(
\left[  0,T\right]  ;\mathbb{S}^{n}\right.  $ $\times\mathbb{R}^{m\times
n}\left.  \times\mathbb{S}^{m}\right)  $ such that $P_{3}(t)>0$ for $t<T$,
then the above deterministic FBLQ problem has a unique optimal control
\[%
\begin{array}
[c]{l}%
\bar{u}(t)\\
=-D_{4}(t)^{-1}\left\{  \left[  D_{1}(t)^{\intercal}\left(  P_{1}%
(t)+P_{2}(t)^{\intercal}P_{3}(t)^{-1}P_{2}(t)\right)  +D_{3}(t)^{\intercal
}P_{3}(t)^{-1}P_{2}(t)\right]  \bar{X}(t)\right. \\
\ \ \left.  -\left(  D_{1}(t)^{\intercal}P_{2}(t)^{\intercal}+D_{3}%
(t)^{\intercal}\right)  P_{3}(t)^{-1}\bar{Y}(t)\right\}  .
\end{array}
\]

\end{proposition}

For $1$-dimensional case ($n=m=1$), if $B_{4}(\cdot)=0$, $B_{1}(\cdot)<0$ and
$F$ is large enough such that $P_{2,i}(\cdot)$ is non-negative and bounded,
then (\ref{eq-p-de}) has a unique solution. The reason is that when
$B_{4}(\cdot)=0$, $B_{1}(\cdot)<0$ and $P_{2,i}(\cdot)\geq0$, we have
\[
P_{1,i}(t)\leq Ge^{\int_{t}^{T}2A_{1}(s)ds}+\int_{t}^{T}A_{4}(s)e^{\int%
_{t}^{s}2A_{1}(r)dr}ds.
\]
Thus, $P_{1,i}(\cdot)$ is bounded and we obtain the desired result due to
Theorem \ref{th-conver}.

\section{Appendix}

This appendix is devoted to proofs of Theorem \ref{th-p-ex}, Lemma
\ref{le-con-equ} and Lemma \ref{le-gamma}. Before giving the proofs, let's
give some notations.

Set
\begin{equation}
\left(  \tilde{R}(t)+\tilde{D}(t)^{\intercal}{}_{i}\tilde{P}(t)\tilde
{D}(t)\right)  ^{-1}=\left(
\begin{array}
[c]{cc}%
a_{11}(t) & a_{21}(t)^{\intercal}\\
a_{21}(t) & a_{22}(t)
\end{array}
\right)  , \label{app-def1}%
\end{equation}%
\begin{equation}
\left(  \tilde{B}(t)^{\intercal}{}_{i}\tilde{P}(t)+\tilde{D}(t)^{\intercal}%
{}_{i}\tilde{P}(t)\tilde{C}(t)\right)  ^{\intercal}=\left(
\begin{array}
[c]{cc}%
a(t) & b(t)\\
c(t) & d(t)
\end{array}
\right)  \label{app-def2}%
\end{equation}
where
\begin{equation}%
\begin{array}
[c]{rl}%
a(t)= & \tilde{P}_{1,i}(t)D_{1}(t)-\tilde{P}_{2,i}(t)^{\intercal}%
D_{3}(t)+A_{2}(t)^{\intercal}\tilde{P}_{1,i}(t)D_{2}(t),\\
b(t)= & \tilde{P}_{1,i}(t)C_{1}(t)-\tilde{P}_{2,i}(t)^{\intercal}%
C_{3}(t)+A_{2}(t)^{\intercal}\tilde{P}_{1,i}(t)C_{2}(t)+A_{2}(t)^{\intercal
}\tilde{P}_{2,i}(t)^{\intercal},\\
c(t)= & \tilde{P}_{2,i}(t)D_{1}(t)-\tilde{P}_{3,i}(t)D_{3}(t)+B_{2}%
(t)^{\intercal}\tilde{P}_{1,i}(t)D_{2}(t),\\
d(t)= & \tilde{P}_{2,i}(t)C_{1}(t)-\tilde{P}_{3,i}(t)C_{3}(t)+B_{2}%
(t)^{\intercal}\tilde{P}_{1,i}(t)C_{2}(t)+B_{2}(t)^{\intercal}\tilde{P}%
_{2,i}(t)^{\intercal},\\
a_{11}(t)= & \left\{  D_{4}(t)+D_{2}(t)^{\intercal}\tilde{P}_{1,i}%
(t)D_{2}(t)\right. \\
& \left.  -D_{2}(t)^{\intercal}\left(  \tilde{P}_{1,i}(t)C_{2}(t)+\tilde
{P}_{2,i}(t)^{\intercal}\right)  D(t)^{-1}\left(  C_{2}(t)^{\intercal}%
\tilde{P}_{1,i}(t)+\tilde{P}_{2,i}(t)\right)  D_{2}(t)\right\}  ^{-1},\\
a_{21}(t)= & -D(t)^{-1}\left(  C_{2}(t)^{\intercal}\tilde{P}_{1,i}%
(t)+\tilde{P}_{2,i}(t)\right)  D_{2}(t)a_{11}(t),\\
a_{22}(t)= & D(t)^{-1}+D(t)^{-1}\left(  C_{2}(t)^{\intercal}\tilde{P}%
_{1,i}(t)+\tilde{P}_{2,i}(t)\right)  D_{2}(t)a_{11}(t)D_{2}(t)^{\intercal
}\left(  \tilde{P}_{1,i}(t)C_{2}(t)+\tilde{P}_{2,i}(t)^{\intercal}\right)
D(t)^{-1},\\
D(t)= & C_{4}(t)+\left(  C_{2}(t)^{\intercal}\tilde{P}_{1,i}(t)+\tilde
{P}_{2,i}(t)\right)  C_{2}(t)+C_{2}(t)^{\intercal}\tilde{P}_{2,i}%
(t)^{\intercal}+\tilde{P}_{3,i}(t).
\end{array}
\label{def-nota-1}%
\end{equation}
And set%
\begin{equation}%
\begin{array}
[c]{c}%
\tilde{B}(t)^{\intercal}{}_{i}\tilde{P}(t)=\left(
\begin{array}
[c]{cc}%
\tilde{a}(t) & \tilde{b}(t)\\
\tilde{c}(t) & \tilde{d}(t)
\end{array}
\right)  ,\\
\tilde{D}(t)^{\intercal}{}_{i}\tilde{P}(t)=\left(
\begin{array}
[c]{cc}%
\bar{a}(t) & \bar{b}(t)\\
\bar{c}(t) & \bar{d}(t)
\end{array}
\right)  ,
\end{array}
\label{def-nota-2}%
\end{equation}
where
\begin{equation}%
\begin{array}
[c]{rlll}%
\tilde{a}(t)= & D_{1}(t)^{\intercal}\tilde{P}_{1,i}(t)-D_{3}(t)^{\intercal
}\tilde{P}_{2,i}(t), & \bar{a}(t)= & D_{2}(t)^{\intercal}\tilde{P}_{1,i}(t),\\
\tilde{b}(t)= & D_{1}(t)^{\intercal}\tilde{P}_{2,i}(t)^{\intercal}%
-D_{3}(t)^{\intercal}\tilde{P}_{3,i}(t), & \bar{b}(t)= & D_{2}(t)^{\intercal
}\tilde{P}_{2,i}(t)^{\intercal},\\
\tilde{c}(t)= & C_{1}(t)^{\intercal}\tilde{P}_{1,i}(t)-C_{3}(t)^{\intercal
}\tilde{P}_{2,i}(t), & \bar{c}(t)= & C_{2}(t)^{\intercal}\tilde{P}%
_{1,i}(t)+\tilde{P}_{2,i}(t),\\
\tilde{d}(t)= & C_{1}(t)^{\intercal}\tilde{P}_{2,i}(t)^{\intercal}%
-C_{3}(t)^{\intercal}\tilde{P}_{3,i}(t)^{\intercal}, & \bar{d}(t)= &
C_{2}(t)^{\intercal}\tilde{P}_{2,i}(t)^{\intercal}.
\end{array}
\label{def-nota-app}%
\end{equation}

\subsection{Proof of Theorem \ref{th-p-ex}}

Before we prove Theorem \ref{th-p-ex}, we list the following relations which
can be verified directly:%
\begin{equation}%
\begin{array}
[c]{rl}%
P_{3,i}(t)C_{2}(t)^{\intercal}\left(  C_{2}(t)^{\intercal}\tilde{P}%
_{1,i}(t)+\tilde{P}_{2,i}(t)\right)  ^{\intercal}= & P_{3,i}(t)D(t)-\left(
I_{m}-P_{2,i}(t)C_{2}(t)+P_{3,i}(t)C_{4}(t)\right)  ,
\end{array}
\label{app-1}%
\end{equation}%
\begin{equation}%
\begin{array}
[c]{rl}%
\left(  C_{2}(t)^{\intercal}\tilde{P}_{1,i}(t)+\tilde{P}_{2,i}(t)\right)
^{\intercal}D(t)^{-1}= & L_{2,i}(t)^{-1}\left(  C_{2}(t)^{\intercal}\tilde
{P}_{1,i}(t)+\tilde{P}_{2,i}(t)\right)  ^{\intercal}\\
& \cdot\left(  I_{m}-P_{2,i}(t)C_{2}(t)+P_{3,i}(t)C_{4}\right)  ^{-1}%
P_{3,i}(t),
\end{array}
\label{app-2}%
\end{equation}%
\begin{equation}%
\begin{array}
[c]{rl}%
P_{3,i}(t)d(t)= & -(P_{2,i}(t)C_{1}(t)+C_{3}(t))+P_{3,i}(t)B_{2}%
(t)^{\intercal}\left(  C_{2}(t)^{\intercal}\tilde{P}_{1,i}(t)+\tilde{P}%
_{2,i}(t)\right)  ^{\intercal},
\end{array}
\label{app-3}%
\end{equation}%
\begin{equation}%
\begin{array}
[c]{rl}%
L_{2,i}(t)^{-1}S_{1,i}(t)= & \tilde{P}_{1,i}(t)-\left(  C_{2}(t)^{\intercal
}\tilde{P}_{1,i}(t)+\tilde{P}_{2,i}(t)\right)  ^{\intercal}D(t)^{-1}\left(
C_{2}(t)^{\intercal}\tilde{P}_{1,i}(t)+\tilde{P}_{2,i}(t)\right)  ,
\end{array}
\label{app-4}%
\end{equation}%
\begin{equation}%
\begin{array}
[c]{rl}%
a_{11}(t)= & \left(  D_{4}(t)+D_{2}(t)^{\intercal}L_{2,i}(t)^{-1}%
S_{1,i}(t)D_{2}(t)\right)  ^{-1},
\end{array}
\label{app-5}%
\end{equation}%
\begin{equation}%
\begin{array}
[c]{rl}%
L_{2,i}(t)^{-1}L_{3,i}(t)= & \left[  \tilde{P}_{1,i}(t)-\left(  C_{2}%
(t)^{\intercal}\tilde{P}_{1,i}(t)+\tilde{P}_{2,i}(t)\right)  ^{\intercal
}D(t)^{-1}\left(  C_{2}(t)^{\intercal}\tilde{P}_{1,i}(t)+\tilde{P}%
_{2,i}(t)\right)  \right] \\
& \cdot\left(  A_{2}(t)+B_{2}(t)P_{2,i}(t)\right)  -\left(  C_{2}%
(t)^{\intercal}\tilde{P}_{1,i}(t)+\tilde{P}_{2,i}(t)\right)  ^{\intercal
}D(t)^{-1}C_{1}(t)^{\intercal}P_{1,i}(t),
\end{array}
\label{app-6}%
\end{equation}%
\begin{equation}%
\begin{array}
[c]{rl}%
L_{2,i}(t)^{-1}L_{4,i}(t)= & \left(  C_{2}(t)^{\intercal}\tilde{P}%
_{1,i}(t)+\tilde{P}_{2,i}(t)\right)  ^{\intercal}D(t)^{-1}d(t)^{\intercal
}\tilde{P}_{3,i}(t)^{-1}-\tilde{P}_{1,i}(t)B_{2}(t)\tilde{P}_{3,i}(t)^{-1},
\end{array}
\label{app-7}%
\end{equation}%
\begin{equation}%
\begin{array}
[c]{rl}%
b(t)^{\intercal}+d(t)^{\intercal}P_{2,i}(t)= & \left(  C_{2}(t)^{\intercal
}\tilde{P}_{1,i}(t)+\tilde{P}_{2,i}(t)\right)  \left(  A_{2}(t)+B_{2}%
(t)P_{2,i}(t)\right)  +C_{1}(t)^{\intercal}P_{1,i}(t),
\end{array}
\label{app-8}%
\end{equation}%
\begin{equation}%
\begin{array}
[c]{rl}%
a(t)^{\intercal}+c(t)^{\intercal}P_{2,i}(t)= & D_{1}(t)^{\intercal}%
P_{1,i}(t)+D_{2}(t)^{\intercal}\tilde{P}_{1,i}(t)\left(  A_{2}(t)+B_{2}%
(t)P_{2,i}(t)\right)  .
\end{array}
\label{app-9}%
\end{equation}

\textbf{Proof of Theorem \ref{th-p-ex}:} The proof is divided into five steps.
The first three steps we verify the relationship between $_{i}P(\cdot)$ and
$_{i}\tilde{P}(\cdot)$, that is,%
\[%
\begin{array}
[c]{rl}%
P_{1,i}(t)= & \tilde{P}_{1,i}(t)-\tilde{P}_{2,i}(t)^{\intercal}\tilde{P}%
_{3,i}(t)^{-1}\tilde{P}_{2,i}(t),\\
P_{2,i}(t)= & -\tilde{P}_{3,i}(t)^{-1}\tilde{P}_{2,i}(t),\\
P_{3,i}(t)= & \tilde{P}_{3,i}(t)^{-1}.
\end{array}
\]
or equivalently
\begin{equation}%
\begin{array}
[c]{rl}%
\tilde{P}_{1,i}(t)= & P_{1,i}(t)+P_{2,i}(t)^{\intercal}P_{3,i}(t)^{-1}%
P_{2,i}(t),\\
\tilde{P}_{2,i}(t)= & -P_{3,i}(t)^{-1}P_{2,i}(t),\\
\tilde{P}_{3,i}(t)= & P_{3,i}(t)^{-1}.
\end{array}
\label{re-appen}%
\end{equation}

Recall that the equations satisfied by $\tilde{P}_{i}(\cdot)$ are
\begin{equation}
\left\{
\begin{array}
[c]{l}%
\dot{\tilde{P}}_{1,i}(t)+\tilde{P}_{1,i}(t)A_{1}-\tilde{P}_{2,i}%
(t)^{\intercal}A_{3}(t)+A_{1}(t)^{\intercal}\tilde{P}_{1,i}(t)-A_{3}%
(t)^{\intercal}\tilde{P}_{2,i}(t)^{\intercal}+A_{2}(t)^{\intercal}\tilde
{P}_{1,i}(t)A_{2}(t)+A_{4}(t)\\
-\left[  a(t)a_{11}(t)a(t)^{\intercal}+b(t)a_{21}(t)a(t)^{\intercal
}+a(t)a_{21}(t)^{\intercal}b(t)^{\intercal}+b(t)a_{22}(t)b(t)^{\intercal
}\right]  =0,\\
\tilde{P}_{1,i}(T)=G,
\end{array}
\right.  \label{eq-p1-til}%
\end{equation}%
\begin{equation}
\left\{
\begin{array}
[c]{l}%
\dot{\tilde{P}}_{2,i}(t)+\tilde{P}_{2,i}(t)A_{1}-\tilde{P}_{3,i}%
(t)A_{3}(t)+B_{1}(t)^{\intercal}\tilde{P}_{1,i}(t)-B_{3}(t)^{\intercal}%
\tilde{P}_{2,i}(t)+B_{2}(t)^{\intercal}\tilde{P}_{1,i}(t)A_{2}(t)\\
-\left[  c(t)a_{11}(t)a(t)^{\intercal}+d(t)a_{21}(t)a(t)^{\intercal
}+c(t)a_{21}(t)^{\intercal}b(t)^{\intercal}+d(t)a_{22}(t)b(t)^{\intercal
}\right]  =0,\\
\tilde{P}_{2,i}(T)=F,
\end{array}
\right.  \label{eq-p2-til}%
\end{equation}%
\begin{equation}
\left\{
\begin{array}
[c]{l}%
\dot{\tilde{P}}_{3,i}(t)+\tilde{P}_{2,i}(t)B_{1}-\tilde{P}_{3,i}%
(t)B_{3}(t)+B_{1}(t)^{\intercal}\tilde{P}_{2,i}(t)^{\intercal}-B_{3}%
(t)^{\intercal}\tilde{P}_{3,i}(t)+B_{2}(t)^{\intercal}\tilde{P}_{1,i}%
(t)B_{2}(t)+B_{4}(t)\\
-\left[  c(t)a_{11}(t)c(t)^{\intercal}+d(t)a_{21}(t)c(t)^{\intercal
}+c(t)a_{21}(t)^{\intercal}d(t)^{\intercal}+d(t)a_{22}(t)d(t)^{\intercal
}\right]  =0,\\
\tilde{P}_{3,i}(T)=0.
\end{array}
\right.  \label{eq-p3-til}%
\end{equation}

\textbf{Step 1:} In this step we verify $P_{3,i}(\cdot)=\tilde{P}_{3,i}%
(\cdot)^{-1}$.

$\tilde{P}_{3,i}(\cdot)^{-1}$ satisfies the following equation:
\begin{equation}%
\begin{array}
[c]{l}%
\left(  \tilde{P}_{3,i}(t)^{-1}\right)  ^{^{\prime}}-\tilde{P}_{3,i}%
(t)^{-1}\{\tilde{P}_{2,i}(t)B_{1}-\tilde{P}_{3,i}(t)B_{3}(t)+B_{1}%
(t)^{\intercal}\tilde{P}_{2,i}(t)^{\intercal}\\
-B_{3}(t)^{\intercal}\tilde{P}_{3,i}(t)+B_{2}(t)^{\intercal}\tilde{P}%
_{1,i}(t)B_{2}(t)+B_{4}(t)\\
-\left[  c(t)a_{11}(t)c(t)^{\intercal}+d(t)a_{21}(t)c(t)^{\intercal
}+c(t)a_{21}(t)^{\intercal}d(t)^{\intercal}+d(t)a_{22}(t)d(t)^{\intercal
}\right]  \}\tilde{P}_{3,i}(t)^{-1}=0.
\end{array}
\label{eq-p3-til1}%
\end{equation}
Note that $_{i}P(\cdot)$ and $P(\cdot)$ are governed by the same equations
except the terminal conditions. Putting the relation (\ref{re-appen}) into
(\ref{eq-p3-til1}) and comparing with (\ref{eq-p4}), we need to verify
\begin{equation}%
\begin{array}
[c]{l}%
-P_{3,i}(t)B_{2}(t)^{\intercal}\left(  P_{1,i}(t)+P_{2,i}(t)^{\intercal
}P_{3,i}(t)^{-1}P_{2,i}(t)\right)  B_{2}(t)P_{3,i}(t)\\
+P_{3,i}(t)\left[  c(t)a_{11}(t)c(t)^{\intercal}+d(t)a_{21}(t)c(t)^{\intercal
}+c(t)a_{21}(t)^{\intercal}d(t)^{\intercal}+d(t)a_{22}(t)d(t)^{\intercal
}\right]  P_{3,i}(t)\\
=-\left[  \left(  P_{2,i}(t)C_{1}(t)+C_{3}(t)\right)  L_{11,i}(t)-P_{3,i}%
(t)B_{2}(t)^{\intercal}L_{9,i}(t)+\left(  P_{2,i}(t)D_{1}(t)+D_{3}(t)\right)
L_{7,i}(t)\right]  .
\end{array}
\label{appendix-compare-p3}%
\end{equation}
We compare the coefficients of $a_{11}(\cdot)$ and the remainder terms on both
sides of the above equation. The coefficient of $a_{11}(\cdot)$ on the left
hand side (LHS) is%
\[
P_{3,i}(t)c(t)-P_{3,i}(t)d(t)D(t)^{-1}\left(  C_{2}(t)^{\intercal}\tilde
{P}_{1,i}(t)+\tilde{P}_{2,i}(t)\right)  D_{2}(t),
\]
and the one on the right hand side (RHS) is%
\[%
\begin{array}
[c]{l}%
\lbrack\left(  P_{2,i}(t)C_{1}(t)+C_{3}(t)\right)  D(t)^{-1}(C_{2}%
(t)^{\intercal}\tilde{P}_{1,i}(t)+\tilde{P}_{2,i}(t))D_{2}(t)\\
+P_{2,i}(t)D_{1}(t)+D_{3}(t)-P_{3,i}(t)B_{2}(t)^{\intercal}L_{2,i}%
(t)^{-1}S_{1,i}(t)D_{2}(t)].
\end{array}
\]
The remainder terms on the LHS is
\[
-P_{3,i}(t)B_{2}(t)^{\intercal}\left(  P_{1,i}(t)+P_{2,i}(t)^{\intercal
}P_{3,i}(t)^{-1}P_{2,i}(t)\right)  B_{2}(t)P_{3,i}(t)+P_{3,i}(t)d(t)D(t)^{-1}%
d(t)P_{3,i}(t),
\]
and the ones on the RHS is%
\[%
\begin{array}
[c]{l}%
\left[  \left(  P_{2,i}(t)C_{1}(t)+C_{3}(t)\right)  \left(  I_{m}%
-P_{2,i}(t)C_{2}(t)+P_{3,i}(t)C_{4}(t)\right)  ^{-1}P_{3,i}(t)C_{2}%
(t)^{\intercal}+P_{3,i}(t)B_{2}(t)^{\intercal}\right]  L_{2,i}(t)^{-1}%
L_{4,i}(t)\\
+\left(  P_{2,i}(t)C_{1}(t)+C_{3}(t)\right)  \left(  I_{m}-P_{2,i}%
(t)C_{2}(t)+P_{3,i}(t)C_{4}(t)\right)  ^{-1}\\
\cdot\left(  P_{2,i}(t)B_{2}(t)P_{3,i}(t)+P_{3,i}(t)C_{3}(t)^{\intercal
}+P_{3,i}(t)C_{1}(t)^{\intercal}P_{2,i}(t)^{\intercal}\right)  .
\end{array}
\]
By relations (\ref{app-3}), (\ref{app-4}), (\ref{app-7}) and (\ref{re-appen}),
we obtain that (\ref{appendix-compare-p3}) holds.

\textbf{Step 2}: In this step we verify $P_{2,i}(t)=-\tilde{P}_{3,i}%
(t)^{-1}\tilde{P}_{2,i}(t)$.

$-\tilde{P}_{3,i}(\cdot)^{-1}\tilde{P}_{2,i}(\cdot)$ satisfies the following
equation:
\begin{equation}%
\begin{array}
[c]{l}%
\left(  -\tilde{P}_{3,i}(t)^{-1}\tilde{P}_{2,i}(\cdot)\right)  ^{^{\prime}%
}-P_{3,i}(t)\{\tilde{P}_{2,i}(t)A_{1}-\tilde{P}_{3,i}(t)A_{3}(t)+B_{1}%
(t)^{\intercal}\tilde{P}_{1,i}(t)\\
-B_{3}(t)^{\intercal}\tilde{P}_{2,i}(t)+B_{2}(t)^{\intercal}\tilde{P}%
_{1,i}(t)A_{2}(t)\\
-\left[  c(t)a_{11}(t)a(t)^{\intercal}+d(t)a_{21}(t)a(t)^{\intercal
}+c(t)a_{21}(t)^{\intercal}b(t)^{\intercal}+d(t)a_{22}(t)b(t)^{\intercal
}\right]  \}\\
-\{P_{2,i}(t)B_{1}(t)P_{3,i}(t)-P_{2,i}(t)C_{1}(t)L_{11,i}(t)-P_{2,i}%
(t)D_{1}(t)L_{7,i}(t)\\
+P_{3,i}(t)B_{3}(t)^{\intercal}+P_{3,i}(t)B_{1}(t)^{\intercal}P_{2,i}%
(t)^{\intercal}+P_{3,i}(t)B_{2}(t)^{\intercal}L_{9,i}(t)-P_{3,i}%
(t)B_{4}(t)P_{3,i}(t)\\
+B_{3}(t)P_{3,i}(t)-C_{3}(t)L_{11,i}(t)-D_{3}(t)L_{7,i}(t)\}\tilde{P}%
_{2,i}(t)=0.
\end{array}
\label{eq-p32-til}%
\end{equation}
Putting the relation (\ref{re-appen}) into (\ref{eq-p32-til}) and comparing
with (\ref{eq-p3}), we only need to verify%
\begin{equation}%
\begin{array}
[c]{l}%
-P_{3,i}(t)\{B_{1}(t)^{\intercal}P_{2,i}(t)^{\intercal}P_{3,i}(t)^{-1}%
P_{2,i}(t)-B_{3}(t)^{\intercal}\tilde{P}_{2,i}(t)+B_{2}(t)^{\intercal}%
\tilde{P}_{1,i}(t)A_{2}(t)\\
-\left[  c(t)a_{11}(t)a(t)^{\intercal}+d(t)a_{21}(t)a(t)^{\intercal
}+c(t)a_{21}(t)^{\intercal}b(t)^{\intercal}+d(t)a_{22}(t)b(t)^{\intercal
}\right]  \}\\
-\{-P_{2,i}(t)C_{1}(t)L_{11,i}(t)-P_{2,i}(t)D_{1}(t)L_{7,i}(t)+P_{3,i}%
(t)B_{3}(t)^{\intercal}+P_{3,i}(t)B_{1}(t)^{\intercal}P_{2,i}(t)^{\intercal}\\
+P_{3,i}(t)B_{2}(t)^{\intercal}L_{9,i}(t)-C_{3}(t)L_{11,i}(t)-D_{3}%
(t)L_{7,i}(t)\}\tilde{P}_{2,i}(t)\\
=\left[  \left(  P_{2,i}(t)C_{1}(t)+C_{3}(t)\right)  L_{10,i}(t)-P_{3,i}%
(t)B_{2}(t)^{\intercal}L_{8,i}(t)+\left(  P_{2,i}(t)D_{1}(t)+D_{3}(t)\right)
L_{6,i}(t)\right]  .
\end{array}
\label{appendix-compare-p2}%
\end{equation}
The coefficient of $a_{11}(\cdot)$ on the LHS is
\[%
\begin{array}
[c]{l}%
P_{3,i}(t)[c(t)-d(t)D(t)^{-1}(C_{2}(t)^{\intercal}\tilde{P}_{1,i}(t)+\tilde
{P}_{2,i}(t))D_{2}(t)]a_{11}(t)\\
\lbrack a(t)^{\intercal}-D_{2}(t)^{\intercal}\left(  \tilde{P}_{1,i}%
(t)C_{2}(t)+\tilde{P}_{2,i}(t)^{\intercal}\right)  D(t)^{-1}b(t)^{\intercal
}]\\
+[\left(  P_{2,i}(t)C_{1}(t)+C_{3}(t)\right)  D(t)^{-1}(C_{2}(t)^{\intercal
}\tilde{P}_{1,i}(t)+\tilde{P}_{2,i}(t))D_{2}(t)-P_{2,i}(t)D_{1}(t)-D_{3}(t)\\
+P_{3,i}(t)B_{2}(t)^{\intercal}L_{2,i}(t)^{-1}S_{1,i}(t)D_{2}(t)]a_{11}%
(t)[c(t)^{\intercal}-D_{2}(t)^{\intercal}\left(  \tilde{P}_{1,i}%
(t)C_{2}(t)+\tilde{P}_{2,i}(t)^{\intercal}\right)  D(t)^{-1}d(t)^{\intercal
}]P_{2,i}(t),
\end{array}
\]
and the one on the RHS is
\[%
\begin{array}
[c]{l}%
\lbrack\left(  P_{2,i}(t)C_{1}(t)+C_{3}(t)\right)  D(t)^{-1}(C_{2}%
(t)^{\intercal}\tilde{P}_{1,i}(t)+\tilde{P}_{2,i}(t))D_{2}(t)-P_{2,i}%
(t)D_{1}(t)-D_{3}(t)\\
+P_{3,i}(t)B_{2}(t)^{\intercal}L_{2,i}(t)^{-1}S_{1,i}(t)D_{2}(t)]a_{11}%
(t)[a(t)^{\intercal}+c(t)^{\intercal}P_{2,i}(t)\\
-D_{2}(t)^{\intercal}\left(  \tilde{P}_{1,i}(t)C_{2}(t)+\tilde{P}%
_{2,i}(t)^{\intercal}\right)  D(t)^{-1}(b(t)^{\intercal}+d(t)^{\intercal
}P_{2,i}(t)].
\end{array}
\]
By calculation, we need to prove
\[%
\begin{array}
[c]{l}%
P_{3,i}(t)[c(t)-d(t)D(t)^{-1}(C_{2}(t)^{\intercal}\tilde{P}_{1,i}(t)+\tilde
{P}_{2,i}(t))D_{2}(t)]\\
=[\left(  P_{2,i}(t)C_{1}(t)+C_{3}(t)\right)  D(t)^{-1}(C_{2}(t)^{\intercal
}\tilde{P}_{1,i}(t)+\tilde{P}_{2,i}(t))D_{2}(t)\\
\text{ \ }+P_{2,i}(t)D_{1}(t)+D_{3}(t)-P_{3,i}(t)B_{2}(t)^{\intercal}%
L_{2,i}(t)^{-1}S_{1,i}(t)D_{2}(t)],
\end{array}
\]
which has already been verified in Step 1. The remainder terms on the LHS is%
\[%
\begin{array}
[c]{l}%
-P_{3,i}(t)B_{2}(t)^{\intercal}\tilde{P}_{1,i}(t)A_{2}(t)+P_{3,i}%
(t)d(t)D(t)^{-1}b(t)^{\intercal}\\
+\left(  P_{2,i}(t)C_{1}(t)+C_{3}(t)\right)  D(t)^{-1}d(t)P_{3,i}(t)\tilde
{P}_{2,i}(t)-P_{3,i}(t)B_{2}(t)^{\intercal}L_{2,i}(t)^{-1}L_{4,i}(t)\tilde
{P}_{2,i}(t)
\end{array}
\]
and the ones on the RHS is%
\[
-\left(  P_{2,i}(t)C_{1}(t)+C_{3}(t)\right)  D(t)^{-1}\left(  b(t)^{\intercal
}+d(t)^{\intercal}P_{2,i}(t)\right)  -P_{3,i}(t)B_{2}(t)^{\intercal}%
L_{2,i}(t)^{-1}L_{3,i}(t).
\]
By relations (\ref{app-3}), (\ref{app-6}), (\ref{app-7}), (\ref{app-8}) and
(\ref{re-appen}), we obtain that (\ref{appendix-compare-p2}) holds.

\textbf{Step 3}: In this step we verify $P_{1,i}(t)=\tilde{P}_{1,i}%
(t)-\tilde{P}_{2,i}(t)^{\intercal}\tilde{P}_{3,i}(t)^{-1}\tilde{P}_{2,i}(t)$.

Since $P_{2,i}(t)=-\tilde{P}_{3,i}(t)^{-1}\tilde{P}_{2,i}(t)$, $P_{3,i}%
(t)=\tilde{P}_{3,i}(t)^{-1}$, we have
\[
\tilde{P}_{1,i}(t)-\tilde{P}_{2,i}(t)^{\intercal}\tilde{P}_{3,i}(t)^{-1}%
\tilde{P}_{2,i}(t)=\tilde{P}_{1,i}(t)+\tilde{P}_{2,i}(t)^{\intercal}%
P_{2,i}(t).
\]
Deriving on both sides of the above equation,
\begin{equation}%
\begin{array}
[c]{l}%
\left(  \tilde{P}_{1,i}(t)-\tilde{P}_{2,i}(t)^{\intercal}\tilde{P}%
_{3,i}(t)^{-1}\tilde{P}_{2,i}(t)\right)  ^{^{\prime}}+\tilde{P}_{1,i}%
(t)A_{1}-\tilde{P}_{2,i}(t)^{\intercal}A_{3}(t)+A_{1}(t)^{\intercal}\tilde
{P}_{1,i}(t)\\
-A_{3}(t)^{\intercal}\tilde{P}_{2,i}(t)^{\intercal}+A_{2}(t)^{\intercal}%
\tilde{P}_{1,i}(t)A_{2}(t)+A_{4}(t)\\
-\left[  a(t)a_{11}(t)a(t)^{\intercal}+b(t)a_{21}(t)a(t)^{\intercal
}+a(t)a_{21}(t)^{\intercal}b(t)^{\intercal}+b(t)a_{22}(t)b(t)^{\intercal
}\right] \\
+\{\tilde{P}_{2,i}(t)A_{1}-\tilde{P}_{3,i}(t)A_{3}(t)+B_{1}(t)^{\intercal
}\tilde{P}_{1,i}(t)-B_{3}(t)^{\intercal}\tilde{P}_{2,i}(t)+B_{2}%
(t)^{\intercal}\tilde{P}_{1,i}(t)A_{2}(t)\\
-\left[  c(t)a_{11}(t)a(t)^{\intercal}+d(t)a_{21}(t)a(t)^{\intercal
}+c(t)a_{21}(t)^{\intercal}b(t)^{\intercal}+d(t)a_{22}(t)b(t)^{\intercal
}\right]  \}^{\intercal}P_{2,i}(t)\\
+\tilde{P}_{2,i}(t)^{\intercal}\{P_{2,i}(t)A_{1}(t)+P_{2,i}(t)B_{1}%
(t)P_{2,i}(t)+P_{2,i}(t)C_{1}(t)L_{10,i}(t)+P_{2,i}(t)D_{1}(t)L_{6,i}(t)\\
-P_{3,i}(t)B_{1}(t)^{\intercal}P_{1,i}(t)-P_{3,i}(t)B_{2}(t)^{\intercal
}L_{8,i}(t)-P_{3,i}(t)B_{4}(t)P_{2,i}(t)\\
+A_{3}(t)+B_{3}(t)P_{2,i}(t)+C_{3}(t)L_{10,i}(t)+D_{3}(t)L_{6,i}(t)\}=0.
\end{array}
\label{eq-p321-til}%
\end{equation}
Putting the relation (\ref{re-appen}) into (\ref{eq-p321-til}) and comparing
with (\ref{eq-p1}), we need to verify%
\[%
\begin{array}
[c]{l}%
P_{2,i}(t)^{\intercal}P_{3,i}(t)^{-1}P_{2,i}(t)A_{1}(t)+A_{1}(t)^{\intercal
}P_{2,i}(t)^{\intercal}P_{3,i}(t)^{-1}P_{2,i}(t)-\tilde{P}_{2,i}%
(t)^{\intercal}A_{3}(t)-A_{3}(t)^{\intercal}\tilde{P}_{2,i}(t)\\
+A_{2}(t)^{\intercal}\tilde{P}_{1,i}(t)A_{2}(t)-\left[  a(t)a_{11}%
(t)a(t)^{\intercal}+b(t)a_{21}(t)a(t)^{\intercal}+a(t)a_{21}(t)^{\intercal
}b(t)^{\intercal}+b(t)a_{22}(t)b(t)^{\intercal}\right] \\
+\{\tilde{P}_{2,i}(t)A_{1}-\tilde{P}_{3,i}(t)A_{3}(t)+B_{1}(t)^{\intercal
}P_{2,i}(t)^{\intercal}P_{3,i}(t)^{-1}P_{2,i}(t)-B_{3}(t)^{\intercal}\tilde
{P}_{2,i}(t)+B_{2}(t)^{\intercal}\tilde{P}_{1,i}(t)A_{2}(t)\\
-\left[  c(t)a_{11}(t)a(t)^{\intercal}+d(t)a_{21}(t)a(t)^{\intercal
}+c(t)a_{21}(t)^{\intercal}b(t)^{\intercal}+d(t)a_{22}(t)b(t)^{\intercal
}\right]  \}^{\intercal}P_{2,i}(t)\\
+\tilde{P}_{2,i}(t)^{\intercal}\{P_{2,i}(t)A_{1}(t)+P_{2,i}(t)B_{1}%
(t)P_{2,i}(t)+P_{2,i}(t)C_{1}(t)L_{10,i}(t)+P_{2,i}(t)D_{1}(t)L_{6,i}(t)\\
-P_{3,i}(t)B_{2}(t)^{\intercal}L_{8,i}(t)+A_{3}(t)+B_{3}(t)P_{2,i}%
(t)+C_{3}(t)L_{10,i}(t)+D_{3}(t)L_{6,i}(t)\}\\
=P_{1,i}(t)C_{1}(t)L_{10,i}(t)+P_{1,i}(t)D_{1}(t)L_{6,i}(t)+(A_{2}%
(t)+B_{2}(t)P_{2,i}(t))^{\intercal}L_{8,i}(t),
\end{array}
\]
that is
\[%
\begin{array}
[c]{l}%
A_{2}(t)^{\intercal}\left(  P_{1,i}(t)+P_{2,i}(t)^{\intercal}P_{3,i}%
(t)^{-1}P_{2,i}(t)\right)  (A_{2}(t)+B_{2}(t)P_{2,i}(t))\\
-\tilde{P}_{2,i}(t)^{\intercal}\left(  P_{2,i}(t)C_{1}(t)+C_{3}(t)\right)
D(t)^{-1}\left(  b(t)^{\intercal}+d(t)^{\intercal}P_{2,i}(t)\right) \\
-\tilde{P}_{2,i}(t)^{\intercal}P_{3,i}(t)B_{2}(t)^{\intercal}L_{2,i}%
(t)^{-1}L_{3,i}(t)\\
-\left[  a(t)a_{11}(t)a(t)^{\intercal}+b(t)a_{21}(t)a(t)^{\intercal
}+a(t)a_{21}(t)^{\intercal}b(t)^{\intercal}+b(t)a_{22}(t)b(t)^{\intercal
}\right] \\
-\left[  c(t)a_{11}(t)a(t)^{\intercal}+d(t)a_{21}(t)a(t)^{\intercal
}+c(t)a_{21}(t)^{\intercal}b(t)^{\intercal}+d(t)a_{22}(t)b(t)^{\intercal
}\right]  ^{\intercal}P_{2,i}(t)\\
-\tilde{P}_{2,i}(t)^{\intercal}\{\left(  P_{2,i}(t)C_{1}(t)+C_{3}(t)\right)
[a_{21}(t)\left(  a(t)^{\intercal}+c(t)^{\intercal}P_{2,i}(t)\right)
+D(t)^{-1}\left(  C_{2}(t)^{\intercal}\tilde{P}_{1,i}(t)+\tilde{P}%
_{2,i}(t)\right) \\
\cdot D_{2}(t)a_{11}(t)D_{2}(t)^{\intercal}\left(  \tilde{P}_{1,i}%
(t)C_{2}(t)+\tilde{P}_{2,i}(t)^{\intercal}\right)  D(t)^{-1}\left(
b(t)^{\intercal}+d(t)^{\intercal}P_{2,i}(t)\right)  ]\}\\
-\tilde{P}_{2,i}(t)^{\intercal}\{\left(  P_{2,i}(t)D_{1}(t)+D_{3}%
(t)-P_{3,i}(t)B_{2}(t)^{\intercal}L_{2,i}(t)^{-1}S_{1}(t)D_{2}(t)\right) \\
\cdot\left[  a_{11}(t)\left(  a(t)^{\intercal}+c(t)^{\intercal}P_{2,i}%
(t)\right)  +a_{21}(t)^{\intercal}\left(  b(t)^{\intercal}+d(t)^{\intercal
}P_{2,i}(t)\right)  \right]  \}\\
=-\left\{  P_{1,i}(t)C_{1}(t)D(t)^{-1}\left(  C_{2}(t)^{\intercal}\tilde
{P}_{1,i}(t)+\tilde{P}_{2,i}(t)\right)  D_{2}(t)-P_{1,i}(t)D_{1}(t)\right. \\
\left.  -(A_{2}(t)+B_{2}(t)P_{2,i}(t))^{\intercal}L_{2,i}(t)^{-1}S_{1}%
(t)D_{2}(t)\displaystyle\right\} \\
a_{11}(t)[D_{2}(t)^{\intercal}\left(  C_{2}(t)^{\intercal}\tilde{P}%
_{1,i}(t)+\tilde{P}_{2,i}(t)\right)  ^{\intercal}D(t)^{-1}\left(
b(t)^{\intercal}+d(t)^{\intercal}P_{2,i}(t)\right)  -a(t)^{\intercal
}-c(t)^{\intercal}P_{2,i}(t)]\\
+P_{1,i}(t)C_{1}(t)D(t)^{-1}\left(  b(t)^{\intercal}+d(t)^{\intercal}%
P_{2,i}(t)\right)  +(A_{2}(t)+B_{2}(t)P_{2,i}(t))^{\intercal}L_{2,i}%
(t)^{-1}L_{3,i}(t).
\end{array}
\]
The coefficient of $a_{11}(\cdot)$ on the LHS is
\[%
\begin{array}
[c]{l}%
-\left\{  b(t)D(t)^{-1}\left(  C_{2}(t)^{\intercal}\tilde{P}_{1,i}%
(t)+\tilde{P}_{2,i}(t)\right)  D_{2}(t)-a(t)+\tilde{P}_{2,i}(t)^{\intercal
}\right. \\
\lbrack\left(  P_{2,i}(t)C_{1}(t)+C_{3}(t)\right)  D(t)^{-1}\left(
C_{2}(t)^{\intercal}\tilde{P}_{1,i}(t)+\tilde{P}_{2,i}(t)\right)
D_{2}(t)-P_{2,i}(t)D_{1}(t)\\
\left.  -D_{3}(t)+P_{3,i}(t)B_{2}(t)^{\intercal}L_{2,i}(t)^{-1}S_{1,i}%
(t)D_{2}(t)]\displaystyle\right\} \\
a_{11}(t)[D_{2}(t)^{\intercal}\left(  C_{2}(t)^{\intercal}\tilde{P}%
_{1,i}(t)+\tilde{P}_{2,i}(t)\right)  ^{\intercal}D(t)^{-1}\left(
b(t)^{\intercal}+d(t)^{\intercal}P_{2,i}(t)\right)  -a(t)^{\intercal
}-c(t)^{\intercal}P_{2,i}(t)],
\end{array}
\]
and the one on the RHS is
\[%
\begin{array}
[c]{l}%
-\displaystyle\left\{  P_{1,i}(t)C_{1}(t)D(t)^{-1}\left(  C_{2}(t)^{\intercal
}\tilde{P}_{1,i}(t)+\tilde{P}_{2,i}(t)\right)  D_{2}(t)-P_{1,i}(t)D_{1}%
(t)\right. \\
\left.  -(A_{2}(t)+B_{2}(t)P_{2,i}(t))^{\intercal}L_{2,i}(t)^{-1}%
S_{1,i}(t)D_{2}(t)\displaystyle\right\} \\
a_{11}(t)[D_{2}(t)^{\intercal}\left(  C_{2}(t)^{\intercal}\tilde{P}%
_{1,i}(t)+\tilde{P}_{2,i}(t)\right)  ^{\intercal}D(t)^{-1}\left(
b(t)^{\intercal}+d(t)^{\intercal}P_{2,i}(t)\right)  -a(t)^{\intercal
}-c(t)^{\intercal}P_{2,i}(t)].
\end{array}
\]
By (\ref{app-4}), (\ref{app-8}), (\ref{app-9}), we derive%
\[%
\begin{array}
[c]{l}%
b(t)D(t)^{-1}\left(  C_{2}(t)^{\intercal}\tilde{P}_{1,i}(t)+\tilde{P}%
_{2,i}(t)\right)  D_{2}(t)-a(t)+\tilde{P}_{2,i}(t)[\left(  P_{2,i}%
(t)C_{1}(t)+C_{3}(t)\right)  D(t)^{-1}\\
\cdot\left(  C_{2}(t)^{\intercal}\tilde{P}_{1,i}(t)+\tilde{P}_{2,i}(t)\right)
D_{2}(t)-P_{2,i}(t)D_{1}(t)-D_{3}(t)+P_{3,i}(t)B_{2}(t)^{\intercal}%
L_{2,i}(t)^{-1}S_{1,i}(t)D_{2}(t)]\\
=P_{1,i}(t)C_{1}(t)D(t)^{-1}\left(  C_{2}(t)^{\intercal}\tilde{P}%
_{1,i}(t)+\tilde{P}_{2,i}(t)\right)  D_{2}(t)-P_{1,i}(t)D_{1}(t)\\
\text{ \ \ }-(A_{2}(t)+B_{2}(t)P_{2,i}(t))^{\intercal}L_{2,i}(t)^{-1}%
S_{1,i}(t)D_{2}(t)\\
=[D_{2}(t)^{\intercal}\left(  C_{2}(t)^{\intercal}\tilde{P}_{1,i}(t)+\tilde
{P}_{2,i}(t)\right)  ^{\intercal}D(t)^{-1}\left(  b(t)^{\intercal
}+d(t)^{\intercal}P_{2,i}(t)\right)  -a(t)^{\intercal}-c(t)^{\intercal}%
P_{2,i}(t)]^{\intercal}.
\end{array}
\]
The remainder terms on the LHS is
\[%
\begin{array}
[c]{l}%
A_{2}(t)^{\intercal}\left(  P_{1,i}(t)+P_{2,i}(t)^{\intercal}P_{3,i}%
(t)^{-1}P_{2,i}(t)\right)  (A_{2}(t)+B_{2}(t)P_{2,i}(t))+P_{2,i}%
(t)^{\intercal}B_{2}(t)^{\intercal}L_{2,i}(t)^{-1}L_{3,i}(t)\\
-\left[  b(t)+\tilde{P}_{2,i}(t)\left(  P_{2,i}(t)C_{1}(t)+C_{3}(t)\right)
\right]  D(t)^{-1}\left(  b(t)^{\intercal}+d(t)^{\intercal}P_{2,i}(t)\right)
,
\end{array}
\]
and the ones on the RHS is
\[
-P_{1,i}(t)C_{1}(t)D(t)^{-1}\left(  b(t)^{\intercal}+d(t)^{\intercal}%
P_{2,i}(t)\right)  +(A_{2}(t)+B_{2}(t)P_{2,i}(t))^{\intercal}L_{2,i}%
(t)^{-1}L_{3,i}(t).
\]
By the definition of $b(t)$ and
\[%
\begin{array}
[c]{rc}%
A_{2}(t)^{\intercal}L_{2,i}(t)^{-1}L_{3,i}(t)= & A_{2}(t)^{\intercal}\left(
P_{1,i}(t)+P_{2,i}(t)^{\intercal}P_{3,i}(t)^{-1}P_{2,i}(t)\right)
(A_{2}(t)+B_{2}(t)P_{2,i}(t))\\
& -A_{2}(t)^{\intercal}\left(  C_{2}(t)^{\intercal}\tilde{P}_{1,i}%
(t)+\tilde{P}_{2,i}(t)\right)  ^{\intercal}D(t)^{-1}\left(  b(t)^{\intercal
}+d(t)^{\intercal}P_{2,i}(t)\right)  ,
\end{array}
\]
the remainder terms on both sides are consistent.

In the following two steps, we verify the relationship between $_{i}%
\varphi(\cdot)$ and $_{i}\tilde{\varphi}(\cdot)$, that is,
\[%
\begin{array}
[c]{cc}%
\tilde{\varphi}_{1,i}(t)=-P_{1,i}(t)^{-1}\varphi_{1,i}(t), & \tilde{\varphi
}_{2,i}(t)=\varphi_{2,i}(t)-P_{2,i}(t)P_{1,i}(t)^{-1}\varphi_{1,i}(t),\\
\tilde{v}_{1,i}(t)=-P_{1,i}(t)^{-1}v_{1,i}(t), & \tilde{v}_{2,i}%
(t)=v_{2,i}(t)-P_{2,i}(t)P_{1,i}(t)^{-1}v_{1,i}(t).
\end{array}
\]
Since $d_{i}\tilde{\varphi}(t)=-$ $_{i}\tilde{\gamma}(t)dt+$ $_{i}\tilde
{v}(t)dB(t)$, the above relations are equivalent to
\[
\tilde{\gamma}_{1,i}(t)=-\left\{  P_{1,i}(t)^{-1}\dot{P}_{1,i}(t)P_{1,i}%
(t)\varphi_{1,i}(t)+P_{1,i}(t)^{-1}\gamma_{1,i}(t)\right\}  ,
\]
and
\[%
\begin{array}
[c]{rl}%
\tilde{\gamma}_{2,i}(t)= & \gamma_{2,i}(t)+\dot{P}_{2,i}(t)P_{1,i}%
(t)^{-1}\varphi_{1,i}(t)-P_{2,i}(t)P_{1,i}(t)^{-1}\dot{P}_{1,i}(t)P_{1,i}%
(t)\varphi_{1,i}(t)-P_{2,i}(t)P_{1,i}(t)^{-1}\gamma_{1,i}(t)\\
= & \gamma_{2,i}(t)+\dot{P}_{2,i}(t)P_{1,i}(t)^{-1}\varphi_{1,i}%
(t)+P_{2,i}(t)\tilde{\gamma}_{1,i}(t).
\end{array}
\]
In the completion-of-squares technique, $_{i}\tilde{\gamma}(t)$ satisfies
\[%
\begin{array}
[c]{l}%
\left(  \tilde{B}(t)^{\intercal}\text{ }_{i}\tilde{P}(t)+\tilde{D}%
(t)^{\intercal}\text{ }_{i}\tilde{P}(t)\tilde{C}(t)\right)  ^{\intercal
}(\tilde{R}(t)+\tilde{D}(t)^{\intercal}\text{ }_{i}\tilde{P}(t)\tilde
{D}(t))^{-1}\left(  \tilde{B}(t)^{\intercal}\text{ }_{i}\tilde{P}(t)\text{
}_{i}\tilde{\varphi}(t)-\tilde{D}(t)^{\intercal}\text{ }_{i}\tilde{P}(t)\text{
}_{i}\tilde{v}(t)\right) \\
\text{ \ \ \ }-\tilde{A}(t)^{\intercal}\text{ }_{i}\tilde{P}(t)\text{ }%
_{i}\tilde{\varphi}(t)-\text{ }_{i}\dot{\tilde{P}}(t)\text{ }_{i}%
\tilde{\varphi}(t)-\text{ }_{i}\tilde{P}(t)\text{ }_{i}\tilde{\gamma
}(t)-\tilde{C}(t)^{\intercal}\text{ }_{i}\tilde{P}(t)\text{ }_{i}\tilde
{v}(t)=0.
\end{array}
\]
Then
\[
\text{ }_{i}\tilde{P}(t)\text{ }_{i}\tilde{\gamma}(t)=\left(
\begin{array}
[c]{c}%
-\left[  \dot{P}_{1,i}(t)P_{1,i}(t)^{-1}\varphi_{1,i}(t)+\gamma_{1,i}%
(t)\right]  +P_{2,i}(t)^{\intercal}\left(  \gamma_{2,i}(t)+\dot{P}%
_{2,i}(t)P_{1,i}(t)^{-1}\varphi_{1,i}(t)\right) \\
\tilde{P}_{3,i}(t)\left(  \gamma_{2,i}(t)+\dot{P}_{2,i}(t)P_{1,i}%
(t)^{-1}\varphi_{1,i}(t)\right)
\end{array}
\right)  ,
\]
and we need to verify the following two equalities:%
\begin{equation}%
\begin{array}
[c]{l}%
-\left[  \left(  a(t)a_{11}(t)+b(t)a_{21}(t)\right)  c(t)^{\intercal}+\left(
a(t)a_{21}(t)^{\intercal}+b(t)a_{22}(t)\right)  d(t)^{\intercal}\right]
\varphi_{2,i}(t)-a(t)S_{3,i}(t)-b(t)S_{5,i}(t)\\
+A_{1}(t)^{\intercal}\left(  \tilde{P}_{1,i}(t)\tilde{\varphi}_{1,i}%
(t)+\tilde{P}_{2,i}(t)^{\intercal}\tilde{\varphi}_{2,i}(t)\right)
-A_{3}(t)^{\intercal}\left(  \tilde{P}_{2,i}(t)\tilde{\varphi}_{1,i}%
(t)+\tilde{P}_{3,i}(t)\tilde{\varphi}_{2,i}(t)\right)  +\dot{\tilde{P}}%
_{1,i}(t)\tilde{\varphi}_{1,i}(t)\\
+\dot{\tilde{P}}_{2,i}(t)\tilde{\varphi}_{2,i}(t)+A_{2}(t)^{\intercal}\left(
\tilde{P}_{1,i}(t)\tilde{v}_{1,i}(t)+\tilde{P}_{2,i}(t)^{\intercal}\tilde
{v}_{2,i}(t)\right) \\
=\left(  \tilde{P}_{2,i}(t)^{\intercal}\dot{P}_{2,i}(t)P_{1,i}(t)^{-1}-\dot
{P}_{1,i}(t)P_{1,i}(t)^{-1}\right)  \varphi_{1,i}(t)+\left(  \tilde{P}%
_{2,i}(t)^{\intercal}P_{2,i}(t)-P_{1,i}(t)\right) \\
\cdot\left[  B_{1}(t)\varphi_{2,i}(t)+C_{1}(t)S_{5,i}(t)+D_{1}(t)S_{3,i}%
(t)\right] \\
+\tilde{P}_{2,i}(t)^{\intercal}\left[  B_{3}(t)\varphi_{2,i}(t)+C_{3}%
(t)S_{5,i}(t)+D_{3}(t)S_{3,i}(t)\right]  -\left(  A_{1}(t)^{\intercal}%
\tilde{\varphi}_{1,i}(t)+A_{1}(t)^{\intercal}S_{4,i}(t)\right)
\end{array}
\label{eq-app-124-1}%
\end{equation}
and%
\begin{equation}%
\begin{array}
[c]{l}%
-\left[  \left(  c(t)a_{11}(t)+d(t)a_{21}(t)\right)  c(t)^{\intercal}+\left(
c(t)a_{21}(t)^{\intercal}+d(t)a_{22}(t)\right)  d(t)^{\intercal}\right]
\varphi_{2}(t)-c(t)S_{3,i}(t)-d(t)S_{5,i}(t)\\
+B_{1}(t)^{\intercal}\left(  \tilde{P}_{1,i}(t)\tilde{\varphi}_{1,i}%
(t)+\tilde{P}_{2,i}(t)^{\intercal}\tilde{\varphi}_{2,i}(t)\right)
-B_{3}(t)^{\intercal}\left(  \tilde{P}_{2,i}(t)\tilde{\varphi}_{1,i}%
(t)+\tilde{P}_{3,i}(t)\tilde{\varphi}_{2,i}(t)\right)  +\dot{\tilde{P}}%
_{2,i}(t)\tilde{\varphi}_{1,i}(t)\\
+\dot{\tilde{P}}_{3,i}(t)\tilde{\varphi}_{2,i}(t)+B_{2}(t)^{\intercal}\left(
\tilde{P}_{1,i}(t)\tilde{v}_{1,i}(t)+\tilde{P}_{2,i}(t)^{\intercal}\tilde
{v}_{2,i}(t)\right) \\
=\tilde{P}_{3,i}(t)\left\{  P_{2,i}(t)\left[  B_{1}(t)\varphi_{2,i}%
(t)+C_{1}(t)S_{5,i}(t)+D_{1}(t)S_{3,i}(t)\right]  \right. \\
-P_{3,i}(t)\left[  B_{1}(t)^{\intercal}\varphi_{1,i}(t)+B_{2}(t)^{\intercal
}S_{4,i}(t)+B_{4}(t)\varphi_{2,i}(t)\right] \\
\left.  +B_{3}(t)\varphi_{2,i}(t)+C_{3}(t)S_{5,i}(t)+D_{3}(t)S_{3,i}%
(t)\right\}  +\tilde{P}_{3,i}(t)\dot{P}_{2,i}(t)P_{1,i}(t)^{-1}\varphi
_{1,i}(t).
\end{array}
\label{eq-app-124-2}%
\end{equation}

\textbf{Step 4: }Verification of (\ref{eq-app-124-2}):

The equation (\ref{eq-app-124-2}) can be simplified to
\[%
\begin{array}
[c]{l}%
-\left[  \left(  c(t)a_{11}(t)+d(t)a_{21}(t)\right)  c(t)^{\intercal}+\left(
c(t)a_{21}(t)^{\intercal}+d(t)a_{22}(t)\right)  d(t)^{\intercal}\right]
\varphi_{2,i}(t)\\
+B_{1}(t)^{\intercal}\left(  \tilde{P}_{1,i}(t)\tilde{\varphi}_{1,i}%
(t)+\tilde{P}_{2,i}(t)^{\intercal}\tilde{\varphi}_{2,i}(t)\right)
-B_{3}(t)^{\intercal}\left(  \tilde{P}_{2,i}(t)\tilde{\varphi}_{1,i}%
(t)+\tilde{P}_{3,i}(t)\tilde{\varphi}_{2,i}(t)\right)  +\dot{\tilde{P}}%
_{2,i}(t)\tilde{\varphi}_{1,i}(t)\\
+\dot{\tilde{P}}_{3,i}(t)\tilde{\varphi}_{,i2}(t)+B_{2}(t)^{\intercal}\left(
\tilde{P}_{1,i}(t)\tilde{v}_{1,i}(t)+\tilde{P}_{2,i}(t)^{\intercal}\tilde
{v}_{2,i}(t)\right) \\
=\tilde{P}_{3,i}(t)\left\{  P_{2,i}(t)B_{1}(t)\varphi_{2,i}(t)-P_{3,i}%
(t)\left[  B_{1}(t)^{\intercal}\varphi_{1,i}(t)+B_{4}(t)\varphi_{2,i}%
(t)\right]  +B_{3}(t)\varphi_{2,i}(t)\right\} \\
+\tilde{P}_{3,i}(t)\dot{P}_{2,i}(t)P_{1,i}(t)^{-1}\varphi_{1,i}(t)-B_{2}%
(t)^{\intercal}P_{1,i}(t)B_{2}(t)\varphi_{2,i}(t)-B_{2}(t)^{\intercal}%
P_{2,i}(t)^{\intercal}C_{1}(t)^{\intercal}\varphi_{1,i}(t)-B_{2}%
(t)^{\intercal}v_{1,i}(t)\\
-B_{2}(t)^{\intercal}P_{2,i}(t)^{\intercal}P_{3,i}(t)^{-1}\left[
P_{2,i}(t)B_{2}(t)\varphi_{2,i}(t)+v_{2,i}(t)-P_{3,i}(t)C_{1}(t)^{\intercal
}\varphi_{1,i}(t)\right]  .
\end{array}
\]
Then we compare the coefficients of $\varphi_{1,i}(\cdot)$, $\varphi
_{2,i}(\cdot)$, $v_{1,i}(\cdot)$ and $v_{2,i}(\cdot)$ on both sides of the
above equation. The coefficient of $\varphi_{1,i}(\cdot)$ on the LHS is%
\[%
\begin{array}
[c]{l}%
-B_{1}(t)^{\intercal}\left(  \tilde{P}_{1,i}(t)P_{1,i}(t)^{-1}+\tilde{P}%
_{2,i}(t)^{\intercal}P_{2,i}(t)P_{1,i}(t)^{-1}\right)  +B_{3}(t)^{\intercal
}\left(  \tilde{P}_{2,i}(t)P_{1,i}(t)^{-1}+\tilde{P}_{3,i}(t)P_{2,i}%
(t)P_{1,i}(t)^{-1}\right) \\
-\dot{\tilde{P}}_{2,i}(t)P_{1,i}(t)^{-1}+\dot{\tilde{P}}_{3,i}(t)P_{2,i}%
(t)P_{1,i}(t)^{-1}%
\end{array}
\]
and the one on the RHS is%
\[
-B_{1}(t)^{\intercal}+\tilde{P}_{3,i}(t)\dot{P}_{2,i}(t)P_{1,i}(t)^{-1}%
-B_{2}(t)^{\intercal}P_{2,i}(t)^{\intercal}C_{1}(t)^{\intercal}+B_{2}%
(t)^{\intercal}P_{2,i}(t)^{\intercal}P_{3,i}(t)^{-1}P_{3,i}(t)C_{1}%
(t)^{\intercal}.
\]
The coefficient of $\varphi_{2,i}(\cdot)$ on the LHS is%
\[%
\begin{array}
[c]{l}%
-\left[  \left(  c(t)a_{11}(t)+d(t)a_{21}(t)\right)  c(t)^{\intercal}+\left(
c(t)a_{21}(t)^{\intercal}+d(t)a_{22}(t)\right)  d(t)^{\intercal}\right] \\
+B_{1}(t)^{\intercal}\tilde{P}_{2,i}(t)^{\intercal}-B_{3}(t)^{\intercal}%
\tilde{P}_{3,i}(t)+\dot{\tilde{P}}_{3,i}(t)
\end{array}
\]
and the one on the RHS is%
\[
\tilde{P}_{3,i}(t)P_{2,i}(t)B_{1}(t)-B_{4}(t)+\tilde{P}_{3,i}(t)B_{3}%
(t)-B_{2}(t)^{\intercal}P_{1,i}(t)B_{2}(t)-B_{2}(t)^{\intercal}P_{2,i}%
(t)^{\intercal}P_{3,i}(t)^{-1}P_{2,i}(t)B_{2}(t).
\]
The coefficient of $v_{1,i}(\cdot)$ on the LHS is
\[
-B_{2}(t)^{\intercal}\left(  \tilde{P}_{1,i}(t)P_{1,i}(t)^{-1}+\tilde{P}%
_{2,i}(t)^{\intercal}P_{2,i}(t)P_{1,i}(t)^{-1}\right)
\]
and the one on the RHS is $-B_{2}(t)^{\intercal}.$ The coefficient of
$v_{2,i}(\cdot)$ on the LHS is $B_{2}(t)^{\intercal}\tilde{P}_{2,i}%
(t)^{\intercal}$ and the one on the RHS is $-B_{2}(t)^{\intercal}%
P_{2,i}(t)^{\intercal}P_{3,i}(t)^{-1}.$ By the notations in (\ref{def-nota-1})
and (\ref{re-appen}), we obtain (\ref{eq-app-124-2}) holds.

\textbf{Step 5: }Verification of (\ref{eq-app-124-1}):

The equation (\ref{eq-app-124-1}) can be simplified to
\[%
\begin{array}
[c]{l}%
-\left[  \left(  a(t)a_{11}(t)+b(t)a_{21}(t)\right)  c(t)^{\intercal}+\left(
a(t)a_{21}(t)^{\intercal}+b(t)a_{22}(t)\right)  d(t)^{\intercal}\right]
\varphi_{2,i}(t)\\
+A_{1}(t)^{\intercal}\left(  \tilde{P}_{1,i}(t)\tilde{\varphi}_{1,i}%
(t)+\tilde{P}_{2,i}(t)^{\intercal}\tilde{\varphi}_{2,i}(t)\right)
-A_{3}(t)^{\intercal}\left(  \tilde{P}_{2,i}(t)\tilde{\varphi}_{1,i}%
(t)+\tilde{P}_{3,i}(t)\tilde{\varphi}_{2,i}(t)\right)  +\dot{\tilde{P}}%
_{1,i}(t)\tilde{\varphi}_{1,i}(t)\\
+\dot{\tilde{P}}_{2,i}(t)^{\intercal}\tilde{\varphi}_{2,i}(t)+A_{2}%
(t)^{\intercal}\left(  \tilde{P}_{1,i}(t)\tilde{v}_{1,i}(t)+\tilde{P}%
_{2,i}(t)^{\intercal}\tilde{v}_{2,i}(t)\right) \\
=\left(  \tilde{P}_{2,i}(t)^{\intercal}\dot{P}_{2,i}(t)P_{1,i}(t)^{-1}-\dot
{P}_{1,i}(t)P_{1,i}(t)^{-1}\right)  \varphi_{1,i}(t)+\left(  \tilde{P}%
_{2,i}(t)^{\intercal}P_{2,i}(t)-P_{1,i}(t)\right)  B_{1}(t)\varphi_{2,i}(t)\\
+\tilde{P}_{2,i}(t)^{\intercal}B_{3}(t)\varphi_{2,i}(t)-A_{1}(t)^{\intercal
}\tilde{\varphi}_{1,i}(t)+A_{2}(t)^{\intercal}\left\{  \left(  -P_{2,i}%
(t)^{\intercal}P_{3,i}(t)^{-1}\right)  \right. \\
\cdot\left(  P_{2,i}(t)B_{2}(t)\varphi_{2,i}(t)-P_{3,i}(t)C_{1}(t)^{\intercal
}\varphi_{1,i}(t)+v_{2,i}(t)\right)  \left.  -P_{1,i}(t)B_{2}(t)\varphi
_{2,i}(t)-P_{2,i}(t)^{\intercal}C_{1}(t)^{\intercal}\varphi_{1,i}%
(t)-v_{1,i}(t)\right\}  .
\end{array}
\]
By comparing the coefficients of $\varphi_{1,i}(\cdot)$, $\varphi_{2,i}%
(\cdot)$, $v_{1,i}(\cdot)$ and $v_{2,i}(\cdot)$ on both sides of the above
equation, we deduce that (\ref{eq-app-124-1}) holds. $\blacksquare$

\subsection{Proof of Lemma \ref{le-con-equ}}

{From (\ref{eq-Y-m}) and (\ref{eq-con-th}), the optimal control $(\bar
{u}(\cdot),\bar{Z}(\cdot))$ has the following form%
\[%
\begin{array}
[c]{rl}%
\bar{u}(t)= & -\left[  a_{11}(t)\left(  a(t)^{\intercal}+c(t)^{\intercal
}P_{2,i}(t)\right)  +a_{21}(t)^{\intercal}\left(  b(t)^{\intercal
}+d(t)^{\intercal}P_{2,i}(t)\right)  \right]  \bar{X}(t)\\
& +\left(  a_{11}(t)c(t)^{\intercal}+a_{21}(t)^{\intercal}d(t)^{\intercal
}\right)  P_{3,i}(t)h(t)+a_{11}(t)\left[  \tilde{b}(t)\varphi_{2,i}%
(t)-\tilde{a}(t)\tilde{\varphi}_{1,i}(t)\right. \\
& \left.  -\tilde{b}(t)\tilde{\varphi}_{2,i}(t)+\bar{a}(t)B_{2}(t)\varphi
_{2,i}(t)-\bar{a}(t)\tilde{v}_{1,i}(t)-\bar{b}(t)\tilde{v}_{2,i}(t)\right] \\
& +a_{21}(t)^{\intercal}\left[  \tilde{d}(t)\varphi_{2,i}(t)-\tilde
{c}(t)\tilde{\varphi}_{1,i}(t)-\tilde{d}(t)\tilde{\varphi}_{2,i}(t)+\bar
{c}(t)B_{2}(t)\varphi_{2,i}(t)-\bar{c}(t)\tilde{v}_{1,i}(t)-\bar{d}%
(t)\tilde{v}_{2,i}(t)\right]  ,\\
\bar{Z}(t)= & -\left[  a_{21}(t)\left(  a(t)^{\intercal}+c(t)^{\intercal
}P_{2,i}(t)\right)  +a_{22}(t)\left(  b(t)^{\intercal}+d(t)^{\intercal}%
P_{2,i}(t)\right)  \right]  \bar{X}(t)\\
& +\left(  a_{21}(t)c(t)^{\intercal}+a_{22}(t)d(t)^{\intercal}\right)
P_{3,i}(t)h(t)+a_{21}(t)\left[  \tilde{b}(t)\varphi_{2,i}(t)-\tilde
{a}(t)\tilde{\varphi}_{1,i}(t)\right. \\
& \left.  -\tilde{b}(t)\tilde{\varphi}_{2,i}(t)+\bar{a}(t)B_{2}(t)\varphi
_{2,i}(t)-\bar{a}(t)\tilde{v}_{1,i}(t)-\bar{b}(t)\tilde{v}_{2,i}(t)\right] \\
& +a_{22}(t)\left[  \tilde{d}(t)\varphi_{2,i}(t)-\tilde{c}(t)\tilde{\varphi
}_{1,i}(t)-\tilde{d}(t)\tilde{\varphi}_{2,i}(t)+\bar{c}(t)B_{2}(t)\varphi
_{2,i}(t)-\bar{c}(t)\tilde{v}_{1,i}(t)-\bar{d}(t)\tilde{v}_{2,i}(t)\right]  .
\end{array}
\]
}

The following relations can be verified directly:%
\begin{equation}
\left(  \bar{a}(t)+\bar{b}(t)P_{2,i}(t)\right)  P_{1,i}(t)^{-1}=D_{2}%
(t)^{\intercal}, \label{app-2-1}%
\end{equation}%
\begin{equation}
\left(  \bar{c}(t)+\bar{d}(t)P_{2,i}(t)\right)  P_{1,i}(t)^{-1}=C_{2}%
(t)^{\intercal}, \label{app-2-2}%
\end{equation}%
\begin{equation}
\left(  \tilde{a}(t)+\tilde{b}(t)P_{2,i}(t)\right)  P_{1,i}(t)^{-1}%
=D_{1}(t)^{\intercal}, \label{app-2-3}%
\end{equation}%
\begin{equation}
\left(  \tilde{c}(t)+\tilde{d}(t)P_{2,i}(t)\right)  P_{1,i}(t)^{-1}%
=C_{1}(t)^{\intercal}, \label{app-2-4}%
\end{equation}%
\begin{equation}%
\begin{array}
[c]{l}%
L_{1,i}(t)^{-1}\left(  P_{2,i}(t)-P_{3,i}(t)C_{2}(t)^{\intercal}%
L_{2,i}(t)^{-1}S_{1,i}(t)\right) \\
=-D(t)^{-1}\left(  C_{2}(t)^{\intercal}\tilde{P}_{1,i}(t)+\tilde{P}%
_{2,i}(t)\right)  ,
\end{array}
\label{app-2-5}%
\end{equation}%
\begin{equation}%
\begin{array}
[c]{l}%
L_{2,i}(t)^{-1}\left[  P_{2,i}(t)^{\intercal}-(P_{1,i}(t)C_{2}(t)+P_{2,i}%
(t)^{\intercal}C_{4}(t))L_{1,i}(t)^{-1}P_{3,i}(t)\right] \\
=-\left(  C_{2}(t)^{\intercal}\tilde{P}_{1,i}(t)+\tilde{P}_{2,i}(t)\right)
^{\intercal}D(t)^{-1},
\end{array}
\label{app-2-6}%
\end{equation}%
\begin{equation}%
\begin{array}
[c]{l}%
\left(  \tilde{P}_{1,i}(t)C_{2}(t)+\tilde{P}_{2,i}(t)^{\intercal}-P_{_{1,i}%
}(t)C_{2}(t)-P_{2,i}(t)^{\intercal}C_{4}(t)\right)  L_{1,i}(t)^{-1}\\
=-P_{2,i}(t)^{\intercal}P_{3,i}(t)^{-1},
\end{array}
\label{app-2-7}%
\end{equation}%
\begin{equation}%
\begin{array}
[c]{l}%
\left[  I_{m}-C_{2}(t)^{\intercal}P_{2,i}(t)^{\intercal}+C_{2}(t)^{\intercal
}\left(  P_{1,i}(t)C_{2}(t)+P_{2,i}(t)^{\intercal}C_{4}(t)\right)
L_{1,i}(t)^{-1}P_{3,i}(t)\right] \\
=D(t)L_{1,i}(t)^{-1}P_{3,i}(t).
\end{array}
\label{app-2-8}%
\end{equation}

By notations in (\ref{def-nota-1}), (\ref{app-5}), (\ref{app-8}) and
(\ref{app-9}), it can be verified that%
\begin{equation}%
\begin{array}
[c]{rl}%
L_{6,i}(t)= & -\left[  a_{11}(t)\left(  a(t)^{\intercal}+c(t)^{\intercal
}P_{2,i}(t)\right)  +a_{21}(t)^{\intercal}\left(  b(t)^{\intercal
}+d(t)^{\intercal}P_{2,i}(t)\right)  \right]  ,\\
L_{3,i}(t)= & \left(  a_{11}(t)c(t)^{\intercal}+a_{21}(t)^{\intercal
}d(t)^{\intercal}\right)  P_{3,i}(t),\\
L_{10,i}(t)= & -\left[  a_{21}(t)\left(  a(t)^{\intercal}+c(t)^{\intercal
}P_{2,i}(t)\right)  +a_{22}(t)\left(  b(t)^{\intercal}+d(t)^{\intercal}%
P_{2,i}(t)\right)  \right]  ,\\
L_{11,i}(t)= & \left(  a_{21}(t)c(t)^{\intercal}+a_{22}(t)d(t)^{\intercal
}\right)  P_{3,i}(t).
\end{array}
\label{eq-re1}%
\end{equation}

Before proving Lemma \ref{le-con-equ}, we give the following lemma:

\begin{lemma}
\label{le-re-s}Under the same assumptions as Theorem \ref{th-p-ex}, for $i\geq
i_{0}$, we have%
\begin{equation}%
\begin{array}
[c]{rl}%
S_{3,i}(t)= & a_{11}(t)\left[  \tilde{b}(t)\varphi_{2,i}(t)-\tilde{a}%
(t)\tilde{\varphi}_{1,i}(t)-\tilde{b}(t)\tilde{\varphi}_{2,i}(t)+\bar
{a}(t)B_{2}(t)\varphi_{2,i}(t)-\bar{a}(t)\tilde{v}_{1,i}(t)-\bar{b}%
(t)\tilde{v}_{2,i}(t)\right] \\
& +a_{21}(t)^{\intercal}\left[  \tilde{d}(t)\varphi_{2,i}(t)-\tilde
{c}(t)\tilde{\varphi}_{1,i}(t)-\tilde{d}(t)\tilde{\varphi}_{2,i}(t)+\bar
{c}(t)B_{2}(t)\varphi_{2,i}(t)-\bar{c}(t)\tilde{v}_{1,i}(t)-\bar{d}%
(t)\tilde{v}_{2,i}(t)\right]  ,\\
S_{5,i}(t)= & a_{21}(t)\left[  \tilde{b}(t)\varphi_{2,i}(t)-\tilde{a}%
(t)\tilde{\varphi}_{1,i}(t)-\tilde{b}(t)\tilde{\varphi}_{2,i}(t)+\bar
{a}(t)B_{2}(t)\varphi_{2,i}(t)-\bar{a}(t)\tilde{v}_{1,i}(t)-\bar{b}%
(t)\tilde{v}_{2,i}(t)\right] \\
& +a_{22}(t)\left[  \tilde{d}(t)\varphi_{2,i}(t)-\tilde{c}(t)\tilde{\varphi
}_{1,i}(t)-\tilde{d}(t)\tilde{\varphi}_{2,i}(t)+\bar{c}(t)B_{2}(t)\varphi
_{2,i}(t)-\bar{c}(t)\tilde{v}_{1,i}(t)-\bar{d}(t)\tilde{v}_{2,i}(t)\right]  .
\end{array}
\label{eq-s-app}%
\end{equation}

\end{lemma}

\begin{proof}
We first prove the equality for $S_{3,i}(\cdot)$. Compare the coefficients of
$\varphi_{1,i}(\cdot)$, $\varphi_{2,i}(\cdot)$, $v_{1,i}(\cdot)$ and
$v_{2,i}(\cdot)$ for $S_{3,i}(\cdot)$ in (\ref{eq-s-app}) with the ones in
(\ref{def-L1}). By the notations in (\ref{def-nota-1})-(\ref{app-9}) and
(\ref{app-2-1})-(\ref{app-2-8}), we obtain the equality for $S_{3,i}(\cdot)$ holds.

Then, we prove the equality for $S_{5,i}(\cdot)$. Putting $S_{2,i}(\cdot)$,
$S_{3,i}(\cdot)$ and $S_{4,i}(\cdot)$ into $S_{5,i}(\cdot)$, the equality for
$S_{5,i}(\cdot)$ becomes%
\begin{equation}%
\begin{array}
[c]{l}%
S_{5,i}(t)\\
=-\left\{  \left[  P_{3,i}(t)C_{1}(t)^{\intercal}+\left(  P_{2,i}%
(t)-P_{3,i}(t)C_{2}(t)^{\intercal}L_{2,i}(t)^{-1}S_{1,i}(t)\right)
D_{2}(t)a_{11}(t)D_{1}(t)^{\intercal}\right]  \varphi_{1,i}(t)-P_{2,i}%
(t)B_{2}(t)\varphi_{2,i}(t)\right. \\
-v_{2,i}(t)+\left[  P_{3,i}(t)C_{2}(t)^{\intercal}L_{2,i}(t)^{-1}+\left(
P_{2,i}(t)-P_{3,i}(t)C_{2}(t)^{\intercal}L_{2,i}(t)^{-1}S_{1,i}(t)\right)
D_{2}(t)a_{11}(t)D_{2}(t)^{\intercal}L_{2,i}(t)^{-1}\right]  S_{2,i}(t).
\end{array}
\label{eq-s3}%
\end{equation}
Compare the coefficients of $\varphi_{1,i}(\cdot)$, $\varphi_{2,i}(\cdot)$,
$v_{1,i}(\cdot)$ and $v_{2,i}(\cdot)$ for $S_{5,i}(\cdot)$ in (\ref{eq-s-app})
with the ones in (\ref{eq-s3}). By the notations in (\ref{def-nota-1}%
)-(\ref{app-9}) and (\ref{app-2-1})-(\ref{app-2-8}), we obtain the equality
for $S_{5,i}(\cdot)$ holds.
\end{proof}

\textbf{Proof of Lemma \ref{le-con-equ}:} By the notations in (\ref{app-def1})
and (\ref{app-def2}), we have
\[%
\begin{array}
[c]{c}%
\left(  \tilde{R}(t)+\tilde{D}(t)^{\intercal}{}_{i}\tilde{P}(t)\tilde
{D}(t)\right)  ^{-1}\left(  \tilde{B}(t)^{\intercal}{}_{i}\tilde{P}%
(t)+\tilde{D}(t)^{\intercal}{}_{i}\tilde{P}(t)\tilde{C}(t)\right) \\
=\left(
\begin{array}
[c]{ccc}%
a_{11}(t)a(t)^{\intercal}+a_{21}(t)^{\intercal}b(t)^{\intercal} &  &
a_{11}(t)c(t)^{\intercal}+a_{21}(t)^{\intercal}d(t)^{\intercal}\\
a_{21}(t)a(t)^{\intercal}+a_{22}(t)^{\intercal}b(t)^{\intercal} &  &
a_{21}(t)c(t)^{\intercal}+a_{22}(t)^{\intercal}d(t)^{\intercal}%
\end{array}
\right)  .
\end{array}
\]
By the notations in (\ref{eq-re1}), we obtain the first relation in Lemma
\ref{le-con-equ} holds.

From the relationship of $_{i}P(\cdot)$, $_{i}\tilde{P}(\cdot)$, $_{i}%
\varphi(\cdot)$, and $_{i}\tilde{\varphi}(\cdot)$, we have{\small
\[%
\begin{array}
[c]{l}%
_{i}\tilde{\varphi}(t)^{\intercal}\text{ }_{i}\tilde{P}(t)\text{ }_{i}%
\tilde{\varphi}(t)\\
=\left[  \left(
\begin{array}
[c]{cc}%
-P_{1,i}(t)^{-1} & 0\\
-P_{2,i}(t)P_{1,i}(t)^{-1} & I_{m}%
\end{array}
\right)  \left(
\begin{array}
[c]{c}%
\varphi_{1,i}(t)\\
\varphi_{2,i}(t)
\end{array}
\right)  \right]  ^{\intercal}\left(
\begin{array}
[c]{cc}%
\tilde{P}_{1,i}(t) & \tilde{P}_{2,i}(t)^{\intercal}\\
\tilde{P}_{2,i}(t) & \tilde{P}_{3,i}(t)
\end{array}
\right)  \left[  \left(
\begin{array}
[c]{cc}%
-P_{1,i}(t)^{-1} & 0\\
-P_{2,i}(t)P_{1,i}(t)^{-1} & I_{m}%
\end{array}
\right)  \left(
\begin{array}
[c]{c}%
\varphi_{1,i}(t)\\
\varphi_{2,i}(t)
\end{array}
\right)  \right] \\
=\varphi_{1,i}(t)^{\intercal}P_{1,i}(t)^{-1}\varphi_{1,i}(t)+\varphi
_{2,i}(t)^{\intercal}P_{3,i}(t)^{-1}\varphi_{2,i}(t).
\end{array}
\]
} Due to (\ref{eq-re1}) and Lemma \ref{le-re-s}, we obtain
\[%
\begin{array}
[c]{l}%
\left(  \tilde{R}(t)+\tilde{D}(t)^{\intercal}{}_{i}\tilde{P}(t)\tilde
{D}(t)\right)  ^{-1}\left(  \tilde{B}(t)^{\intercal}{}_{i}\tilde{P}(t)\text{
}_{i}\tilde{\varphi}(t)+\tilde{D}(t)^{\intercal}{}_{i}\tilde{P}(t)\tilde
{C}(t)_{i}\tilde{v}(t)\right)  ^{\intercal}\\
=\left(
\begin{array}
[c]{c}%
S_{3,i}(t)\\
S_{5,i}(t)
\end{array}
\right)  +\left(
\begin{array}
[c]{c}%
\left(  a_{11}(t)c(t)^{\intercal}+a_{21}(t)^{\intercal}d(t)^{\intercal
}\right)  \varphi_{2,i}(t)\\
\left(  a_{21}(t)c(t)^{\intercal}+a_{22}(t)d(t)^{\intercal}\right)
\varphi_{2,i}(t)
\end{array}
\right) \\
=\left(
\begin{array}
[c]{c}%
S_{3,i}(t)+L_{7,i}(t)P_{3,i}(t)^{-1}\varphi_{2,i}(t)\\
S_{5,i}(t)+L_{11,i}(t)P_{3,i}(t)^{-1}\varphi_{2,i}(t)
\end{array}
\right)  .
\end{array}
\]
This completes the proof. $\blacksquare$

\subsection{Proof of Lemma \ref{le-gamma}\label{app-le-gamma}}

It can be verified that%
\[%
\begin{array}
[c]{l}%
\left(  P_{2,i}(t)C_{2}(t)-I_{m}\right)  L_{1,i}(t)^{-1}P_{3,i}(t)\\
=\left(  P_{2,i}(t)C_{2}(t)-I_{m}-P_{3,i}(t)C_{4}(t)\right)  L_{1,i}%
(t)^{-1}P_{3,i}(t)+P_{3,i}(t)C_{4}(t)L_{1,i}(t)^{-1}P_{3,i}(t),\\
L_{2,i}(t)^{-1}\left[  P_{2,i}(t)^{\intercal}-\left(  P_{2,i}(t)^{\intercal
}C_{4}(t)+P_{1,i}(t)C_{2}(t)\right)  L_{1,i}(t)^{-1}P_{3,i}(t)\right] \\
=-\left(  C_{2}(t)^{\intercal}\tilde{P}_{1,i}(t)+\tilde{P}_{2,i}(t)\right)
^{\intercal}D(t)^{-1},
\end{array}
\]%
\[%
\begin{array}
[c]{rl}%
{\small S}_{3,i}{\small (t)=} & {\small a}_{11}{\small (t)\lambda}%
_{1,i}{\small (t),}\\
{\small L}_{7,i}{\small (t)=} & {\small a}_{11}{\small (t)\lambda}%
_{2,i}{\small (t),}\\
{\small S}_{4,i}{\small (t)=} & {\small L}_{2,i}{\small (t)}^{-1}%
{\small S}_{1,i}{\small (t)D}_{2}{\small (t)a}_{11}{\small (t)\lambda}%
_{1,i}{\small (t),}%
\end{array}
\]

\[%
\begin{array}
[c]{rl}%
{\small S}_{5,i}{\small (t)=} & {\small -D(t)}^{-1}\left(  C_{2}%
(t)^{\intercal}\tilde{P}_{1,i}(t)+\tilde{P}_{2,i}(t)\right)  {\small D}%
_{2}{\small (t)a}_{11}{\small (t)}\lambda_{1,i}(t)\\
{\small =} & {\small -D(t)}^{-1}\left(  C_{2}(t)^{\intercal}\tilde{P}%
_{1,i}(t)+\tilde{P}_{2,i}(t)\right)  {\small D}_{2}{\small (t)a}%
_{11}{\small (t)}\\
& \cdot\left[  \tilde{b}(t)\varphi_{2,i}(t)-\tilde{a}(t)\tilde{\varphi}%
_{1,i}(t)-\tilde{b}(t)\tilde{\varphi}_{2,i}(t)+\bar{a}(t)B_{2}(t)\varphi
_{2,i}(t)-\bar{a}(t)\tilde{v}_{1,i}(t)-\bar{b}(t)\tilde{v}_{2,i}(t)\right] \\
& {\small +D(t)}^{-1}\left(  C_{2}(t)^{\intercal}\tilde{P}_{1,i}(t)+\tilde
{P}_{2,i}(t)\right)  {\small D}_{2}{\small (t)a}_{11}{\small (t)D}%
_{2}{\small (t)}^{\intercal}\left(  \tilde{P}_{1,i}(t)C_{2}(t)+\tilde{P}%
_{2,i}(t)^{\intercal}\right)  {\small D(t)}^{-1}\\
& {\small \cdot}\left[  \tilde{d}(t)\varphi_{2,i}(t)-\tilde{c}(t)\tilde
{\varphi}_{1,i}(t)-\tilde{d}(t)\tilde{\varphi}_{2,i}(t)+\bar{c}(t)B_{2}%
(t)\varphi_{2,i}(t)-\bar{c}(t)\tilde{v}_{1,i}(t)-\bar{d}(t)\tilde{v}%
_{2,i}(t)\right] \\
& {\small +\beta}_{i}{\small (t),}\\
{\small L}_{11,i}{\small (t)=} & {\small -D(t)}^{-1}\left(  C_{2}%
(t)^{\intercal}\tilde{P}_{1,i}(t)+\tilde{P}_{2,i}(t)\right)  {\small D}%
_{2}{\small (t)a}_{11}{\small (t)c(t)}^{\intercal}{\small P}_{3,i}%
{\small (t)}\\
& {\small +}\left[  D(t)^{-1}\left(  C_{2}(t)^{\intercal}\tilde{P}%
_{1,i}(t)+\tilde{P}_{2,i}(t)\right)  D_{2}(t)a_{11}(t)D_{2}(t)^{\intercal
}\left(  \tilde{P}_{1,i}(t)C_{2}(t)+\tilde{P}_{2,i}(t)^{\intercal}\right)
D(t)^{-1}\right] \\
& \cdot{\small d(t)}^{\intercal}{\small P}_{3,i}{\small (t)+\tilde{\beta}}%
_{i}{\small (t),}\\
{\small L}_{2,i}{\small (t)}^{-1}{\small S}_{2,i}{\small (t)=} &
{\small \beta}_{5,i}{\small (t)\varphi}_{1,i}{\small (t)+\beta}_{6,i}%
{\small (t)\varphi}_{2,i}{\small (t)+\beta}_{7,i}{\small (t)v}_{1,i}%
{\small (t)+\beta}_{8,i}{\small (t)v}_{2,i}{\small (t),}%
\end{array}
\]
{ }where
\[
\lambda_{1,i}(t)=-\left[  D_{1}(t)^{\intercal}\varphi_{1,i}(t)+D_{2}%
(t)^{\intercal}L_{2,i}(t)^{-1}S_{2,i}(t)\right]  ,
\]%
\[
\lambda_{2,i}(t)=-\left[  D_{1}(t)^{\intercal}P_{2,i}(t)^{\intercal}%
+D_{2}(t)^{\intercal}L_{2,i}(t)^{-1}L_{4,i}(t)+D_{3}(t)^{\intercal}\right]  ,
\]%
\[%
\begin{array}
[c]{rl}%
\beta_{i}(t)= & {\small D(t)}^{-1}\left[  \tilde{d}(t)\varphi_{2,i}%
(t)-\tilde{c}(t)\tilde{\varphi}_{1,i}(t)-\tilde{d}(t)\tilde{\varphi}%
_{2,i}(t)+\bar{c}(t)B_{2}(t)\varphi_{2,i}(t)-\bar{c}(t)\tilde{v}_{1,i}%
(t)-\bar{d}(t)\tilde{v}_{2,i}(t)\right] \\
= & \beta_{1,i}(t)\varphi_{1,i}(t)+\beta_{2,i}(t)\varphi_{2,i}(t)+\beta
_{3,i}(t)v_{1,i}(t)+\beta_{4,i}(t)v_{2,i}(t),
\end{array}
\]%
\[%
\begin{array}
[c]{l}%
\beta_{1,i}(t)={\small D(t)}^{-1}\left[  \tilde{c}(t)P_{1,i}(t)^{-1}+\tilde
{d}(t)P_{2,i}(t)P_{1,i}(t)^{-1}\right]  ,\\
\beta_{2,i}(t)={\small D(t)}^{-1}\bar{c}(t)B_{2}(t),\\
\beta_{3,i}(t)={\small D(t)}^{-1}\left[  \bar{c}(t)P_{1,i}(t)^{-1}+\bar
{d}(t)P_{2,i}(t)P_{1,i}(t)^{-1}\right]  ,\\
\beta_{4,i}(t)=-{\small D(t)}^{-1}\bar{d}(t),
\end{array}
\]%
\[
\tilde{\beta}_{i}(t)=D(t)^{-1}d(t)^{\intercal}P_{3,i}(t),
\]%
\[%
\begin{array}
[c]{l}%
{\small \beta}_{5,i}{\small (t)=L}_{2,i}{\small (t)}^{-1}\left[
P_{2,i}(t)^{\intercal}C_{1}(t)^{\intercal}-(P_{1,i}(t)C_{2}(t)+P_{2,i}%
(t)^{\intercal}C_{4}(t))L_{1,i}(t)^{-1}P_{3,i}(t)C_{1}(t)^{\intercal}\right]
,\\
{\small \beta}_{6,i}{\small (t)=L}_{2,i}{\small (t)}^{-1}\left[
P_{1,i}(t)B_{2}(t)+(P_{1,i}(t)C_{2}(t)+P_{2,i}(t)^{\intercal}C_{4}%
(t))L_{1,i}(t)^{-1}P_{2,i}(t)B_{2}(t)\right]  ,\\
{\small \beta}_{7,i}{\small (t)=L}_{2,i}{\small (t)}^{-1},\\
{\small \beta}_{8,i}{\small (t)=L}_{2,i}{\small (t)}^{-1}(P_{1,i}%
(t)C_{2}(t)+P_{2,i}(t)^{\intercal}C_{4}(t))L_{1,i}(t)^{-1}.
\end{array}
\]

{\textbf{Proof of Lemma \ref{le-gamma}:} }By Lemma \ref{le-con-equ}, and the
relations between $_{i}P(\cdot)$, $_{i}\tilde{P}(\cdot)$, $_{i}\varphi(\cdot)$
and $_{i}\tilde{\varphi}(\cdot)$, we have%
\[%
\begin{array}
[c]{l}%
M_{5,i}(t)\\
=-\tilde{\gamma}(t)^{\intercal}\text{ }_{i}\tilde{P}(t)\text{ }_{i}%
\tilde{\varphi}(t)-\text{ }_{i}\tilde{\varphi}(t)^{\intercal}\text{ }%
_{i}\tilde{P}(t)\text{ }_{i}\tilde{\gamma}(t)+\text{ }_{i}\tilde{\varphi
}(t)^{\intercal}\text{ }_{i}\dot{\tilde{P}}(t)\text{ }_{i}\tilde{\varphi
}(t)+\text{ }_{i}\tilde{v}(t)^{\intercal}\text{ }_{i}\tilde{P}(t)\text{ }%
_{i}\tilde{v}(t)\\
\text{ }-\left(  \tilde{B}(t)^{\intercal}\text{ }_{i}\tilde{P}(t)\text{ }%
_{i}\tilde{\varphi}(t)+\tilde{D}(t)^{\intercal}\text{ }_{i}\tilde{P}(t)\text{
}_{i}\tilde{v}(t)\right)  ^{\intercal}(\tilde{R}(t)+\tilde{D}(t)^{\intercal
}\text{ }_{i}\tilde{P}(t)\tilde{D}(t))^{-1}\\
\text{ }\cdot\left(  \tilde{B}(t)^{\intercal}\text{ }_{i}\tilde{P}(t)\text{
}_{i}\tilde{\varphi}(t)+\tilde{D}(t)^{\intercal}\text{ }_{i}\tilde{P}(t)\text{
}_{i}\tilde{v}(t)\right) \\
=-\gamma_{1,i}(t)^{\intercal}P_{1,i}(t)^{-1}\varphi_{1,i}(t)-\varphi
_{1,i}(t)^{\intercal}P_{1,i}(t)^{-1}\dot{P}_{1,i}(t)P_{1,i}(t)^{-1}%
\varphi_{1,i}(t)-\varphi_{1,i}(t)^{\intercal}P_{1,i}(t)^{-1}\gamma_{1,i}(t)\\
\text{ }+v_{1,i}(t)^{\intercal}P_{1,i}(t)^{-1}v_{1,i}(t)-\gamma_{2,i}%
(t)^{\intercal}P_{3,i}(t)^{-1}\varphi_{2,i}(t)-\varphi_{2,i}(t)^{\intercal
}P_{3,i}(t)^{-1}\dot{P}_{3,i}(t)P_{3,i}(t)^{-1}\varphi_{2,i}(t)\\
\text{ }-\varphi_{2,i}(t)^{\intercal}P_{3,i}(t)^{-1}\gamma_{2,i}%
(t)+v_{2,i}(t)^{\intercal}P_{3,i}(t)^{-1}v_{2,i}(t)\\
\text{ }+\left\{  \left[  \varphi_{1,i}(t)^{\intercal}D_{1}(t)+\varphi
_{2,i}(t)^{\intercal}P_{3,i}(t)^{-1}\left(  P_{2,i}(t)D_{1}(t)+D_{3}%
(t)+v_{1,i}(t)^{\intercal}D_{2}(t)+v_{2,i}(t)^{\intercal}P_{3,i}%
(t)^{-1}P_{2,i}(t)D_{2}(t)\right)  \right]  \right. \\
\text{ }\cdot\left(  S_{3,i}(t)+L_{3,i}(t)\tilde{P}_{3,i}(t)^{-1}\varphi
_{2,i}(t)\right)  +\left[  \varphi_{1,i}(t)^{\intercal}C_{3}(t)+\varphi
_{2,i}(t)^{\intercal}P_{3,i}(t)^{-1}\left(  P_{2,i}(t)C_{1}(t)+C_{3}%
(t)\right)  \right. \\
\text{ }\left.  \left.  +v_{1,i}(t)^{\intercal}C_{2}(t)+v_{2,i}(t)^{\intercal
}P_{3,i}(t)^{-1}P_{2,i}(t)C_{2}(t)\right]  \left(  S_{5,i}(t)+L_{11,i}%
(t)\tilde{P}_{3,i}(t)^{-1}\varphi_{2,i}(t)\right)  \right\}  .
\end{array}
\]
It can be verified that the following two equalities hold
\begin{equation}%
\begin{array}
[c]{l}%
\varphi_{2,i}(t)^{\intercal}\{P_{3,i}(t)^{-1}(P_{2,i}(t)D_{1}(t)+D_{3}%
(t))L_{7,i}(t)P_{3,i}(t)^{-1}+P_{3,i}(t)^{-1}(P_{2,i}(t)C_{1}(t)\\
+C_{3}(t))L_{11,i}(t)P_{3,i}(t)^{-1}-P_{3,i}(t)^{-1}\dot{P}_{3,i}%
(t)P_{3,i}(t)^{-1}\}\varphi_{2,i}(t)\\
=\varphi_{2,i}(t)^{\intercal}P_{3,i}(t)^{-1}\left[  P_{2,i}(t)B_{1}%
(t)P_{3,i}(t)+P_{3,i}(t)B_{3}(t)^{\intercal}+P_{3,i}(t)B_{1}(t)^{\intercal
}P_{2,i}(t)^{\intercal}\right. \\
\left.  +P_{3,i}(t)B_{2}(t)^{\intercal}L_{9,i}(t)-P_{3,i}(t)B_{4}%
(t)P_{3,i}(t)+B_{3}(t)P_{3,i}(t)\right]  P_{3,i}(t)^{-1}\varphi_{2,i}(t),
\end{array}
\label{eq-022301}%
\end{equation}%
\begin{equation}%
\begin{array}
[c]{l}%
\left\{  \left[  \left(  \varphi_{1,i}(t)^{\intercal}+\varphi_{2,i}%
(t)^{\intercal}P_{3,i}(t)^{-1}P_{2,i}(t)\right)  D_{1}(t)+\left(
v_{1,i}(t)^{\intercal}+v_{2,i}(t)^{\intercal}P_{3,i}(t)^{-1}P_{2,i}(t)\right)
D_{2}(t)+\varphi_{2,i}(t)^{\intercal}P_{3,i}(t)^{-1}D_{3}(t)\right]  \right.
\\
\text{ \ }-\left[  \left(  \varphi_{1,i}(t)^{\intercal}+\varphi_{2,i}%
(t)^{\intercal}P_{3,i}(t)^{-1}P_{2,i}(t)\right)  C_{1}(t)+\left(
v_{1,i}(t)^{\intercal}+v_{2,i}(t)^{\intercal}P_{3,i}(t)^{-1}P_{2,i}(t)\right)
C_{2}(t)\right. \\
\text{ \ }\left.  +\varphi_{2,i}(t)^{\intercal}P_{3,i}(t)^{-1}C_{3}%
(t)-v_{2,i}(t)^{\intercal}P_{3,i}(t)^{-1}\right]  \left.  D(t)^{-1}\left(
C_{2}(t)^{\intercal}\tilde{P}_{1,i}(t)+\tilde{P}_{2,i}(t)\right)
D_{2}(t)\right\}  a_{11}(t)\lambda_{1,i}(t)\\
+\left\{  \left[  \varphi_{1,i}(t)^{\intercal}D_{1}(t)+v_{1,i}(t)^{\intercal
}D_{2}(t)+v_{2,i}(t)^{\intercal}P_{3,i}(t)^{-1}P_{2,i}(t)D_{2}(t)\right]
\right. \\
-\left[  \varphi_{1,i}(t)^{\intercal}C_{1}(t)+v_{1,i}(t)^{\intercal}%
C_{2}(t)+v_{2,i}(t)^{\intercal}P_{3,i}(t)^{-1}P_{2,i}(t)C_{2}(t)-v_{2,i}%
(t)^{\intercal}P_{3,i}(t)^{-1}\right] \\
\left.  \cdot D(t)^{-1}\left(  C_{2}(t)^{\intercal}\tilde{P}_{1,i}%
(t)+\tilde{P}_{2,i}(t)\right)  D_{2}(t)\right\}  a_{11}(t)\lambda
_{2,i}(t)P_{3,i}(t)^{-1}\varphi_{2,i}(t)\\
+\lambda_{1,i}(t)^{\intercal}a_{11}(t)\left\{  \left(  P_{2,i}(t)C_{1}%
(t)+C_{3}(t)\right)  D(t)^{-1}\left(  C_{2}(t)^{\intercal}\tilde{P}%
_{1,i}(t)+\tilde{P}_{2,i}(t)\right)  D_{2}(t)-\left(  P_{2,i}(t)D_{1}%
(t)+D_{3}(t)\right)  \right. \\
\left.  +P_{3,i}(t)B_{2}(t)^{\intercal}L_{2,i}(t)^{-1}S_{1,i}(t)D_{2}%
(t)\right\}  ^{\intercal}P_{3,i}(t)^{-1}\varphi_{2,i}(t)\\
+\varphi_{2,i}(t)^{\intercal}P_{3,i}(t)^{-1}\left\{  \left(  P_{2,i}%
(t)C_{1}(t)+C_{3}(t)\right)  D(t)^{-1}\left(  C_{2}(t)^{\intercal}\tilde
{P}_{1,i}(t)+\tilde{P}_{2,i}(t)\right)  D_{2}(t)-\left(  P_{2,i}%
(t)D_{1}(t)+D_{3}(t)\right)  \right. \\
\left.  +P_{3,i}(t)B_{2}(t)^{\intercal}L_{2,i}(t)^{-1}S_{1,i}(t)D_{2}%
(t)\right\}  a_{11}(t)\lambda_{1,i}(t)\\
=-\lambda_{1,i}(t)^{\intercal}a_{11}(t)\lambda_{1,i}(t)-\varphi_{2,i}%
(t)^{\intercal}B_{2}(t)^{\intercal}L_{2,i}(t)^{-1}S_{1,i}(t)D_{2}%
(t)a_{11}(t)\lambda_{2,i}(t)P_{3,i}(t)^{-1}\varphi_{2,i}(t).
\end{array}
\label{eq-022302}%
\end{equation}
By (\ref{eq-022301}) and (\ref{eq-022302}), we have
\[%
\begin{array}
[c]{l}%
M_{5,i}(t)=-\lambda_{1,i}(t)^{\intercal}a_{11}(t)\lambda_{1,i}(t)-\gamma
_{1,i}(t)^{\intercal}P_{1,i}(t)^{-1}\varphi_{1,i}(t)-\varphi_{1,i}%
(t)^{\intercal}P_{1,i}(t)^{-1}\gamma_{1,i}(t)+v_{1,i}(t)^{\intercal}%
P_{1,i}(t)^{-1}v_{1,i}(t)\\
-\varphi_{1,i}(t)^{\intercal}P_{1,i}(t)^{-1}\dot{P}_{1,i}(t)P_{1,i}%
(t)^{-1}\varphi_{1,i}(t)-B_{2}(t)^{\intercal}L_{2,i}(t)^{-1}S_{1,i}%
(t)B_{2}(t)+C_{1}(t)\beta_{1,i}(t)+C_{2}(t)\beta_{3,i}(t)\\
+C_{2}(t)^{\intercal}L_{2,i}(t)^{-1}\left(  P_{1,i}(t)C_{2}(t)+P_{2,i}%
(t)^{\intercal}C_{4}(t)\right)  L_{1,i}(t)^{-1}+C_{4}(t)L_{1,i}(t)^{-1}\\
\cdot\left[  I_{m}-P_{3,i}(t)C_{2}(t)^{\intercal}L_{2,i}(t)^{-1}\right.
\left.  \left(  P_{1,i}(t)C_{2}(t)+P_{2,i}(t)^{\intercal}C_{4}(t)\right)
L_{1,i}(t)^{-1}\right]  +B_{1}(t)^{\intercal}+B_{2}(t)^{\intercal}\beta
_{5,i}(t)\\
+C_{1}(t)\beta_{2,i}(t)+B_{1}(t)+\beta_{5,i}(t)^{\intercal}B_{2}%
(t)+C_{1}(t)L_{1,i}(t)^{-1}\left\{  P_{2,i}(t)B_{2}(t)+P_{3,i}(t)C_{2}%
(t)^{\intercal}L_{2,i}(t)^{-1}\right. \\
\left[  -P_{1,i}(t)-\left(  P_{1,i}(t)C_{2}(t)+P_{2,i}(t)^{\intercal}%
C_{4}(t)\right)  \right.  \left.  \left.  L_{1,i}(t)^{-1}P_{2,i}(t)\right]
B_{2}(t)\right\} \\
+C_{1}(t)\beta_{3,i}(t)+C_{2}(t)\beta_{1,i}(t)+C_{1}(t)\beta_{4,i}(t)+\left[
I_{m}-C_{4}(t)L_{1,i}(t)^{-1}P_{3,i}(t)\right]  \left(  C_{2}(t)^{\intercal
}\beta_{5,i}(t)+C_{1}(t)^{\intercal}\right) \\
+C_{2}(t)\beta_{4,i}(t)+\left[  C_{4}(t)L_{1,i}(t)^{-1}P_{3,i}(t)-I_{m}%
\right]  C_{2}(t)^{\intercal}L_{2,i}(t)^{-1}+B_{2}(t)\beta_{7,i}%
(t)+C_{2}(t)\beta_{2,i}(t)+\beta_{7,i}(t)^{\intercal}B_{2}(t)\\
+C_{2}(t)L_{1,i}(t)^{-1}\cdot\left\{  P_{2,i}(t)B_{2}(t)+P_{3,i}%
(t)C_{2}(t)^{\intercal}L_{2,i}(t)^{-1}\left[  -P_{1,i}(t)-\left(
P_{1,i}(t)C_{2}(t)+P_{2,i}(t)^{\intercal}C_{4}(t)\right)  \right.  \right. \\
\left.  \left.  L_{1,i}(t)^{-1}P_{2,i}(t)\right]  B_{2}(t)\right\}
+B_{2}(t)^{\intercal}\beta_{8,i}(t)+\beta_{8,i}(t)^{\intercal}B_{2}%
(t)C_{2}(t)^{\intercal}\beta_{6,i}(t)\\
+C_{4}(t)L_{1,i}(t)^{-1}\left[  P_{2,i}(t)B_{2}(t)-P_{3,i}(t)C_{2}%
(t)^{\intercal}\beta_{6,i}(t)\right]  .
\end{array}
\]
{Under the bounded assumptions in Lemma \ref{le-gamma}, one can check that
$\left\{  |{a_{11}}(t)|\right\}  _{i\geq i_{0}}$, $\left\{  |{\beta_{j,i}%
(t)}|\right\}  _{i\geq i_{0}}$, $j=1,2,...,8$, $\left\{  \mathbb{E}\int%
_{0}^{T}|\lambda_{1,i}(t)|^{2}dt\right\}  _{i\geq i_{0}}$, $\left\{
\mathbb{E}\int_{0}^{T}|\varphi_{1,i}(t)|^{2}dt\right\}  _{i\geq i_{0}}$,
$\left\{  \mathbb{E}\int_{0}^{T}|v_{1,i}(t)|^{2}dt\right\}  _{i\geq i_{0}}$,
$\left\{  \mathbb{E}\int_{0}^{T}|\varphi_{2,i}(t)|^{2}dt\right\}  _{i\geq
i_{0}}$, $\left\{  \mathbb{E}\int_{0}^{T}|v_{2,i}(t)|^{2}dt\right\}  _{i\geq
i_{0}}$ are uniformly bounded. Thus $\left\{  \mathbb{E}\int_{0}^{T}%
|M_{5,i}(t)|dt\right\}  _{i\geq i_{0}}$ is uniformly bounded.} This completes
the proof. $\blacksquare$

\end{document}